\documentclass[11pt, reqno]{amsart}

\usepackage[top=3.75cm, bottom=3cm, left=3cm, right=3cm]{geometry}
\usepackage{amsmath}
\usepackage{amsthm}
\usepackage{amsfonts}
\usepackage{amssymb}
\usepackage{mathtools}
\usepackage{MnSymbol}
\usepackage{bm}
\usepackage[colorlinks=true, citecolor=citation, urlcolor=citation, linkcolor=reference]{hyperref}
\usepackage{color}
\usepackage[all,cmtip]{xy}
\usepackage{enumitem}
\usepackage{tikz}  
\usepackage{graphicx}
\usepackage{scalerel}
\usepackage{framed}
\usepackage{tcolorbox}
\usepackage{comment}
\usepackage[mathscr]{eucal}

\usepackage{stmaryrd}
\SetSymbolFont{stmry}{bold}{U}{stmry}{m}{n}

\usepackage{tikz-cd}

\definecolor{citation}{rgb}{0,.40,.80}
\definecolor{reference}{rgb}{.80,0,.40}

	\usepackage{mathdots}

	\usepackage{bm}

	\graphicspath{ {images/} }

\numberwithin{equation}{section}
\numberwithin{figure}{section}

\newtheorem{theorem}{Theorem}[section]
\newtheorem*{theorem*}{Theorem}

\newtheorem{proposition}[theorem]{Proposition}
\newtheorem*{proposition*}{Proposition}

\newtheorem{corollary}[theorem]{Corollary}

\newtheorem{lemma}[theorem]{Lemma}
\newtheorem*{lemma*}{Lemma}

\newtheorem{lemma-definition}[theorem]{Lemma-Definition}

\theoremstyle{definition}
\newtheorem{definition}[theorem]{Definition}

\newtheorem{remark}[theorem]{Remark}
\newtheorem{construction}[theorem]{Construction}

\newtheorem{example}[theorem]{Example}

\newtheorem{convention}[theorem]{Convention}
\newtheorem{notation}[theorem]{Notation}

\newtheorem{situation}[theorem]{Situation}
\newtheorem{question}[theorem]{Question}

\newcommand{\Set}{\mathscr{S}\mathrm{et}}

\newcommand{\Ab}{\mathscr{A}\mathrm{b}}

\newcommand{\Ch}{\mathrm{Ch}}

\newcommand{\opp}{\mathrm{op}}

\newcommand{\K}{\mathrm{K}}

\DeclareMathOperator{\Mor}{Mor}

\makeatletter
\newcommand{\colim@}[2]{%
  \vtop{\m@th\ialign{##\cr
    \hfil$#1\operator@font colim$\hfil\cr
    \noalign{\nointerlineskip\kern1.5\ex@}#2\cr
    \noalign{\nointerlineskip\kern-\ex@}\cr}}%
}

\newcommand{\limit}{%
  \mathop{\mathpalette\varlim@{\leftarrowfill@\scriptscriptstyle}}\nmlimits@
}
\newcommand{\colimit}{%
  \mathop{\mathpalette\varlim@{\rightarrowfill@\scriptscriptstyle}}\nmlimits@
}
\makeatother

\newcommand{\bL}{{\mathbf L}}

\newcommand{\cO}{\mathscr{O}}
\newcommand{\cA}{\mathscr{A}}
\newcommand{\cC}{\mathscr{C}}

\newcommand{\rC}{\mathrm{C}}

\newcommand{\rF}{\mathrm{F}}

\newcommand{\rH}{\mathrm{H}}

\newcommand{\rK}{\mathrm{K}}

\newcommand{\rN}{\mathrm{N}}
\newcommand{\rR}{\mathrm{R}}

\newcommand{\rP}{\mathrm{P}}
\newcommand{\rQ}{\mathrm{Q}}
\newcommand{\rT}{\mathrm{T}}

\newcommand{\fm}{\mathfrak{m}}

\newcommand{\fS}{\mathfrak{S}}

\newcommand{\bA}{\mathbf{A}}

\newcommand{\bG}{\mathbf{G}}

\newcommand{\bZ}{\mathbf{Z}}
\newcommand{\bP}{\mathbf{P}}

\newcommand{\bR}{\mathbf{R}}

\DeclareMathOperator{\Span}{Span}

\DeclareMathOperator{\Hom}{Hom}

\DeclareMathOperator{\Spec}{Spec}

\DeclareMathOperator{\Pic}{Pic}

\DeclareMathOperator{\Br}{Br}

\newcommand{\CH}{\mathrm{CH}}

\newcommand{\sing}{\mathrm{sing}}

\def\sHom{\mathop{\mathscr{H}\!\mathit{o}\! \kern .4pt \mathit{m}}\nolimits}

\def\sInn{\mathop{\mathscr{I}\! \kern .8pt \mathit{nn}}\nolimits}

\newcommand{\pr}{\mathrm{pr}}

\newcommand{\Var}{\mathrm{Var}}

\newcommand{\id}{\mathrm{id}}

\renewcommand{\setminus}{\smallsetminus}

\renewcommand{\subset}{\subseteq}
\renewcommand{\supset}{\supseteq}

\newcommand{\nc}{\newcommand}
\nc{\p}[2]{\frac{\partial #1}{\partial #2}}
\nc{\pp}[2]{\frac{\partial^2 {#1}} {\partial {#2} ^2}}
\nc{\pmix}[3]{\frac{\partial^2 {#1}}{\partial {#2}\, \partial {#3}}}

\newcommand{\bpm}{\begin{pmatrix}}
\newcommand{\epm}{\end{pmatrix}}

\newcommand{\bbm}{\begin{bmatrix}}
\newcommand{\ebm}{\end{bmatrix}}

	\usepackage{cite}

	\DeclareMathOperator{\Sd}{Sd}

	\newcommand{\TorMod}{\mathscr{T}\mathrm{or}\mathscr{M}\mathrm{od}}

	\newcommand{\cSS}{\mathscr{SS}}
	
	\newcommand{\sd}{\mathrm{sd}}
	\newcommand{\Star}{\mathrm{Star}}
	\newcommand{\cStar}{\overline{\mathrm{Star}}}

	\newcommand{\Hilb}{\mathrm{Hilb}}

	\newcommand{\join}{\mathrm{Join}}	
	\newcommand{\Link}{\mathrm{Link}}

	\newcommand{\RE}{\mathrm{RE}}
	\newcommand{\re}{\mathrm{re}}

	\newcommand{\SB}{\mathrm{SB}}
	\newcommand{\sbir}{\mathrm{sb}}

	\newcommand{\sVol}{\mathscr{V}\mathrm{ol}}
	\newcommand{\Vol}{\mathrm{Vol}}
	\newcommand{\SmProj}{\mathscr{S}\mathrm{m}\mathscr{P}\mathrm{roj}}
	\newcommand{\SmSep}{\mathscr{S}\mathrm{mSep}}

	\newcommand{\kscat}{\mathscr{RE}}

	\newcommand{\cech}{{\mathrm{\check{C}}}}
	\newcommand{\vertex}{\mathrm{Vert}}

	\newcommand{\col}{\colon}
	\newcommand{\ra}{\rightarrow}
	\newcommand{\dra}{\dashrightarrow}

	\newcommand{\bsigma}{\bm{\sigma}}
	
	\newcommand{\btau}{\bm{\tau}}

	\newcommand{\Bl}{\mathrm{Bl}}

\begin{document}

	\title{Complexes of stable birational invariants}
	\author{James Hotchkiss}
    \address{Department of Mathematics, Columbia University, New York, NY 10027 \smallskip}
    \email{james.hotchkiss@columbia.edu}

    \author{David Stapleton}
    \email{math.david.stapleton@gmail.com}

	\begin{abstract} 
		We introduce a new stable birational invariant, which takes the form of a functor sending a degenerating variety to the homotopy type of a chain complex. 
		Our invariant is a categorification of the motivic volume of Nicaise and Shinder.

		From the class of the chain complex in a Grothendieck group, we obtain a motivic obstruction to retract rationality, valued in a quotient of the Grothendieck ring of varieties. 
		In addition, we construct a general class of stable birational invariants, with the invariant above as the universal example, given by applying any chosen stable birational invariant (e.g., unramified cohomology) to the strata of a semistable degeneration. 
	\end{abstract}

	\maketitle


	\section{Introduction} 
	\label{sec:introduction}

    \subsection{The rationality problem} 
	Let $k$ be an algebraically closed field of characteristic $0$. 
	A smooth, projective variety $X$ over $k$ is \emph{rational} if it is birational to a projective space $\bP_k^d$, 
	\emph{stably rational} if $X \times \bP_k^r$ is rational for some $r \geq 0$, 
	and \emph{retract rational} if the identity $X \to X$ factors rationally through a map $\bP_k^n \dra X$.
	The three notions are related as follows:
	\[
		\textrm{rational} \implies \textrm{stably rational} \implies \textrm{retract rational}.
	\]
	The first implication is known to be strict \cite{bctssd}, but it is an open question whether stable rationality and retract rationality differ over an algebraically closed field.
	Retract rationality in turn implies \emph{unirationality}, i.e., that $X$ is dominated by projective space, and it is a general problem to determine whether a given unirational variety belongs to any of the three classes above.

	In \cite{NicShin19}, Nicaise and Shinder introduced a remarkable tool in the study of the behavior of stable rationality under degeneration, the \emph{motivic volume}.
	If $\bar K = \bigcup_{d \geq 1} k(\!(t^{1/d})\!)$ is the field of Puiseux series over $k$,
	then the motivic volume is a specialization homomorphism between Grothendieck groups of varieties,
	\[
	 	\Vol:\K_0(\Var_{\bar K}) \to \K_0(\Var_{k}),
	 \] 
	sending a smooth projective variety $X_{\bar K}$ to a class in $\rK_0(\Var_k)$ built up from the strata of a semistable reduction over $k$.
	From the theorem of Larsen and Lunts \cite{larsen-lunts}, the reductions of $\rK_0(\Var_{\bar K})$ and $\rK_0(\Var_k)$ modulo the Lefschetz motive $\bL$ are isomorphic to $\bZ[\SB_{\bar K}]$ and $\bZ[\SB_k]$, the free abelian groups on stable birational equivalence classes over $\bar K$ and $k$, respectively.
	As a result, the motivic volume descends to a homomorphism
	\[
		\Vol_{\sbir} : \bZ[\SB_{\bar K}] \to \bZ[\SB_k],
	\]
	where $\Vol_{\sbir}([X_{\bar K}]_{\sbir})$ is given by the formula
	\begin{equation}
	\label{eq:motivic-vol-intro}
		\sum_{I} (-1)^{|I| + 1} [D_{I}]_{\sbir}, \quad D_{I} = \bigcap_{i \in I} D_i,
	\end{equation}
	for any semistable model of $X_{\bar K}$ over $\Spec k\llbracket t^{1/d} \rrbracket$ with special fiber $X_0 = D_0 \cup \dots \cup D_n$.
	There is also a birational version of the motivic volume due to Kontsevich--Tschinkel \cite{KonTsch19}, with Burnside groups in place of Grothendieck groups.
	A common generalization of the approaches of Nicaise--Shinder and Kontsevich--Tschinkel may be found in \cite{no-refinement}.

	The existence of the motivic volume implies the celebrated result from \cite{NicShin19} that stable rationality specializes in smooth projective families.
	Moreover, it provides the \emph{motivic obstruction to stable rationality}: 
	If $X_{\bar K}$ is a smooth, projective variety over $\bar K$ with
	\[
		\Vol_{\sbir}([X_{\bar K}]) \neq [\Spec k],
	\]
	then $X_{\bar K}$ is not stably rational.
	The motivic obstruction has been applied by Nicaise--Ottem \cite{NOtropical} to establish stable irrationality in a number of previously open cases, including the very general quartic fivefold.

	Our main goal is to construct a functorial version of the motivic volume,
	in which the homomorphism $\Vol_{\sbir}$ is replaced by a functor, 
	the target $\bZ[\SB_k]$ is replaced by a homotopy category of chain complexes, 
	and the expression \eqref{eq:motivic-vol-intro} appears (in a further quotient of the Grothendieck ring of varieties) as the class of a chain complex in the Grothendieck group of an additive category.

	\subsection{The categorical motivic volume}

	To describe our main result, we first require an analogue of stable birational equivalence for morphisms, compatible with composition.
	The following notion is from \cite{manin}.
	
	\begin{definition}[Manin]
	\label{def:intro-r-equiv}
	Let $X$ and $Y$ be smooth, projective varieties over $k$.
	Two rational maps $f, g \col Y \dra X$ are \emph{directly $\rR$-equivalent} if there is a rational map 
	\[
		\bP^1_k \times Y \dra X
	\]
	whose restriction to $\{0\} \times Y$ is $f$, and to $\{\infty\} \times Y$ is $g$.
	Then \emph{$\rR$-equivalence} is the equivalence relation on $X(k(Y))$ generated by direct $\rR$-equivalence, and the set of $\rR$-equivalence classes is denoted by $X(k(Y))/\rR$.
	\end{definition}

	A sequence of rational maps $Z \dra Y \dra X$ may not generally be composed, but there is a natural composition law for $\rR$-equivalence classes of rational maps.
	In fact, if $\SmProj_k$ denotes the category of smooth, projective varieties over $k$, then Kahn and Sujatha \cite{kahn-sujatha} constructed a category, which we denote by $\kscat_k$ (for ``$\mathscr{R}$-$\mathscr{E}$quivalence''), with a functor
	\[
		\{-\}_{\re}:\SmProj_k \to \kscat_k.
	\]
	The functor identifies the objects of $\kscat_k$ with the objects of $\SmProj_k$, but the morphisms in $\kscat_k$ are the $\rR$-equivalence classes of rational maps between smooth, projective varieties:
		\[
			\Mor_{\kscat_k}(\{Y\}_{\re}, \{X\}_{\re}) = X(k(Y))/\rR.
		\]
	Birational morphisms and projective bundles go to isomorphisms in $\kscat_k$, and it turns out that $\{-\}_{\re}$ is universal with respect to this property:

	\begin{remark}[The universal stable birational invariant]
		The Kahn--Sujatha category $\kscat_k$ is the universal functorial stable birational invariant, in the following sense:
		Any functor
		\[
			\rF : \SmProj_k \to \cC
		\]
		which sends birational morphisms and projective bundles to isomorphisms in a category $\cC$, factors uniquely through $\kscat_k$.
        However, it is not known if \emph{all} isomorphisms in $\kscat_k$ come from stable birational equivalences.\footnote{For example, if $X$ is retract rational, then $\{X\}_{\re}$ is isomorphic to $\{\Spec k\}_{\re}$: Since $\{\bP^n_k\}_{\re}$ is isomorphic to $\{\Spec k\}_{\re}$, the structure map $X \to \Spec k$ is a two-sided inverse to the inclusion of a point $\Spec k \to X$.}
	\end{remark}

	The category $\kscat_k$ plays the role of the set of stable birational equivalence classes $\SB_k$. In analogy with the free abelian group $\bZ[\SB_k]$, there is a simple formal procedure to extend $\kscat_k$ to an additive category $\bZ[\kscat_k]$ by including formal $\bZ$-linear combinations of morphisms, and then freely adjoining finite direct sums of objects and a zero object.
	In particular, we may speak of chain complexes with entries in $\bZ[\kscat_k]$, along with the homotopy category of bounded chain complexes, which we denote by $K^b(\bZ[\kscat_k])$.

	\begin{theorem}[Categorical motivic volume]
	\label{thm:intro-categorical-motivic-volume}
	    Let $k$ be an algebraically closed field of characteristic $0$, and let $\bar K$ be the field of Puiseux series over $k$.
        There is an additive functor
	    \[
	    	\sVol_{\re}:\bZ[\kscat_{\bar K}] \to K^b(\bZ[\kscat_k]),
	    \]
	    satisfying the following property:
		Given a semistable model $X_{\bar K}$ over $\Spec k\llbracket t^{1/d} \rrbracket$ for some $d > 0$ with special fiber $X_0 = D_0 \cup \dots \cup D_n$,
	    $\sVol_{\re}(\{X_{\bar K}\}_{\re})$ is homotopic to the \v{C}ech complex 
	   	\begin{equation}
	   	\label{eq:intro-cech-cpx-ks}
		\begin{tikzcd}[column sep=small]
			0 \ar[r] & \cdots \ar[r] &  \displaystyle{\bigoplus_{i < j < k} \{D_i \cap D_j \cap D_j\}_{\re} } \ar[r] & \displaystyle{\bigoplus_{i < j} \{D_i \cap D_j\}_{\re} } \ar[r] & \displaystyle{\bigoplus_{i} \{D_i\}_{\re} } \ar[r] & 0 .
		\end{tikzcd}
    \end{equation}
    The differential of \eqref{eq:intro-cech-cpx-ks} is the usual one in \v{C}ech theory, given by an alternating sum of the $\rR$-equivalence classes of the inclusions, e.g., $D_i \cap D_j \subset D_i, D_j$.
	Moreover, it may happen that an intersection in \eqref{eq:intro-cech-cpx-ks}, e.g., $D_i \cap D_j$, is disconnected; if so, $\{D_i \cap D_j\}_{\re}$ refers to the formal direct sum of its connected components. 
	\end{theorem}

    Because the invariant is a functor, it may be used to obstruct retract rationality, as we discuss below.
	The proof of Theorem~\ref{thm:intro-categorical-motivic-volume} is based, much like the arguments in \cite{NicShin19}, on the weak factorization theorem.
	However, the \v{C}ech complex \eqref{eq:intro-cech-cpx-ks} is not functorial in a na\"ive way (e.g., with respect to morphisms of snc pairs),
    so in the course of the proof, we replace the \v{C}ech complex with a more complicated (but functorial) complex that has the same homotopy type.

	\subsection{A motivic obstruction to retract rationality}
	For lack of better terminology, we say two varieties $X, Y$ are \emph{$\rR$-equivalent} if they are isomorphic in $\kscat_k$, or equivalently if there exist $f \in X(k(Y))/\rR$ and $g \in Y(k(X))/\rR$ such that $f \circ g$ and $g \circ f$ are $\rR$-equivalent to the identities.
	For example, a pair of varieties which are stably birational to each other are also $\rR$-equivalent, and a retract rational variety is $\rR$-equivalent to $\Spec k$.
	If $\bZ[\RE_k]$ denotes the free abelian group on the set of $\rR$-equivalence classes of varieties over $k$, 
	then there is a surjective homomorphism
\begin{equation*}
    \bZ[\SB_k] \to \bZ[\RE_k],
\end{equation*}
	sending $[X]_{\sbir}$ to the $\rR$-equivalence class $[X]_{\re}$.
	It turns out that $\bZ[\RE_{\bar K}]$ and $\bZ[\RE_k]$ are the Grothendieck groups of the domain and target of the functor $\sVol_{\re}$, so Theorem~\ref{thm:intro-categorical-motivic-volume} directly implies the following:

	\begin{corollary}
	\label{cor:intro-re-motivic-volume}
        Let $k$ be an algebraically closed field of characteristic $0$, and let $\bar K$ be the field of Puiseux series over $k$.
	    There is a group homomorphism
	    \[
	    	\Vol_{\re}:\bZ[\RE_{\bar K}] \to \bZ[\RE_k],
	    \]
	    given by sending $[X_{\bar K}]_{\re}$ to the image of the motivic volume $\Vol_{\sbir}([X_{\bar K}]_{\sbir}) \in \bZ[\SB_k]$ under the natural surjection $\bZ[\SB_k] \to \bZ[\RE_k]$.
	\end{corollary} 

	A variety $X$ over $k$ is \emph{universally $\rR$-trivial} if it is $\rR$-equivalence to a point, or equivalently, if $X(k(Y))/\rR$ is a singleton for each variety $Y$ over $k$.
	In fact, it was shown in \cite{AsokMorel} that universal $\rR$-triviality is equivalent to \emph{$\bA^1$-connectedness} in the sense of $\bA^1$-homotopy theory, and there is a substantial literature on the topic under that name.
	Universal $\rR$-triviality lies downstream of the notions of rationality discussed above:
	\[
		\textrm{retract rational} \implies \textrm{ universally $\rR$-trivial} \implies \begin{cases}
			\textrm{ $\bZ$-decomposition of the diagonal}, \\
			\textrm{ $+$ rationally connected}.
		\end{cases} 
	\]
	It is not known if either implication is strict, as there is no known example of a rationally connected variety with an integral decomposition of the diagonal which is not stably rational, cf. \cite[\S 11]{schreieder-nr}.
	As indicated, universally $\rR$-trivial varieties are rationally connected, but it does not seem to be known if they are unirational. 

	\begin{corollary}[The motivic obstruction to retract rationality]
	\label{cor:motivic-obstruction-to-rr}
	    If $\Vol_{\re}([X_{\bar K}]_{\re}) \neq [\Spec k]_{\re}$, then $X_{\bar K}$ is not universally $\rR$-trivial, and therefore is not retract rational.
	\end{corollary}

    It seems to us to be necessary to go through Theorem~\ref{thm:intro-categorical-motivic-volume} in order to prove Corollary~\ref{cor:intro-re-motivic-volume} and Corollary~\ref{cor:motivic-obstruction-to-rr}.
    Certainly, there is a homomorphism
    \begin{equation}
    \label{eq:composition-mv}
        \bZ[\SB_{\bar K}] \to \bZ[\RE_k],
    \end{equation}
    by composing the Nicaise--Shinder motivic volume with the natural surjection from $\bZ[\SB_k]$ to $\bZ[\RE_k]$.
    If one wishes to show directly, however, that \eqref{eq:composition-mv} factors through $\bZ[\RE_{\bar K}]$,
    then one faces the difficult problem of describing the kernel of $\bZ[\SB_{\bar K}] \to \bZ[\RE_{\bar K}]$.
    
	Of course, Corollary~\ref{cor:motivic-obstruction-to-rr} would be useless if $\Vol_{\re}([X_{\bar K}]_{\re})$ is often equal to $[\Spec k]_{\re}$.
	As we explain below, the methods currently used to show $\Vol_{\sbir}([X_{\bar K}]_{\sbir}) \neq [\Spec k]_{\sbir}$ also imply, in practice, that $\Vol_{\re}([X_{\bar K}]_{\re}) \neq [\Spec k]_{\re}$.

	\begin{remark}[Specialization of universal $\rR$-triviality]
		The existence of $\Vol_{\re}$ in Corollary~\ref{cor:intro-re-motivic-volume} implies that universal $\rR$-triviality specializes in smooth projective families.
		However, there is a simple, direct proof (essentially identical to the argument that a decomposition of the diagonal specializes) via the specialization map for $\rR$-equivalence classes of zero cycles from \cite{kollar-spec}, see \S\ref{ssec:elementary-proof}.
	\end{remark}

	\subsection{Stable rationality and retract rationality}
	In contrast to many of the standard obstructions to stable rationality (e.g., decomposition of the diagonal), Nicaise and Shinder's motivic obstruction does not necessarily vanish for retract rational varieties.
	This has led to the following question, raised in \cite{PSQ5fold}, of whether the motivic volume could be used to construct a retract rational, stably irrational variety:

    \begin{question}
    \label{question:sr-vs-rr}
        Can one find a smooth, projective retract rational variety $X_{\bar K}$ such that 
        \begin{equation}
            \label{eq:ques-motivic-obstruction}
            \Vol_{\sbir}([X_{\bar K}]_{\sbir}) \neq [\Spec k]_{\sbir} ?
        \end{equation}
        If so, then $X_{\bar K}$ is retract rational but stably irrational. 
    \end{question}

    Our results show that Question~\ref{question:sr-vs-rr} is not a practical strategy for constructing retract rational, stably irrational varieties.
    Given a retract rational variety $X_{\bar K}$,
    Corollary~\ref{cor:motivic-obstruction-to-rr} implies that
    \begin{equation*}
        \Vol_{\re}([X_{\bar K}]_{\re}) = [\Spec k]_{\re},
    \end{equation*}
    so the problem of verifying \eqref{eq:ques-motivic-obstruction} is at least as difficult as the problem of showing directly that $\bZ[\SB_k] \to \bZ[\RE_k]$ is not an isomorphism.
    
    On the other hand, it is also possible to use the motivic obstruction to retract rationality (Corollary~\ref{cor:motivic-obstruction-to-rr}) to prove that very general members of certain classes of varieties are retract irrational. 
    As an illustration, we give a quick proof of the retract irrationality of very general quartic fivefolds in characteristic $0$.
    The analogous result for stable irrationality is from \cite{NOtropical}, while retract irrationality is proved \cite{PSQ5fold} through an argument which is more complicated than what follows but works in positive characteristic.

	\begin{example}[Quartic fivefolds]
	\label{ex:intro-quartic-fivefolds}
		The motivic volume of a certain degeneration of a quartic fivefold is computed in \cite[\S 5]{NOtropical}; by Corollary~\ref{cor:intro-re-motivic-volume}, the same formula holds in $\bZ[\RE_k]$:
		\begin{equation}
			\label{eq:volsbir-quartic}
			\Vol_{\re}([X_{\bar K}]_{\re}) = 2[Z]_{\re} - [Q]_{\re}.
		\end{equation}
		Here, $Q$ is a very general quartic double fourfold.
		By \cite{HPTQDF}, $Q$ does not have a decomposition of the diagonal, so $Q$ is not universally $R$-trivial. It follows that $\Vol_{\re}(X_{\bar K}) \neq [\Spec k]_{\re}$.
		By Corollary~\ref{cor:motivic-obstruction-to-rr}, $X_{\bar K}$ is not universally $\rR$-trivial, hence is not retract rational.
	\end{example}

    \begin{remark}
        In a sequel \cite{hotchkiss}, the first named author verifies retract irrationality in a number of situations from \cite{NOtropical} and \cite{moe}, where stable irrationality is proven. 
        The methods are in the spirit of Example~\ref{ex:intro-quartic-fivefolds}.
        In particular, it is shown that Moe's bounds on the stable irrationality of hypersurfaces and double covers (which rely upon and extend earlier breakthrough bounds of Schreieder \cite{SchreiederHyp}) may be upgraded to retract irrationality.
        
        We emphasize, however, that for most of the  results from \cite{NOtropical} and some of the results from \cite{moe}, retract irrationality has been proven later by completely different methods (which are more complicated than ours, but extend to positive characteristic). 
        We refer to \cite{PSQ5fold}, \cite{skauli}, \cite{33fivefold} for proofs of retract irrationality in cases from \cite{NOtropical}, and \cite{schreieder-lange} for a proof of retract irrationality in the hypersurface case of \cite{moe}.
    \end{remark}

	\subsection{A more general class of invariants}

	As indicated above, the functor from $\SmProj_k$ to $\kscat_k$ is the universal functorial stable birational invariant; it is also possible to consider an arbitrary stable birational invariant
	\[
		\rF : \SmProj_k \to \cA,
	\]
	where $\cA$ is an additive category.
	Examples of these include the Picard group $\Pic^0(-)$, the Brauer group $\Br(-)$, unramified cohomology $\rH^i_{nr}(k(-), \bm{\mu}_m)$, holomorphic $p$-forms $\rH^0(-, \Omega^p)$, and Chow groups of $0$-cycles $\CH_0(-)$.
	Some of these are covariant, while others are contravariant; in what follows, we suppose that $\rF$ is covariant, since the contravariant case is handled by replacing $\cA$ with $\cA^\opp$.

	Since any stable birational invariant $\rF$ as above factors through $\bZ[\kscat_k]$, Theorem~\ref{thm:intro-categorical-motivic-volume} is equivalent to the following (a priori stronger) statement:

	\begin{theorem}
	\label{thm:intro-arbitrary-additive-category}
        Let $k$ be an algebraically closed field of characteristic $0$, and let $\bar K$ be the field of Puiseux series over $k$.
	    If $\rF:\SmProj_k \to \cA$ is a functorial stable birational invariant valued in additive category $\cA$, then there is a functor
	    \[
	    	\sVol_{\rF}:\bZ[\kscat_{\bar K}] \to K^b(\cA),
	    \]
	    satisfying the following property:
		Given a semistable model $X_{\bar K}$ over $\Spec k\llbracket t^{1/d} \rrbracket$ for some $d > 0$ with special fiber $X_0 = D_0 \cup \dots \cup D_n$,
	    $\sVol_{\rF}(X_{\bar K})$ is homotopic to the \v{C}ech complex
	    \begin{equation}
		\label{eq:intro-cech-cpx}
		\begin{tikzcd}[column sep=small]
			0 \ar[r] & \cdots \ar[r] &  \displaystyle{\bigoplus_{i < j < k} \rF(D_i \cap D_j \cap D_j) } \ar[r] & \displaystyle{\bigoplus_{i < j} \rF(D_i \cap D_j) } \ar[r] & \displaystyle{\bigoplus_{i} \rF(D_i) } \ar[r] & 0 .
		\end{tikzcd}
	\end{equation}
	\end{theorem}

    Whenever the target category $\cA$ is an abelian category, the homology groups of $\sVol_{\rF}(X_{\bar K})$ provide new examples of stable birational invariants:

	\begin{corollary}	
	\label{cor:cohomology-groups}
	    If $\rF:\SmProj_k \to \cA$ is a functorial stable birational invariant valued in an abelian category $\cA$, then for each $i \in \bZ$, there are functors
	    \[
	    	\rH_i(\sVol_{\rF}(-)):\bZ[\kscat_{\bar K}] \to \cA
	    \]
	    taking $X_{\bar K}$ to the $i$th homology group of $\sVol_{\rF}(X_{\bar K})$.
	    In particular, if $\rH_i(\sVol_{\rF}(X_{\bar K})) \neq 0$ for $i > 0$, then $X_{\bar K}$ is not universally $\rR$-connected. 
	\end{corollary}

	\begin{example}[The dual complex]
		A simple example of a stable birational invariant is the constant functor which sends each smooth, projective variety over $k$ to the abelian group $\bZ$, and each morphism to the identity.
		In this case, Theorem~\ref{thm:intro-categorical-motivic-volume} outputs the chain complex computing the simplicial homology of the dual complex of a semistable model, and homotopy invariance  is well-known.
		Indeed, the homotopy type of the dual complex itself does not depend on the choice of semistable model, cf. \cite{danilov}, \cite{ks}, \cite{stepanov}, \cite{thuillier}, \cite{arapura-et-al}, \cite{payne}.
		The dual complex itself is not useful for applications to the rationality problem; it is contractible if $X_{\bar K}$ is rationally connected, cf. \cite{DFKX}.
	\end{example}

    \begin{example}[Unramified cohomology]
        If $\rF = \rH^i_{nr}(k(-), \bm{\mu}_m)$, then the cochain complex from Theorem~\ref{thm:intro-arbitrary-additive-category} is closely related to the specialization method pioneered by Voisin \cite{VoisinQDS} and Colliot-Th\'el\`ene--Pirutka \cite{ct-pirutka}.
        For example, if $X_{\bar K}$ has a degeneration $X_0$ with isolated nodes, then it is not difficult to check that when a resolution $\tilde X_0$ has nontrivial unramified cohomology in degree $i > 0$, then
        \begin{equation*}
            \rH^0(\sVol_{\rF}(X_{\bar K})) \neq 0.
        \end{equation*}
        But Corollary~\ref{cor:cohomology-groups} yields only that $X_{\bar K}$ is not universally $\rR$-trivial, whereas the usual specialization method gives the stronger conclusion that $X_{\bar K}$ does not have a decomposition of the diagonal.
        We do not know if, for this particular choice of $\rF$ but arbitrary $X_{\bar K}$, the homotopy type of $\sVol_{\rF}(X_{\bar K})$ obstructs the existence of a decomposition of diagonal on $X_{\bar K}$.
    \end{example}

    \begin{remark}[Open varieties]
        There is also a variant of Theorem~\ref{thm:intro-arbitrary-additive-category} for smooth, separated varieties over $k$, which takes the form of a functor
        \begin{equation}
            \sVol_{\rF} : \SmSep_k^{\mathrm{prop}} \to K^b(\cA),
        \end{equation}
        where the domain is smooth, separated varieties with proper morphisms, such that:
        Given a smooth separated variety $U$ over $k$ and a smooth compactification $U \subset X$ with snc boundary $D = X - U$, $\sVol_{\rF}(U)$ is homotopic to the \v{C}ech complex obtained by applying $\rF$ to the strata of $D$.
        The resulting chain complex is reminiscent of the weight complexes from \cite{gillet-soule}.
        We do not treat this variant explicitly in the body of the paper, but the proof is almost identical to that of Theorem~\ref{thm:existence-of-functor}.
    \end{remark}


	\subsection{Organization of the paper}
    In \S \ref{sec:toroidal-embeddings}, we review the theory of toroidal embeddings and cone complexes used throughout the paper. 
    In \S \ref{sec:stable_birational_invariants}, we discuss functorial stable birational invariants, and describe how to apply them to singular varieties and the strata of toroidal embeddings. 
    In \S \ref{sec:chain-complexes}, we define the chain complexes involved in the statement and proof of Theorem~\ref{thm:intro-arbitrary-additive-category}.
    In \S \ref{sec:star-subdivision}, we show that the homotopy type of the chain complex does not change under a star subdivision of a toroidal embedding, and in \S \ref{sec:blowing-up}, we prove the same for the blowup of an snc pair $(X, D)$ along a subvariety having simple normal crossings with $D$.
    In \S \ref{sec:the-categorical-motivic-volume}, we prove Theorem~\ref{thm:intro-categorical-motivic-volume}, Corollary~\ref{cor:intro-re-motivic-volume}, and Theorem~\ref{thm:intro-arbitrary-additive-category}.
    Finally, in \S \ref{sec:spec-and-var}, we record two basic results about the variation of $\rR$-equivalence classes in families.

	\subsection{Notation and conventions}
    Throughout, $k$ is an algebraically closed field of characteristic $0$, and $R = k \llbracket t \rrbracket$. 
    A \emph{toroidal embedding} over $R$ or $k$ always refers to a toroidal embedding without self-intersection.
		
	\subsection{Acknowledgements}
	We thank Aravind Asok, Dori Bejleri, Nathan Chen, Kristin DeVleming, Johan de Jong, J\'anos Koll\'ar, Daniel Litt, John Ottem,  Alex Perry, Stefan Schreieder, Evgeny Shinder, and Burt Totaro for helpful discussions related to this work. 
    We are particularly grateful to Stefan Schreieder for his comments on a draft and a valuable correspondence.

	During the preparation of this paper, the first author was partially supported by NSF grant DMS-2401818.
	Part of this work was also completed while the first author was in residence at the Simons Laufer Mathematical Sciences Institute in Spring 2024, under the support of NSF grant DMS-1928930 and Sloan Foundation grant G-2021-16778.

\section{Toroidal embeddings}
\label{sec:toroidal-embeddings}

While the results described in the introduction concern only semistable models, we develop most of our results in the greater generality of \emph{toroidal embeddings}.
Let us briefly explain the reason. Looking ahead to the proofs of Theorem~\ref{thm:intro-categorical-motivic-volume} or Theorem~\ref{thm:intro-arbitrary-additive-category}, we must compare chain complexes associated to, say, two semistable models of a given $X_{\bar K}$ over $k \llbracket t^{1/d} \rrbracket$ and $k \llbracket t^{1/e} \rrbracket$. 
If $d = e$, the two models are birational, and the comparison is made by resolving indeterminacy and using weak factorization.
When $d \neq e$, then it would be convenient to perform a base change to a common base $k \llbracket t^{1/d'} \rrbracket$ and run the same argument. 
However, the base change introduces toroidal singularities, and to deal with them, we allow such singularities in our constructions from the beginning.

In this section, we briefly review the language of toroidal embeddings used throughout the paper. 
The classic reference is \cite{toroidal-embeddings}, while a detailed overview may be found in \cite{functorial-res}.
A different approach would have been to use the language of smooth log schemes with Zariski-local charts from \cite{kato}, 
as in \cite{NicShin19}. 
The only feature of either theory which we really need is the language of cone complexes, and toroidal embeddings provide a more direct route.

\subsection{Toroidal embeddings}

Throughout, let $k$ be an algebraically closed field of characteristic $0$, and let $R = k\llbracket t \rrbracket$.
We write $0, \eta \in \Spec R$ for the closed and generic points, respectively.
Let $\rT = \bG^n_{m, R}$ be a split torus over $R$. 
A \emph{torus embedding} over $R$ is an open immersion $\rT_{\eta} \subset X$, where $X$ is irreducible, normal, and of finite type over $R$, such that the action of $\rT_{\eta}$ on itself by translation extends to an action of $\rT$ on $X$.

\begin{example}[Affine torus embeddings]
\label{ex:affine-torus-emb}
    Let $M$ and $N$ be the character and cocharacter groups of $\rT$. There is a correspondence between affine torus embeddings $\rT_{\eta} \subset X_{\sigma}$ and maximum-dimensional rational polyhedral cones $\sigma \subset N_{\bR} \times \bR_{\geq 0}$ not contained in $\rN_{\bR} \times \{0\}$.
    The correspondence is given by setting 
    \[
        X_{\sigma} = \Spec R[\dots, x^m t^k, \dots], \quad (m, k) \in (M \times \bZ) \cap \sigma^\vee ,
    \]
    where $\sigma^\vee$ is the dual cone.
    The orbits of $\rT_{\eta}$ contained in $(X_{\sigma})_0$ correspond to the faces of $\sigma$ contained in $N_{\bR} \times \{0\}$, while the orbits of $\rT_0$ in $(X_{\sigma})_0$ correspond to the faces of $\sigma$ not contained in $N_{\bR} \times \{0\}$.
\end{example}

A \emph{toroidal embedding without self-intersection} over $R$ (or, for simplicity, just \emph{toroidal embedding}, cf. Remark~\ref{rem:self-intersection}) consists
of an open immersion $U \subset X$, where $X$ is an irreducible, normal scheme of finite type over $R$, and $U \subset X_{\eta}$ is an open subset, 
such that the following condition holds: 
For each closed point $x \in X$, there is an open neighborhood $U_x$ containing $x$, a torus embedding $\rT_{\eta} \subset Z$ over $R$, and an \'etale morphism 
\[
    \pi:U_x \to Z
\]
such that $U \cap U_x = \pi^{-1}(\rT_{\eta})$. 
The data $(U_x, \pi\col U_x \to Z)$ is called a \emph{toric chart} at $x$.

A \emph{toroidal embedding without self-intersection} over $k$ is defined similarly, but where a torus embedding over $R$ is replaced by a torus embedding over $k$ (i.e., an embedding of $\bG_{m, k}^r$ into a toric variety over $k$).

\begin{convention}
    We typically denote a toroidal embedding $U \subset X$ as a pair $(X, D)$, where $D \subset X$ is the complement of $U$ with its reduced structure.
\end{convention}

\begin{example}[Simple normal crossings]
    An \emph{snc pair} over $R$ or $k$ is a toroidal embedding $(X, D)$ over $R$ or $k$ such that $X$ is regular.
    In that case, $D$ is a simple normal crossings subscheme.
    Indeed, to prove this, one may reduce to the case of an affine torus embedding, where (in the notation of Example~\ref{ex:affine-torus-emb}) the claim is that $X_{\sigma}$ is regular if and only if the primitive integral vectors in the rays of the $\sigma$ generate the lattice $M \times \bZ$.
\end{example}

\begin{remark}[Self-intersection]
\label{rem:self-intersection}
If $U_x$ above is permitted to be an \'etale neighborhood, as opposed to a Zariski neighborhood, then $U \subset X$ could have self-intersection.
We shall never consider toroidal embeddings with self-intersection; by convention, a \emph{toroidal embedding} refers to a toroidal embedding without self-intersection.
\end{remark}

If $(X, D)$ is a toroidal embedding over $R$ or $k$, then $D$ has pure codimension $1$. A \emph{stratum} (which, by convention, refers to a \emph{closed stratum}) is a connected component of 
\[
D_I := \bigcap_{i \in I} D_i,
\]
where $D = D_0 + \cdots + D_n$, and $I \subset \{0, \dots, n\}$.
We shall occassionally refer to an \emph{open stratum}, which is the open subset of a stratum obtained by removing all proper substrata.
Note that the open strata are regular, or in other words, the singular locus of a stratum is contained in proper substrata.
The embedding of any open stratum into its closure is a toroidal embedding, either over $R$ or $k$.

Let $\rP(X, D)$ be the poset of strata, with the partial ordering \emph{opposite} to the partial ordering by inclusion. 
We regard $\rP(X, D)$ as the set which indexes the strata, 
so that for $\sigma \in \rP(X, D)$, the corresponding stratum is denoted $D_{\sigma}$, 
and the partial ordering is given by $\sigma \leq \tau$ whenever $D_{\sigma} \supset D_{\tau}$. 
Furthermore, given a stratum $D_{\sigma}$, we write $D_{\sigma}^\circ$ for the open stratum it contains.

\begin{construction}[The cone associated to a stratum]
\label{constr:cone-of-stratum}
    Let $(X, D)$ be a toroidal embedding over $R$ or $k$, and fix a stratum $\sigma \in \rP(X, D)$. 
    Associated to $\sigma$ is a strongly convex, rational polyhedral cone, which we denote by $|\sigma|$, defined as follows:
    \begin{itemize}
        \item Let $\Star(D_{\sigma}) \subset X$ be the union of $X \setminus D$ with all of the open strata $D_{\tau}^\circ$ whose closures contain $D_{\sigma}$. 
        We note that $\Star(D_{\sigma})$ is an open subscheme of $X$.
        \item Let $M_{\sigma}$ be the lattice of Cartier divisors on $\Star(D_{\sigma})$ set-theoretically supported on $D \cap \Star(D_{\sigma})$, and let $M_{\sigma}^+ \subset M_{\sigma}$ to be the submonoid of effective Cartier divisors.
        \item Let $N_{\sigma} = \Hom(M_{\sigma}, \bZ)$ to be the dual lattice.
    \end{itemize}
    Finally, we define
    \[
        |\sigma| = \{x \in N_{\sigma} \otimes {\bR} \mid \langle m, x \rangle \geq 0,  \forall m \in M_{\sigma}^+ \}.
    \]
    The dimension of the cone $|\sigma|$ is the codimension of $D_{\sigma}$ in $X$.
\end{construction}

As $\sigma$ ranges over $\rP(X, D)$, the cones $|\sigma|$ glue to form a \emph{cone complex} $\Sigma = \Sigma(X, D)$, which is a topological space $|\Sigma|$ presented as the union of a poset of strongly convex, rational polyhedral cones, glued along inclusions of faces.
The cone complex $\Sigma$ of the toroidal embedding $(X, D)$ plays the role of the fan of a toric variety; the main difference is that a toric fan comes with an embedding into an ambient Euclidean space, whereas a cone complex does not.

By construction, the cones in $|\Sigma|$ are in bijection with the elements of $\rP(X, D)$, and the partial ordering on $\rP(X, D)$ is corresponds to the ordering of the cones of $\Sigma$ by inclusion.
Given that $\Sigma(X, D)$ determines $\rP(X, D)$, it may seem excessive to introduce separate notation for $\rP(X, D)$.
However, $\rP(X, D)$ is functorial in a more general situation than $\Sigma(X, D)$.

\subsection{Categories of toroidal embeddings}

For $d > 0$, let $R_d = k \llbracket t^{1/d} \rrbracket$.
Of course, $R_d$ is abstractly isomorphic to $R$, and our discussion of toroidal embeddings over $R$ carries over to toroidal embeddings over $R_d$. 

Given a smooth, projective variety $X_{\bar K}$ over the field of Puiseux series $\bar K$, as in the introduction, the semistable reduction theorem from \cite{toroidal-embeddings} guarantees the existence of a semistable model over $R_d$ for some $d > 0$ with geometric generic fiber $X_{\bar K}$.
We ultimately need a category which allows us to compare different semistable models over $R_d$ and $R_{d'}$, say. 

\begin{definition}[Models]
\label{def:models}
    A \emph{toroidal model} over $R_d$ is a toroidal embedding $(X, D)$ over $R_d$ such that $X$ is proper over $R_d$ and $D$ coincides with the reduced subscheme underlying the special fiber $X_0$. 
    A toroidal model $(X, D)$ is an \emph{snc model} if $X$ is regular; in this case, $D$ is a simple normal crossings subscheme. 
    Finally, a toroidal model $(X, D)$ is \emph{semistable} if it is an snc model and the special fiber $X_0$ is reduced (so that $X_0 = D$).
\end{definition}

\begin{definition}[The category of models]
\label{def:cat-of-models}
     We write $\TorMod_\infty$ for the category whose objects are  toroidal models over $R_d$ for each $d > 0$, and whose morphisms are commutative diagrams
    \[
        \begin{tikzcd}[column sep=large]
            (Y, E) \ar[d] \ar[r, "f"] & (X, D) \ar[d] \\
            \Spec R_{de} \ar[r, "t^{1/de} \mapsto t^{1/d}"] & \Spec R_d,
        \end{tikzcd}
    \]
    where $f$ is a proper morphism.
\end{definition}

\subsection{Functoriality of strata}

In this section, we show that $\rP(X, D)$ is functorial in $(X, D)$.
To discuss functoriality, we first need to address what is meant by a morphism between toroidal embeddings. 
There are several possible notions, but the most convenient is the following:

\begin{definition}
    A \emph{morphism} $f\col (Y, E) \to (X, D)$ between toroidal embeddings over $R$ (resp., $k$) is a morphism between the underlying $R$-schemes (resp., $k$-schemes) $f\col X \to Y$ such that $f(E) \subset D$ and $f(Y - E) \subset X - D$.
\end{definition}

\begin{definition}
\label{def:span}
    \begin{enumerate}
        \item Let $(X, D)$ be a toroidal embedding over $R$ or $k$. 
    Given a closed subset $Z \subset X$ contained in $D$, $\Span(Z)$ is defined to be the smallest stratum of $(X, D)$ containing $Z$.
        \item Suppose that $f\col (X, D) \to (Y, E)$ be a proper morphism between toroidal embeddings over $k$, or a morphism in $\TorMod_{\infty}$. 
    Then we define a map of posets 
    \[
        f_*:\rP(X, D) \to \rP(Y, E)
    \]
    by sending $\sigma$ to $\Span(f(D_{\sigma}))$.
    \end{enumerate}
\end{definition}

\begin{lemma}
\label{lemma:functoriality_of_strata}
    Given a pair of composable, proper morphisms $f, g$ between toroidal embeddings over $R$ or $k$, or a morphism in $\TorMod_{\infty}$, we have $f_* \circ g_* = (f \circ g)_*$.
\end{lemma}

\begin{proof} 
    The proof is identical in any of the above cases.
    It suffices to prove the following stronger claim:
    If $f\col (Y, E) \to (X, D)$ is a morphism, and $Z \subset E$ is an irreducible closed subset, then 
    \begin{equation*}
        \Span(f(\Span(Z))) = \Span(f(Z)).
    \end{equation*}
    The inclusion $\supset$ is clear. 
    To conclude, it is enough to show that
    \begin{equation}
        \label{eq:preimage-of-span}
        \Span(Z) \subset f^{-1}(\Span(f(Z))).
    \end{equation}

    To prove it, we are allowed to restrict over the open stratum inside $\Span(f(Z))$. In particular, we may assume that $X$ is regular, so that the irreducible components of $D$ are effective Cartier divisors. 
    Write $\Span(f(Z))$ as the intersection of $D_1, \dots, D_k$.
    The intersection 
    \begin{equation*}
        f^{-1}(D_1 \cap \dots \cap D_k) = f^{-1}(D_1) \cap \dots \cap f^{-1}(D_k)
    \end{equation*}
    is nonempty, since it contains $Z$.
    Each $f^{-1}(D_j)$ is a sum of irreducible components of $E \subset Y$, at least one of which (say, $E_j$) contains $Z$. 
    The inclusion
    \[
        Z \subset E_0 \cap \dots \cap E_k ,
    \]
    induces an inclusion $\Span(Z) \subset \bigcap E_i$. 
    The right-hand side is contained in $f^{-1}(\Span(Z))$ by construction, which proves \eqref{eq:preimage-of-span}.
\end{proof}

\subsection{Toroidal modifications}
In this section, we consider a far stronger notion of morphism between toroidal embeddings.
A \emph{subdivision} of cone complexes $\Sigma' \to \Sigma$ is a homeomorphism of the underlying topological spaces $|\Sigma'| \to |\Sigma|$ such that any cone in $\Sigma$ is a finite union of cones in $\Sigma'$. 
For example, if $\Sigma$ is the fan associated to a torus embedding $\rT \subset X$ over $R$ or $k$, then any subdivision determines a proper, birational, equivariant map $X' \to X$ of torus embeddings.

If $(X, D)$ is a toroidal embedding over $R$ or $k$, then a \emph{toroidal modification} is a proper, birational morphism of toroidal embeddings $f\col(X', D') \to (X, D)$ such that, locally on $X$, $f$ is the pullback along a toric chart of a subdivision of torus embeddings over $R$ or $k$, respectively.
If $(X, D)$ is a given toroidal embedding, there is a bijection between toroidal modifications $(X', D') \to (X, D)$ and subdivisions of the cone complex $\Sigma$ associated to $(X, D)$.

\begin{lemma}
\label{lem:fibers-of-toroidal-resolution}
    Let $f\col (X', D') \to (X, D)$ be a toroidal modification over $R$ or $k$.
    \begin{enumerate} 
        \item \label{item:proj-bundle-stratum} 
        Given $\sigma' \in \rP(X', D')$, the image of $D'_{\sigma'}$ in $(X', D')$ is a stratum $D_\sigma$ in $(X, D)$, and the induced map $D'_{\sigma'} \to D_{\sigma}$
        is birationally a projective bundle (i.e., the induced extension of function fields is purely transcendental).
        \item \label{item:preimage-stratum}
        Given $\sigma \in \rP(X, D)$, any irreducible component $Z$ of $f^{-1}(D_{\sigma})$ is a stratum. 
    \end{enumerate}
\end{lemma}

\begin{proof}
    The statements are local on the target, and they ultimately reduce to case of proper, birational equivariant morphisms between torus embeddings.
    In that case, the statements are standard, but as they are used extensively in what follows, we recall them for the reader's convenience.
    
    If $\varphi:\Sigma' \to \Sigma$ is a morphism of fans, then the corresponding morphism of torus embeddings, say $f:X' \to X$, is given as follows: 
    By definition, for each $\sigma' \in \Sigma'$, $\varphi(\sigma')$ is contained in a unique, minimal cone $\sigma \in \Sigma$.
    Then $f$ is given by gluing the affine toric morphisms
    \begin{equation}
    \label{eq:induced-map-orbits}
        \rT'_{\sigma'} \to \rT_{\sigma},
    \end{equation}
    where $\rT'_{\sigma'}$ and $\rT_{\sigma}$ are the orbits corresponding to $\sigma'$ and $\sigma$, respectively.
    It follows directly from the construction (or, alternatively, from the equivariance of $f$) that the image of an orbit is an orbit (proving the first part of Item~\ref{item:proj-bundle-stratum}), and that the preimage of an orbit is a finite union of orbits (proving Item~\ref{item:preimage-stratum}).

    If $\varphi$ is moreover a subdivision, and $\varphi(\sigma') = \sigma$ as above, then the induced map \eqref{eq:induced-map-orbits} is a torsor under a subgroup of a torus. It is a $\bG_m^r$-torsor (either over $k$ or the fraction field of $R$, depending on the situation) for some $r \geq 0$ if and only if the induced map of character lattices
    \[
        N_{\sigma'} \to N_{\sigma}   
    \]
    has torsion-free cokernel. 
    But this is immediate from the definition of a subdivision of a fan: Both $\Sigma'$ and $\Sigma$ are fans in $N_{\bR}$, for the same lattice $N$.
\end{proof}

\subsection{Star subdivision}
In this section, we describe the most important example of a toroidal modification, which is a star subdivision.
Let $\Sigma$ be a cone complex, and let $\sigma \in \Sigma$.
    \begin{enumerate} 
        \item The \emph{star} of $\sigma$ is the set $\Star(\sigma) = \{\tau \mid \sigma \leq \tau\} \subset \rP(X, D)$.
        \item The \emph{closed star} of $\sigma$ is set $\cStar(\sigma)$ of all $\tau$ such that $\tau \leq \tau'$, for some $\tau' \in \Star(\sigma)$. 
        \item The \emph{link} of $\sigma$ is the set $\Link(\sigma) = \cStar(\sigma) \setminus \Star(\sigma)$.
     \end{enumerate} 

Fix a primitive, integral vector $v$ in the interior of $|\sigma|$.
For each $\tau \in \Link(\sigma)$, and each maximal cone $\tau'$ in $\Star(\sigma)$ containing $\tau$ as a face,
we write
\[
    \join_{\tau'}(v, \tau) 
\]
for the cone $v + \tau$ in $N_{\tau'} \otimes \bR$. 
Then the star subdivision $\Sigma' \to \Sigma$ is given by setting
\[
    \Sigma' = (\Sigma \setminus \Star(\sigma)) \cup \{\join_{\tau'}(v, \tau) \} \cup \{v\},
\]
as $\tau$ runs over $\Star(\sigma)$, and for fixed $\tau$, $\tau'$ runs over the maximal cones in $\Star(\sigma)$ containing $\tau$ as a face.

If $(X, D)$ is a toroidal embedding over $R$ or $k$ with cone complex $\Sigma$, then the star subdivision $\Sigma' \to \Sigma$ corresponds to a toroidal modification
\[
   f\col (X', D') \to (X, D),
\]
which is an isomorphism away from $B_{\sigma}$.
We often abuse terminology and refer to $f$ itself as a star subdivision.

For a given $\sigma$, there is a standard choice of primitive vector:
If $\{v_0, \dots, v_n\}$ are the primitive vectors in the $1$-dimensional faces of $\sigma$,
then \emph{standard star subdivision} is the star subdivision associated to $v$, which we take to be the \emph{barycenter} of $\sigma$, i.e., the primitive generator of the ray through $v_0 + \cdots + v_n$. 
For example, if $(X, D)$ is an snc pair over $R$ or $k$, then the standard star subdivision at a stratum $\sigma$ corresponds to the blowup of $X$ at $D_{\sigma}$.

\begin{theorem}[Toroidal resolution]
\label{thm:toroidal-resolution-of-singularities}
    Let $(X, D)$ be a toroidal embedding over $R$ or $k$. There exists a toroidal modification
    \[
        f : (X', D') \to (X, D)
    \]
    so that $(X', D')$ is an snc pair.
\end{theorem}

The construction of $X'$ is identical to the construction of a toric resolution, e.g., from \cite{fulton}, but with the toric fan replaced by the cone complex.

\subsection{Barycentric subdivision}

Let $(X, D)$ be a toroidal embedding over $R$ or $k$, and let $N$ be the largest codimension of a stratum in $(X, D)$. 
Then we may form a sequence of toroidal modifications
\begin{equation}
\label{eq:barycentric-sub-sequence}
    \begin{tikzcd}[column sep=small]
        (Y, E) = (Y^0, E^0) \ar[r] & (Y^1, E^1) \ar[r] & \cdots \ar[r] & (Y^N, E^N) \ar[r] & (Y^{N + 1}, E^{N + 1}) = (X, D),
    \end{tikzcd}
\end{equation}
where, for each $\ell \geq 0$, the morphism $Y^{\ell + 1} \to Y^\ell$ is defined inductively as follows:
If $\sigma_0, \dots, \sigma_k$ are the strata of codimension $\ell$ in $(X, D)$, and $\sigma_0^\ell, \dots, \sigma_k^\ell$ are their strict transforms in $Y^{\ell}$, then 
\[
    Y^{\ell + 1} \to Y^\ell
\]
is the composition of the iterated standard star subdivisions at $\sigma^\ell_0, \dots, \sigma^\ell_k$.

Recall that a strongly convex, rational polyhedral cone is simplicial if it is generated by linearly independent vectors.
A toroidal embedding $(X, D)$ with cone complex $\Sigma$ is \emph{simplicial} if the cones in $\Sigma$ are simplicial.

\begin{lemma}
\label{lem:barycentric-subdivision}
    Let $(Y, E) \to (X, D)$ be the barycentric subdivision of a toroidal embedding. 
    Then $(Y, E)$ is simplicial.
\end{lemma}

\begin{proof}
    See \cite[Lemma 4.4.3]{functorial-res}.
\end{proof}

The strata of the barycentric subdivision have a classical description in terms of chains in the poset $\rP(X, D)$.
Let $\Sigma'$ be the cone complex of the barycentric subdivision. 
From the description of $(Y, E)$ as an iterated standard star subdivision, the rays $v_0, \dots, v_\ell$ of any cone $\sigma' \in \Sigma'$ are the rays through the barycenters of cones $\sigma_0, \dots, \sigma_\ell$. 
Each $v_i$ appears at some stage $Y^{r_i} \to Y^{r_{i - 1}}$ in the sequence \eqref{eq:barycentric-sub-sequence}, 
and if one orders $v_0, \dots, v_\ell$ so that $r_0 > \cdots > r_\ell$, then one may check that
\[
    \sigma_0 < \cdots < \sigma_\ell
\]
in $\rP(X, D)$.
To summarize:

\begin{lemma}
\label{lem:barycentric-is-simplicial}
    Let $f\col(Y, E) \to (X, D)$ be the barycentric subdivision of a toroidal embedding. 
        The strata of codimension $\ell + 1$ in $Y$ are in bijection with chains
        \[
            \bsigma = (\sigma_0 < \sigma_1 < \cdots < \sigma_\ell)
        \]
        in the poset $\rP(X, D)$.
        Under this bijection, the partial ordering on $\rP(Y, E)$ is given by $\bsigma \leq \bsigma'$ if $\{\sigma_0, \dots, \sigma_\ell\} \subset \{\sigma_0', \dots, \sigma'_{\ell'}\}$, and the poset morphism
        \[
            f_*:\rP(Y, E) \to \rP(X, D)
        \]
        is given by $\sigma_0 < \cdots < \sigma_\ell \mapsto \sigma_\ell$. \qed
\end{lemma}

\section{Stable birational invariants} 
\label{sec:stable_birational_invariants}


\subsection{Functorial invariants}

Throughout, let $k$ be an algebraically closed field of characteristic $0$, and let $\cC$ be a category. 
We write $\SmProj_k$ for the category of smooth, irreducible, projective varieties over $k$. 
We cryptically observe that $\SmProj_k$ does not contain the empty scheme.
Nothing is gained or lost by working with proper, as opposed to projective, varieties.

\begin{definition}\label{def:stabirinv}
	Let $\cC$ be a category. A \textit{stable birational invariant} valued in $\cC$ is a functor
	\[
		\rF: \SmProj_k \to \cC
	\]
	such that for all birational morphisms $\mu \col X'\ra X$, $\rF(\mu)$ is an isomorphism, 
    and for all projections $p_1\col X \times \bP^n \to X$, 
    $\rF(p_1)$ is an isomorphism.
\end{definition}

If $\rF$ only satisfies the first condition, then $\rF$ is a \emph{birational invariant}. 
Curiously, there is no difference beween the two notions:

\begin{theorem}[Kahn--Sujatha, Colliot--Th\'el\`ene]
\label{thm:colliot-thelene}
    If $\rF\col \SmProj_k \to \cC$ is a birational invariant, then $\rF$ is a stable birational invariant.
\end{theorem}

\begin{proof}
    See \cite[1.7.2]{kahn-sujatha}. The idea is to reduce to the case $\cC = \Set$ using Yoneda's lemma. 
    Then, replacing $\rF$ with the birational invariants $\rF(X \times -)$ as $X$ ranges over $\SmProj_k$, it is enough to show that the natural map $\rF(\bP^1) \to \rF(\Spec k)$ is a bijection. 
    The final step is to apply a result of Colliot--Th\'el\`ene, \cite[Appendix A]{kahn-sujatha}.
\end{proof}

Our convention is that a \emph{stable birational invariant} always refers to a covariant functor; the contravariant case is covered by replacing $\cC$ with $\cC^\opp$.

\begin{remark}
	If $\rF$ is a stable birational invariant, then for any dominant morphism $f\col X\ra Y$ such that the associated extension of function fields is purely transcendental, the corresponding morphism $\rF(X)\to \rF(Y)$ is an isomorphism. In particular, for any stably rational variety $X$, applying $\rF$ to $X \to \Spec(k)$ yields an isomorphism $\rF(X) \simeq \rF(\Spec(k))$.
\end{remark}

\subsection{Evaluating on singular varieties}
\label{ssec:eval-on-singular}

Throughout, we fix a stable birational invariant
\[
    \rF:\SmProj_k \to \cC,
\]
where $k$ is an algebraically closed field of characteristic $0$.
We explain the extent to which $\rF$ may be evaluated on singular varieties over $k$; 
eventually, this construction is applied to the strata of a toroidal embedding.

\begin{definition}
\label{def:singstabirinv}
	If $X$ is an integral, proper variety over $k$, then we define
	\[
		\rF(X) := \limit \rF(X')
	\]
    where the inverse limit runs over all resolutions of singularities $X'\to X$. 
    Note that the limit  exists and is cofiltered (any two resolutions may be dominated by a third), and $\rF(X)$ is canonically isomorphic to $\rF(X')$ for any given resolution $X' \to X$. 

    If $f:X \to Y$ is a morphism of integral, proper varieties such that
\begin{equation}
\label{eqn:singcondition}
            f(X) \not\subset Y^\sing,
\end{equation}
    then we may define a morphism $\rF(X) \to \rF(Y)$ as
    \begin{equation}
    \label{eq:define-map-between-singular}
        \rF(f) = \limit\left( \rF(X') \to \rF(Y') \right),
    \end{equation}
    as the limit runs over all commutative diagrams
\begin{equation}
\label{eq:inverse-system-functorial}
    \begin{tikzcd}
        X' \ar[r] \ar[d] & Y' \ar[d] \\
        X \ar[r, "f"] & Y,
    \end{tikzcd}
\end{equation}
where the vertical maps are resolutions of singularities.
\end{definition}

The condition \eqref{eqn:singcondition} guarantees that the morphism $\rF(X') \to \rF(Y')$ does not depend on the diagram in the inverse system \eqref{eq:inverse-system-functorial}.
In other words, \eqref{eq:define-map-between-singular} is a cofiltered limit over a system of isomorphisms. 

\begin{lemma}\label{lem:indofchoice}
    Let $X, Y, Z$ be integral, proper varieties over $k$.
    If $f\col X \ra Y$ and $g\col Y \ra Z$ are two maps satisfying (\ref{eqn:singcondition}) such that $g\circ f$ also satisfies (\ref{eqn:singcondition}) then $\rF(g)\circ \rF(f) = \rF(g\circ f)$.
\end{lemma}

\begin{proof}
    By resolving closures of graphs of rational maps, we can find a diagram
    \begin{equation}
    \label{eq:triple-system}
        \begin{tikzcd}
            X' \ar[r, "f'"] \ar[d] & Y' \ar[r, "g'"] \ar[d] & Z' \ar[d] \\
            X \ar[r, "f"] & Y \ar[r, "g"] & Z,
         \end{tikzcd}
    \end{equation}
    where the vertical maps are resolutions. 
    Writing $h = f \circ g$, $h' = f' \circ g'$, we get two diagrams:
    \[
        \begin{tikzcd}
            \rF(X') \ar[r, "\rF(f')"] \ar[d, "\simeq"] & \rF(Y') \ar[r, "\rF(g')"] \ar[d, "\simeq"] & \rF(Z') \ar[d, "\simeq"] \\
            \rF(X) \ar[r, "\rF(f)"] & \rF(Y) \ar[r, "\rF(g)"] & \rF(Z),
        \end{tikzcd} \quad
        \begin{tikzcd}
            \rF(X') \ar[r, "\rF(h')"] \ar[d, "\simeq"] & \rF(Z') \ar[d, "\simeq"] \\
            \rF(X) \ar[r, "\rF(h)"] & \rF(Z).
        \end{tikzcd}
    \]
    We conclude from the fact that $\rF(g') \circ \rF(f') = \rF(h')$.
\end{proof}

\subsection{Evaluating on strata}

We continue with the situation of \S\ref{ssec:eval-on-singular}, so that $\rF$ is a fixed stable birational invariant from $\SmProj_k$ to a category $\cC$.
Our ultimate goal is to apply $\rF$ to the strata of a toroidal embedding over $R = k\llbracket t \rrbracket$ or $k$, and to the inclusions between strata, to construct a chain complex. 
Since $\rF$ is only defined on $k$-schemes, we require that (in the case when $X$ is a toroidal embedding over $R$) all of the strata lie in the special fiber:
\begin{definition}
    A toroidal embedding $(X, D)$ over $R$ or $k$ is \emph{vertical} if $D$ is contained in the special fiber $X_0$.
    We hasten to add that any toroidal embedding $(X, D)$ over $k$ is automatically vertical; the definition is made in this case only for the sake of brevity in what follows.
\end{definition}
Let $(X, D)$ be a vertical toroidal embedding over $R$ or $k$.
In \S\ref{ssec:eval-on-singular}, we have described how to apply $\rF$ to the strata of $(X, D)$.
In order to avoid notational clutter, we often write
\begin{equation*}
    \rF(\sigma) := \rF(D_{\sigma}),
\end{equation*}
for $\sigma \in \rP(X, D)$.

Our next goal is to show that $\rF(-)$ can be applied in a suitable manner to inclusions between strata. 
On singular varieties, $\rF(-)$ is only functorial for composable morphisms $f, g$ such that the images of $f, g,$ and $g \circ f$ meet the nonsingular loci of the targets, cf. \S \ref{lem:indofchoice}.
This property does not generally hold for a composable pair of inclusions between strata of a toroidal embedding. 
Nevertheless, by exploiting properties of toroidal resolutions, we show in Proposition~\ref{prop:existence}  below that there are natural homomorphisms
\[
    \rho_{\sigma, \tau} : \rF(\sigma) \to \rF(\tau),
\]
for every inclusion of strata $D_{\sigma} \subset D_{\tau}$, compatible with composition.
These are eventually used to construct chain complexes associated to toroidal embeddings.

\begin{remark}[Vertical morphisms]
\label{rem:vertical-morphisms}
Let $f\col(Y, E) \to (X, D)$ be a proper morphism between vertical toroidal embeddings over $R$ or $k$, or a morphism in $\TorMod_{\infty}$. 
For $\sigma \in \rP(Y, E)$, the image of the morphism $E_{\sigma} \to D_{f_* \sigma}$ meets the nonsingular locus of $D_{f_* \sigma}$.
Indeed, the singular locus of $D_{f_* \sigma}$ is contained in the union of the proper substrata, while $f(E_{\sigma})$ meets the interior of $D_{f_* \sigma}$ from the definition of $f_* \sigma$.
As a result, the morphism
\[
    \rF(f|_{E_{\sigma}}) : \rF(\sigma) \to \rF(f_* \sigma)
\]
is defined from our discussion in \S\ref{ssec:eval-on-singular}.
\end{remark}

\begin{proposition}
\label{prop:existence}
    Let $(X, D)$ be a vertical toroidal embedding over $R$ or $k$. 
    For each $\sigma_0 \leq \sigma_1$ in $\rP(X, D)$, there is a morphism
    \[
        \rho_{\sigma_0, \sigma_1} : \rF(\sigma_1) \to \rF(\sigma_0),
    \]
    uniquely determined by the following properties:
    \begin{enumerate} 
        \item \label{item:nonsingular-req} If $D_{\sigma_1}$ is not contained in the singular locus of $D_{\sigma_0}$, then $\rho_{\sigma_0, \sigma_1} = \rF(\iota)$, $\iota:D_{\sigma_1} \subset D_{\sigma_0}$.
        \item \label{item:hor-functoriality} Given $\sigma_0 \leq \sigma_1 \leq \sigma_2$, $\rho_{\sigma_0, \sigma_1} \circ \rho_{\sigma_1, \sigma_2} = \rho_{\sigma_0, \sigma_2}$.
        \item \label{item:vert-functoriality} If $f\col(Y, E) \to (X, D)$ is a proper morphism of toroidal embeddings over $k$ (resp., morphism in $\TorMod_{\infty}$), then for any $\tau_0 \leq \tau_1$ in $\rP(Y, E)$, the diagram
        \[
            \begin{tikzcd}[column sep=huge]
                \rF(\tau_1) \ar[d, "\rF(f|_{E_{\tau_1}})"'] \ar[r, "\rho_{\tau_0, \tau_1}"] & \rF(\tau_0) \ar[d, "\rF(f|_{E_{\tau_0}})"] \\
                \rF(f_* \tau_1) \ar[r, "\rho_{f_* \tau_0, f_* \tau_1}"] & \rF(\tau_0)
            \end{tikzcd}
        \]
        commutes. 
    \end{enumerate}
\end{proposition}

While the morphisms $\rho_{\sigma_0, \sigma_1}$ are defined in an ad hoc fashion in the proof, we briefly indicate a direct definition.
Let us say that a \emph{complete flag} is a sequence 
\[
    \sigma_0 < \cdots < \sigma_n
\]
such that $D_{\sigma_{i + 1}}$ has codimension $1$ in $D_{\sigma_i}$ for each $i$.
Since each $D_{\sigma_i}$ is normal, $D_{\sigma_{i + 1}}$ is not contained in the singular locus of $D_{\sigma_i}$, so there is a morphism
\[
    \rF(D_{\sigma_{i + 1}}) \to \rF(D_{\sigma_i}).
\]
One may check that any given pair of strata $\sigma_0$ and $\sigma_n$ are connected by a complete flag as above.\footnote{For example, using the correspondence between strata and cones in the cone complex $\Sigma(X, D)$, one reduces to the corresponding statement about the faces of strongly convex, rational polyhedral cones.} 
Hence, we obtain a morphism $\rF(\sigma_0) \to \rF(\sigma_n)$ by composition, which must be equal to $\rho_{\sigma_0, \sigma_1}$ by the properties listed in Proposition~\ref{prop:barycentric-subdivision} (and, in particular, does not depend on the choice of complete flag).

\subsection{Proof of Proposition~\ref{prop:existence}}

First, let $(W, B)$ be a toroidal embedding over $k$. 
Fix a stratum $\sigma \in \rP(W, B)$, and a toroidal resolution $g\col(W', B') \to (W, B)$.
For any choice of a stratum $\sigma' \in \rP(W', B')$ such that $g_* \sigma' = \sigma$, we may define a morphism $\varphi_{\sigma}$ by requiring the diagram
\begin{equation}
\label{eq:defining-rho}
    \begin{tikzcd}
        \rF(B'_{\sigma'}) \ar[d, "\simeq"] \ar[r] & \rF(W') \ar[d, "\simeq"] \\
        \rF(B_{\sigma}) \ar[r, dotted, "\varphi_{\sigma}"] & \rF(W) 
    \end{tikzcd}
\end{equation}
to  commute.

\begin{lemma}
\label{lem:rho-does-not-depend}
    Keeping the resolution $g \col (W', B') \to (W, B)$ fixed, the morphism $\varphi_{\sigma}$ defined above does not depend on the choice of $\sigma'$.
\end{lemma}

\begin{proof}
    For a stratum $\sigma'$ such that $g_* \sigma' = \sigma$, we write $\varphi(\sigma')$ for the morphism from \eqref{eq:defining-rho}. 
    The first observation is that if $B'_{\sigma'} \subset B'_{\sigma''}$, then $\varphi(\sigma') = \varphi(\sigma'')$.
    Indeed, in the diagram
    \[
        \begin{tikzcd}
            \rF(B'_{\sigma'}) \ar[dr, "\simeq"'] \ar[r] & \rF(B'_{\sigma''}) \ar[r] \ar[d, "\simeq"] & \rF(W') \ar[d, "\simeq"] \\
            & \rF(B_{\sigma}) \ar[r, dotted] & \rF(W) ,
        \end{tikzcd}
    \]
    the left-hand triangle commutes by Lemma~\ref{lem:indofchoice}.
    
    Let $Z_1, \dots, Z_r$ be the irreducible components of $g^{-1}(B_{\sigma})$. Note that each $Z_i$ is a stratum, by Lemma~\ref{lem:fibers-of-toroidal-resolution}.
    Any $\sigma'$ mapping to $\sigma$ is contained in some $Z_i$, and from the previous paragraph, $\varphi(\sigma') = \varphi(Z_i)$.
    Hence, we just need to show that $\varphi(Z_i) = \varphi(Z_j)$ for all $i,j$.

    For brevity, we say that $Z_i$ and $Z_j$ are \emph{adjacent} if the intersection $Z_i \cap Z_j$ (which is a finite disjoint union of strata) contains a connected component, say $Z_{ij}^0$, which dominates $B_{\sigma}$.
    If $Z_i$ and $Z_j$ are adjacent, then $\varphi(Z_{ij}^0) = \varphi(Z_i) = \varphi(Z_j)$. 
    Therefore, we just need to show that any two $Z_i$ and $Z_j$ are related by a sequence of adjacent components. 

    Let $\eta_1, \dots, \eta_r$ be the generic points of $Z_1, \dots, Z_r$, respectively, and let $\eta_\sigma$ be the generic point of $B_{\sigma}$.
    Then $Z_i$ and $Z_j$ are adjacent if and only if $\eta_i$ and $\eta_j$ have a common specialization within the fiber $g^{-1}(\eta_\sigma)$, so we just have to show that $\eta_i$ and $\eta_j$ are related by a zigzag of specializations within $g^{-1}(\eta_\sigma)$, or, equivalently, that $g^{-1}(\eta_\sigma)$ is connected.
    But $W$ is normal and $g$ is a proper birational morphism, so $g^{-1}(\eta_\sigma)$ is connected by Zariski's main theorem.
\end{proof}

We return to the proof of Proposition~\ref{prop:existence}.
    In the case when $D_{\sigma_1} \subset D_{\sigma_0}$ are both nonsingular,
    take Item~\ref{item:nonsingular-req} as the definition of $\rho_{\sigma_0, \sigma_1}$.
    To define $\rho_{\sigma_0, \sigma_1}$ in general, 
    fix strata $\sigma_0 < \sigma_1$ in $\rP(X, D)$.
    For simplicity, we write $W = D_{\sigma_0}$, and let $(W, B) := (D_{\sigma_0}, D|_{D_{\sigma_0}})$ be the toroidal embedding over $k$ obtained by restriction from $(X, D)$.
    Then $D_{\sigma_1} = B_{\sigma_1}$ may be regarded as a stratum in $W$.
    Fix a toroidal resolution $g\col (W', B') \to (W, B)$ be a toroidal resolution, 
    and let $\rho_{\sigma_0, \sigma_1}$ be the morphism $\varphi_{\sigma_1}$ defined above.
    
    We show Item~\ref{item:vert-functoriality}, which implies that $\rho_{\sigma_0, \sigma_1}$ does not depend on the choices made above. 
    Throughout, we write $\tau_0 < \tau_1$ for the chain in $\rP(Y, E)$ as in the statement of Item~\ref{item:vert-functoriality}, and $\sigma_0 = f_* \tau_0$, $\sigma_1 = f_* \tau_1$.
    The morphism $\rho_{\tau_0, \tau_1}$ has been defined by the previous paragraph, i.e., let $(V, C) = (E_{\tau_0}, E|_{E_{\tau_0}})$, and let $(V', C') \to (V, C)$ be the toroidal embedding defining $\rho_{\tau_0, \tau_1}$.
    We write $C_{\tau_1} = E_{\tau_1}$, and let $C'_{\tau_1} \subset V'$ be any stratum which dominates $C_{\tau_1}$. 
    Then we need to show that the diagram
    \[
        \begin{tikzcd}
            \rF(C'_{\tau_1}) \ar[d] \ar[r] & \rF(V') \ar[d] \\
            \rF(B_{\sigma_1}) \ar[r, "\varphi_{\sigma_1}"] & \rF(W)
        \end{tikzcd}
    \]
    commutes.
    By blowing up $V'$ at $C_{\tau_1}'$, we can assume that $C'_{\tau_1}$ is a divisor in $V'$.

    To simplify notation, replace $V'$ with $V$ and refer to the divisor $C'_{\tau_1}$ as $C \subset V$.
    The situation is now that we have a proper morphism $V \to W =  D_{\sigma_0}$ from a smooth variety $V$, sending a smooth divisor $C \subset Y$ to $B_{\sigma_1} = D_{\sigma_1}$.
    The morphisms $V \to W$, $C \to B_{\sigma_1}$ may not be dominant, but their images are not contained in any proper substrata. 
    Our goal is to show that the diagram
    \[
        \begin{tikzcd}[column sep=large]
            \rF(C) \ar[r] \ar[d] & \rF(V) \ar[d] \\
            \rF(B_{\sigma_1}) \ar[r, "\varphi_{\sigma_1}"] & \rF(W)
        \end{tikzcd}
    \]
    commutes, where all arrows are the resulting of applying $\rF$ to a morphism of varieties, except for $\varphi_{\sigma_1}$, which is defined by the toroidal resolution $(W', B')$ from above.

    To prove it, 
    consider the commutative diagram
    \begin{equation}
    \label{eq:five-pt-diagram}
        \begin{tikzcd}
            C' \ar[r] & V' \ar[d, "f'"] \ar[r, "h"] & V \ar[d, "f"] \\
             & W' \ar[r, "g"] & W,
        \end{tikzcd} 
    \end{equation}
    where:
    \begin{itemize}
         \item $h:V' \to V$ is a proper, birational map of smooth varieties resolving the indeterminacy of the rational map $V \dra W'$;
         \item $C'$ is a resolution of the strict transform of $C$ in $Y'$.
     \end{itemize} 
    Since $C'$ eventually maps onto $B_{\sigma_1} \subset W$, the image of $C'$ in $W'$ lies in an irreducible compoment, say $Z'$, of $g^{-1}(B_{\sigma_1})$.
    By Lemma~\ref{lem:fibers-of-toroidal-resolution}, $Z'$ is a stratum of $W'$ which dominates $B_{\sigma_1}$. 
    In particular, $Z'$ is smooth, and by Lemma~\ref{lem:rho-does-not-depend}, the diagram
    \begin{equation}
    \label{eq:zprime-square}
        \begin{tikzcd}
            \rF(Z') \ar[d] \ar[r] & \rF(W') \ar[d] \\
            \rF(B_{\sigma_1}) \ar[r, "\varphi_{\sigma_1}"] & \rF(W)
        \end{tikzcd}
    \end{equation}
    commutes, where the unmarked morphisms come from applying $\rF$ to the obvious maps of varieties.
    We get two diagrams
    \[
        \begin{tikzcd}
            \rF(C') \ar[d]  \ar[dr, phantom, "a"] \ar[r, "\simeq"] & \rF(C) \ar[d] \ar[dr, phantom, "b"] \ar[r] & \rF(V) \ar[d] \\
            \rF(Z') \ar[r, "\simeq"] & \rF(B_{\sigma_1}) \ar[r, "\varphi_{\sigma_1}"'] & \rF(W) ,
         \end{tikzcd} \quad
         \begin{tikzcd}
             \rF(C') \ar[d] \ar[dr, phantom, "c"] \ar[r] & \rF(Y') \ar[d] \ar[dr, phantom, "d"] \ar[r, "\simeq"] & \rF(V) \ar[d] \\
            \rF(Z') \ar[r] & \rF(W') \ar[r, "\simeq"] & \rF(W) .
         \end{tikzcd}
    \]
    Except for $\varphi_{\sigma_1}$, all of the displayed morphisms are the result of applying $\rF$ to a morphism in \eqref{eq:five-pt-diagram}.
    Our goal is to show that Square~$b$ commutes.

    Now, Square~$a$ commutes by Lemma~\ref{lem:indofchoice}, so Square~$b$ commutes if and only if the outer rectangle, which we call Rectangle~$ab$, commutes.
    On the other hand, \eqref{eq:zprime-square} implies that Rectangle $ab$ equals Rectangle $cd$, so it is equivalent to show that Rectangle~$cd$ commutes.
    Then Square~$d$ commutes by Lemma~\ref{lem:indofchoice}, so Rectangle~$cd$ commutes if and only if Square~$c$ commutes.
    Finally, Square~$c$ commutes because it comes from a commutative diagram of smooth varieties.
    We conclude that Square~$b$ commutes; this finishes the proof of Item~\ref{item:vert-functoriality}.

    We show Item~\ref{item:hor-functoriality}.
    Suppose that we have constructed morphisms $\rho_{\sigma, \tau}$ satisfying Item~\ref{item:vert-functoriality}.
    Let $f\col(X', D') \to (X, D)$ be a toroidal resolution, with $\sigma_0 \leq \sigma_1 \leq \sigma_2$ a chain in $\rP(X, D)$. 
    Lemma~\ref{lem:path-lifting} below says that we may find a chain $\tau_0 \leq \tau_1 \leq \tau_2$, so that $D'_{\tau_i}$ maps birationally onto $D_{\sigma_i}$ for each $i$.
    Consider the diagrams
    \[
        \begin{tikzcd}
            \rF(\tau_2) \ar[r, "\rho_{\tau_1, \tau_2}"] \ar[d, "\simeq"] & \rF(\tau_1) \ar[r, "\rho_{\tau_0, \tau_1}"] \ar[d, "\simeq"] & \rF(\tau_0) \ar[d, "\simeq"] \\
            \rF(\sigma_2) \ar[r, "\rho_{\sigma_1, \sigma_2}"] & \rF(\sigma_1) \ar[r, "\rho_{\sigma_0, \sigma_1}"] & \rF(\sigma_0),
        \end{tikzcd} \quad
        \begin{tikzcd}
            \rF(\tau_2) \ar[d, "\simeq"] \ar[r, "\rho_{\tau_0, \tau_2}"] & \rF(\tau_0) \ar[d, "\simeq"] \\
            \rF(\sigma_2) \ar[r, "\rho_{\sigma_0, \sigma_2}"] & \rF(\sigma_0)
        \end{tikzcd}
    \]
    Both diagrams are commutative by Item~\ref{item:vert-functoriality}.
    Moreover, since the strata of $(X', D')$ are smooth, $\rho_{\tau_1, \tau_2} \circ \rho_{\tau_0, \tau_1} = \rho_{\tau_0, \tau_2}$.
    These two facts imply that $\rho_{\sigma_1, \sigma_2} \circ \rho_{\sigma_0, \sigma_1} = \rho_{\sigma_0, \sigma_2}$.
    A similar argument shows Item~\ref{item:nonsingular-req} in the case when $D_{\sigma}$ is not contained in the singular locus of $D_{\tau}$. \qed

\begin{lemma}[Path lifting]
\label{lem:path-lifting}
    Let $g\col (Y, E) \to (X, D)$ be a toroidal modification of a toroidal embedding over $R$ or $k$.
    \label{item:lift-one-path} If $\bsigma = (\sigma_0 < \cdots < \sigma_\ell)$
    is an $\ell$-chain in $\rP(X, D)$, then there exists an $\ell$-chain
    \[
        \btau = (\tau_0 < \cdots < \tau_\ell)
    \]
    in $\rP(Y, E)$ such that for all $i$, $E_{\tau_i}$ maps birationally to $D_{\sigma_i}$ via $g$.
\end{lemma}

\begin{proof}
We write $\varphi:\Sigma' \to \Sigma$ for the corresponding map of cone complexes, 
and $|\sigma|^\circ$ for the interior of a cone.
For each $i$, let $W_i \subset |\sigma_i|$ be the union of the images of lower-dimensional cones from $\Sigma'$.
Let $U_i = |\sigma_i| \setminus W_i$ be the complement. 
By definition, if $p \in U_i$ is in the image of $|\tau| \subset |\Sigma'|$, then $|\tau|$ has a face $|\tau'|$ mapping to $|\sigma_i|$ with $\dim |\tau_i| = \dim |\sigma_i|$.

We define $p_0 \in U_0, \dots, p_\ell \in U_\ell$ inductively.
First, choose any $p_0 \in U_0$.
Given $p_{i - 1}$, we construct $p_{i}$ as follows: Regard $p_{i - 1}$ as an element of $|\sigma_i|$ via the inclusion $|\sigma_{i - 1}| \subset |\sigma_i|$.
Since $U_{i}$ is dense in $|\sigma_i|$ and the set of walls $W_i$ is finite, we may choose an element $p_i$ of $U_i$ which is not separated from $p_{i - 1}$ by any wall, i.e., $p_{i - 1}$ lies in the closure of the connected component of $U_i$ containing $p_i$.
If $|\tau| \subset |\Sigma'|$ contains $p_{i}$, then it contains the connected component of $U_i$ containing $p_i$, as well as its closure.
The upshot is that if any $|\tau| \subset |\Sigma'|$ contains $p_i$, then it also contains $p_{i - 1}$.

In the end, we obtain a point $p_\ell$ in $U_\ell$. 
From the definition of a subdivision, there exists $|\tau_\ell|$ of the same dimension as $|\sigma_\ell|$ such that $|\tau_\ell|^\circ$ contains $p_\ell$. By the previous paragraph, $|\tau_\ell|$ contains $p_{\ell - 1} \in U_{\ell - 1}$.
As we explained in the first paragraph, $|\tau_\ell|$ has a face $|\tau_{\ell - 1}|$ of the same dimension as $|\sigma_{\ell - 1}|$ containing $p_{\ell - 1}$. 
Then $|\tau_{\ell - 1}|$ contains $p_{i - 2}$, and we proceed by induction.
In the end, we get a sequence of cones $|\tau_\ell| \supset \cdots \supset |\tau_0|$, such that $\dim |\tau_i| = \dim |\sigma_i|$ and $|\tau_i|$ maps to the interior of $|\sigma_i|$.
From the correspondence between cones and strata, the corresponding strata $\tau_0 < \cdots < \tau_\ell$ satisfy the conclusion of the lemma.
\end{proof}

\section{Chain complexes from birational invariants}
\label{sec:chain-complexes}

\subsection{The subdivision complex}

Throughout, we fix $(X, D)$, a vertical toroidal embedding over $R = k\llbracket t \rrbracket$ or $k$. 
We also fix a stable birational invariant
\[
    \rF : \SmProj_k \to \cA
\]
valued in an \emph{additive} category $\cA$.

Write $\rP = \rP(X, D)$.
We write $\Sd_n(\rP)$ and $\Sd_n^+(\rP)$ for the following sets of chains in the poset $\rP$:
\begin{align*}
    \Sd_n(\rP) &= \{(\sigma_0 < \cdots < \sigma_n) \mid \sigma_i \in \rP\} , \\
    \Sd_n^+(\rP) &= \{(\sigma_0 \leq \cdots \leq \sigma_n) \mid \sigma_i \in \rP\}.
\end{align*} 
The elements of $\Sd_n(\rP)$ are called \emph{nondegenerate $n$-chains}. 
In other words, an $n$-chain is \emph{degenerate} if it contains a repetition.

For $0 \leq i,j \leq n$, there are \emph{face maps} and \emph{degeneracy maps}
\begin{align*}
     d_j: & \Sd^+_{n}(\rP) \to \Sd^+_{n - 1}(\rP), \\
     & (\sigma_0 \leq \cdots \leq \sigma_n) \mapsto (\sigma_0 \leq \cdots \leq \widehat{\sigma}_j \leq \cdots \leq \sigma_n), \\
     s_i: & \Sd_{n}^+(\rP) \to \Sd^+_{n + 1}(\rP), \\
     & (\sigma_0 \leq \cdots \leq \sigma_n) \mapsto (\sigma_0 \leq \cdots \leq \sigma_i \leq \sigma_i \leq \cdots \leq \sigma_n).
\end{align*}
The face maps preserve nondegeneracy, while the degeneracy maps do not.

\begin{remark}
    The face and degeneracy maps make $\Sd^+(\rP)$ into a simplicial set; although we shall not directly use the language of simplicial sets, most of the constructions presented in this section are special cases of standard manipulations involving simplicial sets and simplicial objects in additive categories.
\end{remark}

Given an element $\bsigma = (\sigma_0 \leq \sigma_1 \leq \cdots \leq \sigma_n)$ of $\Sd^+_n(\rP)$, we set
\[
    \rF(\bsigma) := \rF(\sigma_n).
\]
We similarly obtain face and degeneracy maps after applying $\rF$:
\begin{align*}
    d_j: & \rF(\bsigma) \to \rF(d_j \bsigma), \quad d_j = \begin{cases}
        \id : \rF(\sigma_n) \to \rF(\sigma_n) & j < n \\
        \rho_{\sigma_{n - 1}, \sigma_n} : \rF(\sigma_n) \to \rF(\sigma_{n - 1}), & j = n.
    \end{cases}\\
     s_i: & \rF(\bsigma) \to \rF(s_i \bsigma), \quad s_i = \id : \rF(\sigma_n) \to \rF(\sigma_n).
\end{align*}
Here, the morphism $\rho_{\sigma_{n - 1}, \sigma_n}$ was constructed in Proposition~\ref{prop:existence}.

\begin{definition}
\label{def:subdivision-complex}
    With the notation from above, we define the \emph{extended subdivision complex} $\Sd^+(X, D; \rF)$ to be the chain complex
    \[
        \begin{tikzcd}
            \cdots \ar[r] & \displaystyle{\bigoplus_{\bsigma \in \Sd^+_2(\rP)} \rF(\bsigma) } \ar[r] & \displaystyle{\bigoplus_{\bsigma \in \Sd^+_1(\rP)} \rF(\bsigma) } \ar[r] & \displaystyle{\bigoplus_{\bsigma \in \Sd^+_0(\rP)} \rF(\bsigma) } \ar[r] & 0 ,
        \end{tikzcd}
    \]
    where the differential is given on each summand $\rF(\bsigma)$ by 
    \[
        d = (-1)^j d_j, \quad d_j: \rF(\sigma) \to \rF(d_j \sigma).
    \]
    To check that $\Sd^+(X, D; \rF)$ is indeed a chain complex, i.e., that $d^2 = 0$, one argues that the terms of $d^2$ cancel in pairs.
    For this, the only nontrivial point is to show that 
    \begin{equation*}
        \rho_{\sigma_{n - 2}, \sigma_{n - 1}} \circ \rho_{\sigma_{n - 1}, \sigma_n} = \rho_{\sigma_{n - 2}, \sigma_{n}},
    \end{equation*}
    which follows from Proposition~\ref{prop:existence}.

    Similarly, we define the \emph{subdivision complex} $\Sd(X, D; \rF)$, which is the subcomplex of $\Sd^+(X, D; \rF)$ given in degree $n$ by 
    \[
        \Sd_n(X, D; \rF) = \bigoplus_{\bsigma \in \Sd_n(\rP)} \rF(\bsigma).
    \]
    While the extended subdivision complex $\Sd^+(X, D; \rF)$ is unbounded, the subdivision complex $\Sd(X, D; \rF)$ is bounded. 
\end{definition}

 Our next goal is to show that the two complexes are homotopic; to that end, we consider the degenerate chains. 
    For $n \geq 0$, define
    \[
        \Sd^{deg}_n(X, D; \rF) = \bigoplus_{\bsigma} \rF(\bsigma),
    \]
    as $\bsigma$ runs over the degenerate chains in $\Sd_n^+(\rP)$, i.e., the chains with repetitions.
    The following lemma shows that the objects $\Sd^{deg}_n(X, D; \rF)$ form a subcomplex.

\begin{lemma}
Let $(X, D)$ be a toroidal embedding over $R$ or $k$.
\label{lem:inclusion-of-nondegenerate}
    \begin{enumerate} 
        \item The objects $\Sd^{deg}_n(X, D; \rF) \subset \Sd_n^+(X, D; \rF)$, $n \geq 0$ form a subcomplex of $\Sd^+(X, D; \rF)$.
        \item There is a splitting $\Sd^+(X, D; \rF) = \Sd(X, D; \rF) \oplus \Sd^{deg}(X, D; \rF)$, 
        \item The inclusion $\Sd(X, D; \rF) \to \Sd^+(X, D; \rF)$ is a homotopy equivalence, with homotopy inverse given by the projection away from $\Sd^{deg}(X, D; \rF)$.
    \end{enumerate}
\end{lemma}

\begin{proof}
    For simplicity, we write $\Sd^{deg} = \Sd^{deg}(X, D; \rF)$, etc. 
    We check that the differential $d$ carries $\Sd^{deg}_n$ to $\Sd^{deg}_{n - 1}$. 
    If $\bsigma = (\sigma_0 \leq \cdots )$ is degenerate, then $d_j \bsigma$ is degenerate unless $\sigma_j$ appears exactly twice: Either $\sigma_{j - 1} = \sigma_j$ or $\sigma_j = \sigma_{j + 1}$.
    Hence, the nondegenerate terms of $d = (-1)^j d_j$ cancel in pairs.
    It follows that $\Sd^+(X, D; \rF)$ admits the desired splitting.

    For the second claim, the only nontrivial point is to give a homotopy between $\iota \circ \pi$ and the identity, where $\pi$ and $\iota$ are the projection and inclusion of $\Sd$.
    The homotopy is given by the sum $\sum (-1)^i s_i$, where $s_i$ are the degeneracy maps defined above.
\end{proof}

\begin{definition}[Morphisms between subdivision complexes]
\label{def:morphism-between-subdivision}
    Let $f\col (Y, E) \to (X, D)$ be a proper morphism of vertical toroidal embeddings over $R$ or $k$, or a morphism in $\TorMod_{\infty}$.
    We write $\rQ = \rP(Y, E)$, $\rP = \rP(X, D)$.
    Then is a morphism of complexes
    \[
        \Sd^+(f_*) : \Sd^+(Y, E; \rF) \to \Sd^+(X, D; \rF),
    \]
    given as follows: 
    For $\bsigma = (\sigma_0 \leq \cdots \leq \sigma_n) \in \Sd^+_n(\rQ)$, we set 
    \[
        f_* \bsigma = (f_* \sigma_0 \leq \cdots \leq f_* \sigma_n) \in \Sd_n^+(\rP).
    \]
    Then the desired morphism $\Sd^+(f_*)$ is given on the summand $\rF(\bsigma)$ by the map
    \[
        \rF(\bsigma) = \rF(\sigma_n) \to \rF(f_* \sigma_n) = \rF(f_* \bsigma),
    \]
    where the middle morphism is the vertical morphism from Remark~\ref{rem:vertical-morphisms}

    Similarly, we define
    \[
        \Sd(f_*) : \Sd(Y, E; \rF) \to \Sd(X, D; \rF),
    \]
    as $\pi \circ \Sd^+(f_*) \circ \iota$, where $\iota \col \Sd \to \Sd^+$ is the inclusion and $\pi \col \Sd^+ \to \Sd$ is the projection.
\end{definition}

\begin{lemma}
\label{lem:functoriality-of-sd}
    Let $\cC$ be one of the following categories:
    \begin{enumerate} 
        \item Toroidal embeddings over $k$, with proper morphisms of pairs.
        \item Vertical toroidal embeddings over $R$, with proper morphisms of pairs.
        \item The category of proper toroidal models, $\TorMod_{\infty}$.
    \end{enumerate}
    Then there is functor $\cC \to \Ch_{\geq 0} \cA$, given by $(X, D) \mapsto \Sd^+(X, D; \rF)$, $f \mapsto \Sd^+(f_*)$.
    Similarly, there is a functor given by sending $(X, D)$ to $\Sd(X, D; \rF)$, and $f \mapsto \Sd(f_*)$.
\end{lemma}

\begin{proof}
    These all boil down to Proposition~\ref{prop:existence}.
\end{proof}

\subsection{Orderings and the \v{C}ech complex}

For a toroidal embedding $(X, D)$ over $R$ or $k$, we write $\rP = \rP(X, D)$, and $\rP_n \subset \rP$ for the set of strata of codimension $n + 1$ in $X$. 
For example, $\rP_0$ is the set of irreducible components of $D$.
For a stratum $\sigma$, we write $\vertex(\sigma) \subset \rP_0$ for the subset of divisors containing $D_{\sigma}$.

\begin{definition}
\label{def:ordering-on-simplicial-toroidal-embedding}
    Let $(X, D)$ be a simplicial toroidal embedding over $R$ or $k$.
    An \emph{ordering} on $(X, D)$ is a partial ordering on $\rP_0$ which restricts to a linear ordering on $\vertex(\sigma)$ for all $\sigma \in \rP$.
    For an ordered simplicial toroidal embedding, 
    we usually write elements of $\rP_0$ using the symbols $i, j, k$, etc., and the partial ordering as $i \preceq j$, etc.
\end{definition}

Any choice of linear ordering $\preceq$ on $\rP_0$ gives an ordering on $(X, D)$; we refer to such an $(X, D)$ as a \emph{linearly ordered} toroidal embedding.
The reason for considering more general (nonlinear) orderings is specifically to cover the example of a barycentric subdivision $(X', D') \to (X, D)$ of an arbitrary (not necessarily simplicial) toroidal model $(X, D)$. Indeed, $(X', D')$ has a natural ordering which is not globally linear, as is explained below in \S \ref{ssec:barycentric-sub}.

\begin{situation}
\label{sit:ordered-simplicial}
    Throughout, we let $(X, D)$ be an ordered, simplicial, vertical toroidal embedding over $R$ or $k$.
    As above, we write $\rP = \rP(X, D)$, and denote by $\rP_n$ the set of strata of codimension $n$.
\end{situation}

\begin{definition}
\label{def:cech-cpx}
    In Situation~\ref{sit:ordered-simplicial}, the \emph{\v{C}ech complex} $\cech(X, D; \rF)$ is the bounded chain complex
    \[
        \begin{tikzcd}
            0 \ar[r] & \cdots \ar[r] & \displaystyle{\bigoplus_{i \prec j \prec k} \rF(D_{ijk})} \ar[r] & \displaystyle{\bigoplus_{i \prec j} \rF(D_{ij}) } \ar[r] & \displaystyle{\bigoplus_{i} \rF(D_i) } \ar[r] & 0,
        \end{tikzcd}
    \]
    where, by definition, $\rF(D_{I}) = \bigoplus \rF(D_{\sigma})$ as $D_{\sigma}$ runs over the connected components of the intersection $D_{I} = \bigcap_{i \in I} D_i$, and the direct sums run over the linearly ordered subsets $I \subset \rP_0$.
    The differential is the usual differential in \v{C}ech theory; just as with the subdivision complex, the fact that $d^2 = 0$ follows from Proposition~\ref{prop:existence}.
   
    The \emph{extended \v{C}ech complex} $\cech^+(X, D, \rF)$ is the unbounded chain complex given by 
    \[
        \begin{tikzcd}
             \cdots \ar[r] & \displaystyle{\bigoplus_{i \preceq j \preceq k} \rF(D_{ijk})} \ar[r] & \displaystyle{\bigoplus_{i \preceq j} \rF(D_{ij}) } \ar[r] & \displaystyle{\bigoplus_{i} \rF(D_i) } \ar[r] & 0,
        \end{tikzcd}
    \]
    where the direct sums run over the linearly ordered subsets of $\rP_0$ with repetitions. 
\end{definition}

We frequently need to expand the summands $\rF(D_I)$ above into the direct sums $\bigoplus \rF(D_\sigma)$; 
to do so, we introduce a more precise notation, which keeps track of both the linearly ordered set $I$ and the connected component $D_\sigma$.

\begin{notation}[Coordinates]
\label{notation:coordinates}
    In Situation~\ref{sit:ordered-simplicial}, for any stratum of codimension $n$, $\vertex(\sigma)$ may be identified with a linearly ordered subset, say $I = \{p_0 \prec \dots \prec p_n\}$, of $\rP_0$.
    Moreover, any $\tau \leq \sigma$ is uniquely determined by a subset $I' = \{p_{\ell_1}, \dots, p_{\ell_k}\}$ of $\{p_0, \dots, p_n\}$; 
    indeed, $D_\tau$ is the unique connected component of $D_{I'}$ containing $D_{\sigma}$.
    We express the situation notationally by writing
    \[
        \tau = [p_{\ell_0}, \dots, p_{\ell_k}]_{\sigma}.
    \]
    In this notation, the differential $d$ of the \v{C}ech complex is given by 
    \[
        d = (-1)^j d_j : \rF(\sigma) \to \rF(d_j \sigma),
    \]
    where, if $\vertex(\sigma) = \{p_0 \prec \dots \prec p_n\} \subset \rP_0$, then 
    \[
        d_j \sigma = [p_0, \dots, \widehat{p}_j, \dots, p_n]_{\sigma}.
    \]
\end{notation}

The coordinates described above are useful for computations with the summands $\rF(D_{\sigma})$ of the \v{C}ech complex. 
We also require a way to work with the summands of the extended \v{C}ech complex. 
To that end, consider a linearly ordered subset $I = \{p_0 \preceq \cdots \preceq p_n\}$ of $\rP_0$, possibly with repetitions. 
We may equivalently regard $I$ as a poset map $[n] \to \rP_0$, where $[n]$ denotes the linearly ordered set $\{0, \dots, n\}$.
We write $I^\circ$ for the image of $I$, which is a linearly ordered subset $I^\circ \subset \rP_0$, and we set $D_I = D_{I^\circ}$.
The set $I^\circ$ is called the \emph{nondegenerate reduction} of $I$.

The point of introducing these notions is that if $\cech^+_n(X, D; \rF)$ is the degree $n$ term of the extended \v{C}ech complex, then 
\[
    \cech^+_n(X, D; \rF) = \bigoplus_{I\col[n] \to \rP_0} \rF(D_I), \quad \rF(D_I) = \bigoplus_{\sigma^\circ} \rF(D_{\sigma^\circ}),
\]
as $D_{\sigma^\circ}$ runs over the connected components of $D_{I^\circ}$.
The following definition is convenient:

\begin{definition}
\label{def:extended-strata}
    In Situation~\ref{sit:ordered-simplicial}, we define the set of \emph{extended strata} $\rP_n^+$ as the set of pairs $\sigma = (\sigma^\circ, I \col [n] \twoheadrightarrow \vertex(\sigma^\circ))$ such that $D_{\sigma^\circ}$ is a connected component of $D_{I^\circ}$.
\end{definition}

The extended strata index the summands of the extended \v{C}ech complex:
\[
    \cech_n^+(X, D; \rF) = \bigoplus_{\sigma \in \rP_n^+} \rF(D_{\sigma}) ,
\]
where we have set $\rF(D_{\sigma}) := \rF(\sigma^\circ)$.
The stratum $\sigma^\circ$ is called the \emph{nondegenerate reduction} of the extended stratum $\sigma$.

\begin{notation}[Coordinates for extended strata]
    An extended stratum $\tau = (\tau^\circ, I)$ can be specified by allowing repetitions in the coordinate scheme introduced above (Notation~\ref{notation:coordinates}).
    In other words, for a chain $p_0 \preceq \cdots \preceq p_n$ in $\rP_0$, the expression
    \[
        \tau = [p_0, \dots, p_n]_{\sigma}
    \]
    refers to the extended stratum $\tau = (\tau^\circ, I)$, where $I:[n] \to \rP_0$ is a poset map whose image is contained in $\vertex(\sigma)$.
    The image of $I$ determines a stratum $\tau^\circ \leq \sigma$ such that $I$ surjects onto $\vertex(\tau^\circ)$.
\end{notation}

An extended stratum $\sigma = (\sigma^\circ, I)$ is \emph{degenerate} if $I \col [n] \to \rP_0$ is not injective.
We write $\cech^{deg}_n(X, D; \rF) \subset \cech_n^+(X, D; \rF)$ for the subobject spanned by $\rF(D_{\sigma})$, as $\sigma$ ranges over the degenerate extended strata. 

\begin{lemma}
\label{lem:e-equals-c-plus-d} 
Let $(X, D)$ be as in Situation~\ref{sit:ordered-simplicial}.
\begin{enumerate} 
    \item The subobjects $\cech^{deg}_n(X, D; \rF) \subset \cech^+_n(X, D; \rF)$, $n \geq 0$, form a subcomplex $\cech^{deg}(X, D; \rF)$ of $\cech^+(X, D; \rF)$.
    \item There is an isomorphism of chain complexes
    \[
        \cech^+(X, D; \rF) = \cech(X, D; \rF) \oplus \cech^{deg}(X, D; \rF).
    \]
    \item The inclusion $\iota\col\cech(X, D; \rF) \to \cech^+(X, D; \rF)$ is a homotopy equivalence, with homotopy inverse given by the projection $\pi\col\cech^+(X, D; \rF) \to \cech(X, D; \rF)$. 
\end{enumerate}
\end{lemma}

\begin{proof}
    The proof is analogous to the proof of Lemma~\ref{lem:inclusion-of-nondegenerate}.
    For instance, one defines degeneracy maps $s_i$ by taking $[p_0, \dots, p_n]_{\sigma}$ to $[p_0, \dots, p_i, p_i, \dots, p_n]_{\sigma}.$
\end{proof}

Now we relate the extended \v{C}ech complex with $\Sd^+(X, D; \rF)$. 
To do so, we first explain the relation between $\Sd^+_n(\rP)$ and $\rP^+$.
There is a map 
\[
    \max : \rP_n \to \rP_0,
\]
given by sending $\sigma \in \rP$ to the maximal element of $\vertex(\sigma)$ with respect to the chosen ordering on $\rP_0$. 
It is clear that $\max$ is a poset map,
and so for each $n \geq 0$, one has the \emph{last vertex map}
\[
    \lambda^+:\Sd_n^+(\rP) \to \rP_n^+, \quad (\sigma_0 \leq \cdots \leq \sigma_n) \mapsto [\max(\sigma_0), \dots, \max(\sigma_n)]_{\sigma_n}.
\]
The last vertex map upgrades to a map of complexes:

\begin{definition}
\label{def:last-vertex-map}
    In Situation~\ref{sit:ordered-simplicial},
    the \emph{last vertex map}
    \[
        \lambda^+:\Sd^+(X, D; \rF) \to \cech^+(X, D; \rF)
    \]
    is given on each summand $\rF(\bsigma)$ by the natural map $\rF(\bsigma) = \rF(\sigma_n) \to \rF(\lambda(\bsigma))$, which comes from applying $\rF$ to the relation $\lambda(\bsigma) \leq \sigma_n$.

    Similarly, we define
    \[
        \lambda:\Sd(X, D; \rF) \to \cech(X, D; \rF)
    \]
    for the composition of the inclusion $\Sd \to \Sd^+$, $\lambda^+ \col \Sd^+ \to \cech^+$, and the projection $\cech^+ \to \cech$.
\end{definition}

\begin{definition}
\label{def:subdivision-map}
    Let $\fS_{n + 1}$ be the symmetric group on $n + 1$ letters, which acts on $[n]$ via permutation. 
    For each $g \in \fS_{n + 1}$ and each $\sigma \in \rP_n^+$ with associated $n$-chain $p_0 \preceq \cdots \preceq p_n$ in $\rP_0$,
    we define $\sd(\sigma, g)$ as the $n$-chain
    \[
        \left( [p_{g(0)}]_{\sigma}^\circ \leq [p_{g(0)}, p_{g(1)}]_{\sigma}^{\circ} \leq \cdots \leq [p_{g(0)}, \dots, p_{g(n)}]^{\circ}_{\sigma} \right),
    \]
    where we write $(-)^{\circ}$ for the nondegenerate reduction.
    Note that the last term of the chain $\sd(\sigma, g)$ is simply $\sigma^\circ$, so $\rF(\sd(\sigma, g)) = \rF(\sigma^\circ)$ by definition.

    The \emph{subdivision map}
    \[
        \sd^+:\cech^+(X, D; \rF) \to \Sd^+(X, D; \rF)
    \]
    is given on $\rF(\sigma)$ by
    \[
        \sum_{g \in \fS_{n + 1}} (-1)^g \cdot \left[\rF(\sigma) \to \rF(\sd(\sigma, g))\right],
    \]
    where $\rF(\sigma) \to \rF(\sd(\sigma, g)) = \rF(\sigma^\circ)$ is the isomorphism.
    It is straightforward to check that $\sd$ is indeed a map of complexes.

    We write
    \[
        \sd:\cech(X, D; \rF) \to \Sd(X, D; \rF)
    \]
    for the composition of the inclusion $\cech \to \cech^+$, $\sd^+ \col \cech^+ \to \Sd^+$, and the projection $\Sd^+ \to \Sd$.
\end{definition}

\begin{theorem}[\v{C}ech comparison]
\label{thm:cech-comparison}
    Let $(X, D)$ be an ordered, vertical simplicial toroidal embedding over $R$ or $k$.
    Then 
    \[
        \lambda : \Sd(X, D; \rF) \to \cech(X, D; \rF)
    \]
    is a homotopy equivalence, with homotopy inverse given by $\sd$.
\end{theorem}

\subsection{Proof of the comparison}

We begin by showing that 
\[
    \lambda \circ \sigma
\]
is equal to the identity. This follows from the observation that if $\sigma \in \Sd^+_n(\rP)$ is nondegenerate, then $\lambda^+(\sd(\sigma, g))$ is degenerate unless $g = 1 \in \fS_{n + 1}$.

It remains to show that $\sd \circ \lambda$ is homotopic to the identity. By Lemma~\ref{lem:inclusion-of-nondegenerate} and Lemma~\ref{lem:e-equals-c-plus-d}, we reduce to showing that $\sd^+ \circ \lambda^+$ is homotopic to the identity.

\begin{notation}
    Given two $n$-chains $\bsigma = (\sigma_0 \leq \dots)$ and $\bm{\tau} = (\tau_0 \leq \dots)$ such that $\sigma_i \leq \tau_i$, we write 
    \[
        \bsigma \leq_i \bm{\tau} = (\sigma_0 \leq \cdots \leq \sigma_i \leq \tau_i \leq \cdots \leq \tau_n),
    \]
    which is an $(n + 1)$-chain.
\end{notation}

Let $\bsigma = \left(\sigma_0 \leq \cdots \leq \sigma_n\right)$ be an $n$-chain. Let $\sigma = \lambda^+(\bsigma)$. In what follows, given $0 \leq i \leq n$ and $g \in \fS_{i + 1}$, we regard $g$ as an element of $\fS_{n + 1}$ which fixes $i + 1, \dots, n$, so that the notation $\sd(\sigma, g)$ is defined. 

\begin{definition}
\label{def:sd_lambda_homotopy}
    For each $n \geq 0$, we define a morphism 
    \[
        h:\Sd_{n}(X, D; \rF) \to \Sd_{n + 1}(X, D; \rF)
    \]
    as follows. 
    Given $\bsigma = (\sigma_0 \leq \cdots \leq \sigma_n) \in \Sd_{n}(\rP)$, and we write $\sigma = \lambda^+(\bsigma)$ as in the previous paragraph.
    Then $h$ is given on $\rF(\bsigma)$ by the map
    \[
        h|_{\rF(\bsigma)} = \sum_{i = 0}^n \sum_{g \in \fS_{i + 1}} (-1)^i (-1)^g \left[ \rF(\bsigma) \to \rF(\sd(\sigma, g) \leq_i \bsigma) \right].
    \]
    Here, both $\rF(\bsigma)$ and $\rF(\sd(\sigma, g) \leq_i \bsigma)$ are, by definition, $\rF(\sigma_n)$, and the map on the right-hand side is the identity map. 
\end{definition}

We show that $h$ is a homotopy between $\sd^+ \circ \lambda^+$ and the identity, i.e., that
\begin{equation} \label{eq:homotopy_relation}
    \partial \circ h + h \circ \partial = \id - \sd^+ \circ \lambda^+.
\end{equation}
It suffices to check \eqref{eq:homotopy_relation} on each summand $\rF(\bsigma)$ of $\Sd^+_{n}(X, D; \rF)$. 

We first compute $\partial \circ h$:
\begin{equation} \label{eq:partial_of_homotopy}
    \partial \circ h|_{\rF(\bsigma)} = \sum_{i = 0}^n \sum_{j = 0}^{n + 1} \sum_{g \in \fS_{i + 1}} (-1)^{i + j} (-1)^g d_j \circ \left[ \rF(\bsigma) \to \rF(\sd(\sigma, g) \leq_i \bsigma) \right]
\end{equation}
Regarding a term of \eqref{eq:partial_of_homotopy} as being specified by a triple $(i, j, g)$, we break up the sum into six parts:
\[
\begin{array}{rlcrl}
    A: & i = 0, j = 0 & & B: & i = n, j = n + 1  \\
    C: & i > j & & D: &  j \geq i + 2 \\
    UD: & 1 \leq j = i + 1 & & LD: & 1 \leq i = j \leq n
\end{array}
\]
Next, we compute $h \circ \partial$:
\begin{equation} \label{eq:homotopy_of_partial}
    h \circ \partial|_{\rF(\bsigma)} = \sum_{i = 0}^{n - 1} \sum_{j = 0}^n \sum_{g \in \fS_{i + 1}} (-1)^{i + j} (-1)^g \left[ \rF(d_j\bsigma) \to \rF(\sd(d_j\sigma, g) \leq_i d_j\bsigma) \right]
\end{equation}
We break \eqref{eq:homotopy_of_partial} into two parts:
\begin{align*}
    LT &: i \geq j \\
    UT &: j \geq i + 1
\end{align*}

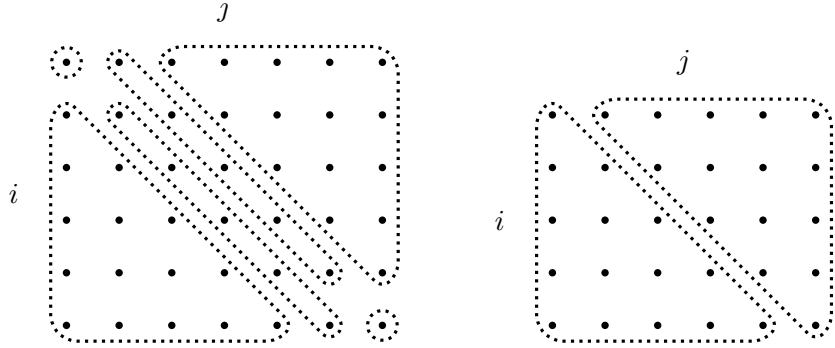
\begin{figure}
\usetikzlibrary{patterns}

	\begin{tikzpicture}[scale=0.7]
  \def\nrows{6}
  \def\ncols{7}
  
  \coordinate (top_left) at (1.7,5.0); 
  \coordinate (top_right) at (2,5.3); 
  \coordinate (bottom_left) at (6,0.7); 
  \coordinate (bottom_right) at (6.3,1); 

  \coordinate (top_left_2) at (1.7,6.0); 
  \coordinate (top_right_2) at (2,6.3); 
  \coordinate (bottom_left_2) at (6,1.7); 
  \coordinate (bottom_right_2) at (6.3,2); 
  
  \draw[very thick, dotted, rounded corners=4pt] (top_left) -- (top_right) -- (bottom_right) -- (bottom_left) -- cycle;

  \draw[very thick, dotted, rounded corners=4pt] (top_left_2) -- (top_right_2) -- (bottom_right_2) -- (bottom_left_2) -- cycle;

  \draw[very thick, dotted] (1,6) circle (8pt); 

  \draw[very thick, dotted] (7,1) circle (8pt); 

  \coordinate (top_left_3) at (0.7,5.0); 
  \coordinate (top_right_3) at (1,5.3); 

  \coordinate (bottom_left_3) at (5,0.7); 
  \coordinate (bottom_right_3) at (5.3,1); 

  \coordinate (corner_left_3) at (0.7,1); 
  \coordinate (corner_right_3) at (1,0.7); 

  \draw[very thick, dotted, rounded corners=4pt] (top_left_3) -- (top_right_3) -- (bottom_right_3) -- (bottom_left_3) -- (corner_right_3) -- (corner_left_3) -- cycle;

   \coordinate (top_left_4) at (2.7,6.0); 
  \coordinate (top_right_4) at (3,6.3); 

  \coordinate (bottom_left_4) at (7,1.7); 
  \coordinate (bottom_right_4) at (7.3,2); 

  \coordinate (corner_left_4) at (7.3,6); 
  \coordinate (corner_right_4) at (7,6.3); 

  \draw[very thick, dotted, rounded corners=4pt] (top_right_4) -- (top_left_4) -- (bottom_left_4) -- (bottom_right_4) -- (corner_left_4) -- (corner_right_4) -- cycle;

   \node at (4,7) {$j$};
  \node at (0,3.5) {$i$};

  \foreach \row in {1,2,3,4,5,6} {
    \foreach \col in {1,2,3,4,5,6,7} {
      \fill (\col,\row) circle (2pt); 
    }
  }

\end{tikzpicture} \hspace{25pt}  \begin{tikzpicture}[scale=0.7] 

\coordinate (top_left_3) at (0.7,5.0); 
  \coordinate (top_right_3) at (1,5.3); 

  \coordinate (bottom_left_3) at (5,0.7); 
  \coordinate (bottom_right_3) at (5.3,1); 

  \coordinate (corner_left_3) at (0.7,1); 
  \coordinate (corner_right_3) at (1,0.7); 

  \draw[very thick, dotted, rounded corners = 4 pt] (top_left_3) -- (top_right_3) -- (bottom_right_3) -- (bottom_left_3) -- (corner_right_3) -- (corner_left_3) -- cycle;

   \coordinate (top_left_4) at (1.7,5.0); 
  \coordinate (top_right_4) at (2,5.3); 

  \coordinate (bottom_left_4) at (6,0.7); 
  \coordinate (bottom_right_4) at (6.3,1); 

  \coordinate (corner_left_4) at (6.3,5); 
  \coordinate (corner_right_4) at (6,5.3); 

  \draw[very thick, dotted, rounded corners = 4 pt] (top_right_4) -- (top_left_4) -- (bottom_left_4) -- (bottom_right_4) -- (corner_left_4) -- (corner_right_4) -- cycle;

  \node at (3.5,6) {$j$};
  \node at (0,3) {$i$};

  \foreach \row in {1,2,3,4,5} {
    \foreach \col in {1,2,3,4,5,6} {
      \fill (\col,\row) circle (2pt); 
    }
  }

\end{tikzpicture}

\caption{\label{figure}
On the left, the decomposition of $\partial \circ h|_{\rF(\bsigma)}$ into $A, B, C, D, UD$, and $LD$. On the right, the decomposition of $h \circ \partial|_{\rF(\bsigma)}$ into $LT$ and $UT$. The case $n = 6$ is shown.}
\end{figure}
The situation is rendered in Figure~\ref{figure}. To verify the homotopy relation \eqref{eq:homotopy_relation}, and hence to complete the proof of Theorem~\ref{thm:cech-comparison}, it suffices to prove the following:

\begin{lemma}
\label{lem:identities}
    With notation as above, the following identities hold:
    \begin{align}
        A &= \id \label{eq:a_relation} \\ 
        B &= - \sd^+ \circ \lambda^+ \label{eq:b_relation} \\
        C &= 0 \label{eq:c_relation} \\
        D + UT &= 0 \label{eq:dut_relation} \\
        UD + LD + LT &= 0 \label{big_relation}
    \end{align}
\end{lemma}

\begin{proof}[Proof]
    Throughout the proof, we will frequently use the following observation: 
    The maps \eqref{eq:partial_of_homotopy} and \eqref{eq:homotopy_of_partial} are sums of face maps $d_j$, 
    since each map 
    \[
        \rF(\bsigma) \to \rF(\sd(\sigma, g) \leq_i \bsigma) 
    \]
    appearing in the definition of $h$ is simply the identity map on the object $\rF(\bsigma_n)$, as explained above. 
    It follows, for example, that in the expression $\partial \circ h + h \circ \partial$, a term $(-1)^r[\rF(\bsigma) \to \rF(\bm{\tau})]$ of the sum \eqref{eq:partial_of_homotopy} cancels with a term $(-1)^s[\rF(\bsigma) \to \rF(\bm{\tau}')]$ of the sum \eqref{eq:homotopy_relation} if and only if $\tau = \tau'$, and $r + s \equiv 0 \mod 2$. 
    (The same observation applies also to the case when both terms are from \eqref{eq:homotopy_of_partial} or from \eqref{eq:partial_of_homotopy}.) 
    We use this without comment in what follows.

    The first two identities \eqref{eq:a_relation} and \eqref{eq:b_relation} are from unwinding definitions. For \eqref{eq:c_relation}, we observe that the map
    \[
        (i, j, g) \mapsto (i, j, g') \quad g' = g \circ (j \;\; j + 1)
    \]
    is an involution of the set of terms of the sum $C$. We claim that the terms of $C$ cancel in the pairs given by the involution. 

    The coefficient of the $(i, j, g)$ term has opposite sign to the coefficient of the $(i, j, g')$ term, so by the discussion at the beginning of the proof, it suffices to show that at the level of $n$-chains, 
    \[
        d_j\left(\sd(\sigma, g) \leq_i \bsigma\right) = d_j\left(\sd(\sigma, g') \leq_i \bsigma \right),
    \]
    which we now check. Writing $\vertex(\sigma) = \{p_0 \preceq \cdots \preceq p_n\}$, 
    \begin{align*}
        d_j(\sd(\sigma, g) \leq_i \bsigma) &= 
        d_j( \cdots \underbrace{\leq [p_{g(0)}, \dots, p_{g(j)}]_{\sigma}^{\circ}}_{j} \leq [p_{g(0)}, \dots, p_{g(j + 1)}]^{\circ}_{\sigma} \leq \cdots \leq \bsigma_i \leq \cdots ) \\
            &= d_j( \cdots \leq \underbrace{[p_{g(0)}, \dots, p_{g(j - 1)}, p_{g(j + 1)}]^{\circ}_{\sigma}}_{j} \leq [p_{g(0)}, \dots, p_{g(j + 1)}]^{\circ}_{\sigma} \leq \cdots \leq \bsigma_i \leq \cdots ) \\
            &= d_j(\sd(\sigma, g') \leq_i \bsigma).
    \end{align*}
    Here, the subtext indicates the position in the chain.

    Next, we prove \eqref{eq:dut_relation}. There is a bijection from the terms of $D$ to the terms of $UT$ given by 
    \[
        (i, j, g) \mapsto (i, j - 1, g).
    \]
    We claim that each term of $D$ cancels with the term of $UT$ given by its image under the bijection. As above, one reduces to showing that, on the level of $n$-chains, 
    \[
        d_j(\sd(\sigma, g) \leq_i \bsigma) = \sd(d_{j - 1}\sigma, g) \leq_i d_{j - 1}(\bsigma).
    \]
    This is done by direct computation:
    \begin{align*}
        d_j(\sd(\sigma, g) \leq_i \bsigma) &= d_j(\cdots \leq [p_{g(0)}, \dots, p_{g(i)}]_{\sigma}^{\circ} \leq \bsigma_i \leq \cdots \leq \bsigma_{j - 1} \leq \cdots) \\
            &= (\cdots \leq [p_{g(0)}, \dots, p_{g(i)}]_{\sigma}^{\circ} \leq \bsigma_i \leq \cdots \leq \bsigma_{j - 2} \leq \bsigma_{j} \leq \cdots) \\
            &= (\cdots \leq [p_{g(0)}, \dots, p_{g(i)}]_{d_{j - 1}\sigma}^{\circ} \leq (d_{j - 1}\bsigma)_i \leq \cdots \leq (d_{j - 1}\bsigma)_{j - 2} \leq (d_{j - 1}\bsigma)_{j - 1}  \leq \cdots) \\
            &= \sd(d_{j - 1}\sigma, g) \leq_i d_{j - 1}(\bsigma).
    \end{align*}

    Finally, we turn to \eqref{big_relation}. We begin by breaking $LD$ into two parts: 
    \begin{align*}
        LD_I &= \{(i, i, g) \mid \quad g(i) = i\}, \\
        LD_{II} &= \{(i, i, g) \mid \quad g(i) \neq i\}.
    \end{align*}
    We show that $LD_I + UD = 0$ and $LD_{II} + LT = 0$. 

    First, let us show that $LD_I + UD = 0$. There is a bijection from the set of terms of $LD_I$ to the set of terms of $UD$ given by 
    \[
        (i, i, g) \mapsto (i - 1, i, g'),
    \]
    where $g' \in \Sigma_{i}$ is determined by observing that, by definition of $LD_I$, $g$ lies in the subgroup $\Sigma_i \subset \Sigma_{i + 1}$ of elements fixing $i$. 

    The fact that the terms cancel follows from the observation that their coefficients have opposite sign, along with the calculation
    \begin{align*}
        d_i\left(\sd(\sigma, g) \leq_i \bsigma\right) &= d_i(\cdots \leq [p_{g(0)}, \dots, p_{g(i - 1)}]_{\sigma}^{\circ} \leq \underbrace{[p_{g(0)}, \dots, p_{g(i)}]_{\sigma}^{\circ}}_{i} \leq  \bsigma_i \leq \cdots) \\
            &= d_i(\cdots \leq [p_{g(0)}, \dots, p_{g(i - 1)}]_{\sigma}^{\circ} \leq \underbrace{\bsigma_{i - 1}}_{i} \leq \bsigma_i \leq \cdots) \\
            &= d_{i}\left(\sd(\sigma, g') \leq_{i - 1} \bsigma\right).
    \end{align*}
    Note that the condition that $g(i) = i$ guarantees that $\{p_{g(0)}, \dots, p_{g(i - 1)}\}$ is a subset of $\vertex(\bsigma_{i - 1})$, so that the stated inequality between the two holds.

    Second, we check that $LD_{II} + LT = 0$. We consider a bijection from the terms of $LD_{II}$ to the terms of $LT$ given by 
    \[
        (i, i, g) \mapsto \left(i - 1, g(i), g'\right), 
    \]
    where $g'$ is given as follows: First, consider the composition 
    \[
        g'' =  \left( i \;\; \cdots \;\; g(i) \right) \circ g
    \]
    of $g$ with the cycle $\left( i \;\; \cdots \;\; g(i) \right)$. Then $g''$ fixes $i$, and the restriction of $g''$ to an element of $\fS_{i}$ is the permutation $g'$. Since
    \[
        (-1)^{g'} = (-1)^g \cdot (-1)^{i - g(i)},
    \]
    the coefficient of the $(i - 1, g(i), g')$ term of $LT$, which is $(-1)^{i - 1 +  g(i)} (-1)^{g'}$, has the opposite sign as the coefficient of the $(i, i, g)$ term of $LD_{II}$, which is $(-1)^{2i}(-1)^g = (-1)^g$.

    As in the previous cases, it remains to check equality of $n$-chains. Writing $\sigma' = d_{g(i)} \sigma$ for simplicity, we see that for any $0 \leq k \leq i - 1$,
    \[
        [p_{g(0)}, \dots, p_{g(k)}]_{\sigma} = [p_{g'(0)}, \dots, p_{g'(k)}]_{\sigma'} \in \rP^+_{i - 1}.
    \]
    To show equality of $n$-chains, we calculate:
    \begin{align*}
        \sd(\sigma', g') \leq_{i - 1} d_{g(i)}(\bsigma) &= ([p_{g'(0)}]^{\circ}_{\sigma'} \leq \cdots \leq [p_{g'(0)}, \dots, p_{g'(i - 1)}]^{\circ}_{\sigma'}  \leq \bsigma_{i} \leq \cdots ) \\
        &= ([p_{g(0)}]^{\circ}_{\sigma} \leq \cdots \leq [p_{g(0)}, \dots, p_{g(i - 1)}]^{\circ}_{\sigma}  \leq \bsigma_{i} \leq \cdots) \\ 
        &= d_i([p_{g(0)}]^{\circ}_{\sigma} \leq \cdots \leq [p_{g(0)}, \dots, p_{g(i)}]^{\circ}_{\sigma}  \leq \bsigma_{i} \leq \cdots).
    \end{align*}
    This concludes the proof of Lemma~\ref{lem:identities}, and also of Theorem~\ref{thm:cech-comparison}.
\end{proof}

\subsection{Barycentric subdivision}
\label{ssec:barycentric-sub}

\begin{proposition}
\label{prop:barycentric-subdivision}
    Let $(X, D)$ be a vertical toroidal embedding over $R$ or $k$, and consider the barycentric subdivision $f\col(X', D') \to (X, D)$ defined in \S\ref{sec:toroidal-embeddings}.
    Then 
    \[
        \Sd(f_*) : \Sd(X', D'; \rF) \to \Sd(X, D; \rF)
    \]
    is a homotopy equivalence.
\end{proposition}

\begin{proof}

From Lemma~\ref{lem:barycentric-is-simplicial}, $(X', D')$ is simplicial. 
For $\rP' = \rP(X', D')$ and $\rP = \rP(X, D)$, Lemma~\ref{lem:barycentric-subdivision} gives a bijection
\[
    \varphi:\rP' \to \Sd(\rP),
\]
commuting with the degeneracy maps on each side. 
One feature of $\varphi$ is that it sends $\sigma' \in \rP'_n$ to a chain $\sigma_0 \leq \dots \leq \sigma_n$, where $\sigma_n = f_* \sigma'$.
In particular, there is a natural map
\[
     \rF(\sigma') \to \rF(f_* \sigma') = \rF(\varphi_*\sigma'),
 \] 
 which is an isomorphism: Since $f$ is a toroidal modification, $D'_{\sigma'_n}$ is birationally a projective bundle over $D_{\varphi_* f'_n}$ by Lemma~\ref{lem:fibers-of-toroidal-resolution}.

There is an isomorphism of complexes
    \[
        \Phi : \cech(X', D'; \rF) \to \Sd(X, D; \rF), 
    \]
given on the summand $\rF(\sigma')$ by the isomorphism $\rF(\sigma') \to \rF(\varphi(\sigma))$ described above.

Moreover, the following diagram commutes:
\[
\begin{tikzcd}
    \Sd(X, D; \rF) \ar[d, "\lambda"'] \ar[dr, "\Sd(f_*)"] \\ 
    \cech(X, D; \rF) \ar[r, "\Phi"', "\simeq"] & \Sd(X^0, D^0; \rF)
\end{tikzcd}
\]
By Theorem~\ref{thm:cech-comparison}, $\lambda$ is a homotopy equivalence, so $\Sd(f_*)$ is a homotopy equivalence.
\end{proof}

\section{Star subdivision}
\label{sec:star-subdivision}

The goal of this section is the prove the following general result:

\begin{theorem}
\label{thm:star-subdivision-at-a-stratum}
    Let $\mu\col(X', D') \to (X, D)$ be a star subdivision of a vertical simplicial toroidal embedding over $R$ or $k$. 
    Then $\Sd(\mu_*):\Sd(X', D'; \rF) \to \Sd(X, D; \rF)$ is a homotopy equivalence.
\end{theorem}

The statement refers implicitly to the star subdivision at a chosen primitive integral vector $v$ in the interior of the cone $|\sigma_0|$ in the cone complex of $(X, D)$, where $\sigma_0$ is a stratum.
When $(X, D)$ is an snc pair, the standard stellar subdivision at a stratum $\sigma \in \rP(X, D)$ is simply the blowup at $D_{\sigma}$. 
Hence, we obtain:

\begin{corollary}	
\label{cor:blowup-of-snc-in-stratum}
    Let $\mu\col(X', D') \to (X, D)$ be the blowup of a vertical snc pair over $R$ or $k$ in a stratum. 
    Then $\Sd(\mu_*):\Sd(X', D'; \rF) \to \Sd(X, D; \rF)$ is a homotopy equivalence.
\end{corollary}

For our final application, we observe that any toroidal embedding has a toroidal resolution, which is a combination of the barycentric subdivision and a sequence of star subdivisions. 
Hence:

\begin{corollary}    
\label{cor:toroidal-resolution}
    Let $(X, D)$ be a vertical toroidal embedding over $R$ or $k$. 
    Then there is a toroidal resolution
    \[
        f\col(X', D') \to (X, D)
    \]
    such that $\Sd(f_*):\Sd(X', D'; \rF) \to \Sd(X, D; \rF)$ is a homotopy equivalence.
\end{corollary}

\begin{proof}
    By Proposition~\ref{prop:barycentric-subdivision}, we may replace $(X, D)$ with its barycentric subdivision; in particular, we may suppose that $(X, D)$ is simplicial.
    Any simplicial toroidal embedding has a resolution by star subdivisions.
\end{proof}

\subsection{A morphism between \v{C}ech complexes}
Throughout, we write $\sigma_0 \in \rP(X, D)$ for the stratum whose associated cone, $|\sigma_0|$, contains the primitive vector where we are starring.

\begin{convention}[Ordering of divisors]
	We choose an ordering 
\[
	D = D_0 + \cdots + D_a + D_{a + 1} + \cdots + D_b
\]
of the irreducible components of $D$ so that $\{D_{a + 1}, \dots, D_{b}\}$ are the divisors containing $\sigma_0$.
We order the irreducible components of $D'$ as
\[
	D' = D'_0 + \cdots + D'_a + D'_{a + 1} + \cdots + D'_b + D'_{b + 1}, 
\]
where $D'_i$ is the strict transform of $D_i$ for $0 \leq i \leq b$, and $D'_{b + 1}$ is the exceptional divisor, i.e., the divisor corresponding to the ray through $v \in |\sigma_0|$ in the star subdivision.
\end{convention}

\begin{definition}
\label{def:cech-map}
    We define a map between \v{C}ech complexes (with respect to the orderings defined above)
    \[
        \cech(f_*) : \cech(X', D'; \rF) \to \cech(X, D; \rF)
    \]
    as the composition
    \[
        \begin{tikzcd}[column sep=large]
            \cech(X', D'; \rF) \ar[r, "\sd"] & \Sd(X', D'; \rF) \ar[d, "\Sd(\mu_*)"] \\
                & \Sd(X, D; \rF) \ar[r, "\lambda"] & \cech(X, D; \rF).
        \end{tikzcd}
    \]
    By Theorem~\ref{thm:cech-comparison}, $\Sd(f_*)$ is a homotopy equivalence if and only if $\cech(f_*)$ is a homotopy equivalence.
\end{definition}

\begin{notation}[Bar notation]
\label{notation:bar}
    Given $\tau = [i_0, \dots, i_n]_{\tau} \in \rP(X, D)$ in Notation~\ref{notation:coordinates}, we shall use bar notation
	\[
	 	\tau = [i_0, \dots \mid i_{\ell + 1}, \dots, i_n]_{\tau}
	 \] 
	to indicate that $\{D_{i_0}, \dots, D_{i_{\ell}}\} \subset \{D_0, \dots, D_a\}$ and $\{D_{\ell + 1}, \dots, D_{n}\} \subset \{D_{a + 1}, \dots, D_b\}$.
	Of course, it may happen that none of the $D_0,\dots, D_a$ contain $D_{\tau}$, which corresponds to the case $\ell = -1$.

	We also use the same notation to specify faces of $\tau$. 
	For example, 
	\[
		d_j \tau = \begin{cases}
			[i_0, \dots, \hat{i}_j, \dots \mid i_{\ell + 1}, \dots, i_n]_{\tau} & 0 \leq j \leq \ell \\
			[i_0, \dots \mid i_{\ell + 1}, \dots \hat{i}_j, \dots, i_n]_{\tau} & \ell + 1 \leq j \leq n .
		\end{cases}
	\]
\end{notation}

\begin{notation}[Double bar notation]
\label{notation:double-bar}
	Given $\tau' \in \rP(X', D')$, we shall use the double bar notation, which depends on whether $D'_{\tau'}$ is contained in the exceptional divisor or not.
    \begin{enumerate} 
        \item  If $D'_{\tau'}$ is contained in the exceptional divisor $D'_{b + 1}$, we write 
    \[
        \tau' = [i_0, \dots \mid i_{\ell + 1}, \dots, i_{n - 1} \mid b + 1]_{\tau'},  
    \]
    where $\{D'_{i_0}, \dots, D'_{i_{\ell}}\} \subset \{D'_0, \dots, D'_a\}$, $\{D'_{\ell + 1}, \dots, D'_{n - 1}\} \subset \{D'_{a + 1}, \dots, D'_b\}$.
    The image of $\tau'$ under $\mu_*$ is given by
    \[
            \mu_* \tau' = [i_0, \dots, i_n \mid a + 1, \dots, b]_{\mu_* \tau'}.
        \]
        \item If $D'_{\tau'}$ is not contained in the exceptional divisor, we write
    \[ 
        \tau' = [i_0, \dots \mid i_{\ell + 1}, \dots, i_n \mid -]_{\tau'} .
    \]
    where $\{D'_{i_0}, \dots, D'_{i_{\ell}}\} \subset \{D'_0, \dots, D'_a\}$ and $\{D'_{\ell + 1}, \dots, D'_{n}\} \subset \{D'_{a + 1}, \dots, D'_b\}$.
    The image of $\tau'$ under $\mu_*$ is given by 
    \[
            \mu_* \tau' = [i_0, \dots \mid i_{\ell + 1}, \dots, i_n]_{\mu_* \tau'}.
        \]
    \end{enumerate}
\end{notation}

\begin{lemma}
\label{lem:compute-cech-map}
Let $\tau' = [i_0, \dots, i_n]$ be a stratum of $(X', D')$. 
\begin{enumerate} 
	\item \label{item:not-in-exceptional}
    If $D_{\tau'} \not \subset D_{b + 1}$, then the restriction of $\cech(\mu_*)$ to $\rF(\tau')$ is given by the vertical morphism
	\[
		 \rF(\tau') \to \rF(\mu_* \tau').
	\]
	\item \label{item:in-exceptional}
    If $D_{\tau'} \subset D_{b + 1}$ and $D_{\tau'} \not \subset D'_b$, then, writing $\tau' = [i_0, \dots \mid i_{\ell + 1}, \dots i_{n - 1} \mid b + 1]_{\tau'}$, $\cech(\mu_*)$ is given on $\rF(\tau')$ as the natural map
	\[
	  	\rF(\tau') \to \rF([i_0, \dots \mid i_{\ell + 1}, \dots, i_{n - 1}, b]_{\mu_* \tau'}),
	  \]  
	which is the composition of $\rF(\tau') \to \rF(\mu_* \tau')$, with the inclusion from $\mu_* \tau'$ to the stratum 
	$[i_0, \dots \mid i_{\ell + 1}, \dots, i_{n - 1}, b]_{\mu_* \tau'}$ of $(X, D)$.
	\item \label{item:weird-case}
    Otherwise (i.e., if $D_{\tau'}$ is contained in both $D'_b$ and $D'_{b + 1}$), the restriction of $\cech(\mu_*)$ to $\rF(\tau')$ is $0$.
\end{enumerate}
\end{lemma}

\begin{proof}
    This is an explicit computation, but let us indicate the key points in each of the three cases.
    In each case, let $g \in \fS_{n + 1}$; we consider the expression $\mu_*\sd(\tau', g)$ which appears in the definition of $\Sd^+(f_*)$.
    The key observation is that, in Case~\ref{item:not-in-exceptional} and Case~\ref{item:in-exceptional}, $\lambda^+(\mu_* \sd(\tau', g))$ is degenerate unless $g = \id$.
    In Case~\ref{item:weird-case}, $\lambda^+(\mu_* \sd(\tau', g))$ is always degenerate, because $D_{\sigma_0}$ (which is the image of $D'_{b + 1}$) and $D_b$ have the same maximal vertex.
\end{proof}

\subsection{The homotopy inverse}

If $\tau \in \Star(\sigma_0)$, then we may write
\[
	\tau = [i_0, \dots, i_\ell \mid a + 1, \dots, b]_{\tau},
\]
since any divisor containing $D_{\sigma_0}$ also contains $D_{\tau}$. 
The preimage $\mu^{-1}(D_{\tau})$ is a birationally a projective bundle over $D_\tau$, but the following lemma produces a natural set of sections.

\begin{lemma}
\label{lem:section-lemma}
    For $\tau = [i_0, \dots, i_\ell \mid a + 1, \dots, b]_{\tau}$ be an element of $\Star(\sigma_0)$. For each $a + 1 \leq k \leq b$, 
    there is a unique connected component, denoted $\sec(\tau, k)$, of 
    \[
    	[i_0, \dots, i_\ell \mid a + 1, \dots, \hat{k}, \dots, b \mid b + 1]
    \]
    whose image contains $D_{\tau}$.
    The induced map $D'_{\sec(\tau, k)} \to D_{\tau}$ is birational. 
\end{lemma}

\begin{definition}
\label{def:section-stratum}
    We write 
    \[
    	\sec(\tau, k) = [i_0, \dots, i_\ell \mid a + 1, \dots, \hat{k}, \dots, b \mid b + 1]_{\tau} .
    \]
    for the connected component described by Lemma~\ref{lem:section-lemma}.
    The subscript $[\dots]_\tau$ indicates that the connected component is determined by the property that it maps into $D_{\tau}$.
\end{definition}

\begin{proof}[Proof of Lemma~\ref{lem:section-lemma}]
    Let $v \in |\sigma_0|$ be the primitive integral vector at which we are starring. 
    In combinatorial terms, the strata $\sec(\tau, k)$ correspond to the join of $v$ in $|\tau|$ with the elements 
    \[
    	[i_0, \dots, i_\ell \mid a + 1, \dots, \hat{k}, \dots, b]_{\tau}
    \]
    of $\Link(\sigma_0)$.
    The fact that $D'_{\sec{\tau, k}} \to D_{\tau}$ is birational follows from the fact that they have the same dimension, and Lemma~\ref{lem:fibers-of-toroidal-resolution}.
\end{proof}

\begin{definition}
\label{def:section-map}
\begin{enumerate} 
	\item If $\tau \notin \Star(\sigma_0)$, then we write $\tau'$ for the strict transform of $\tau$, and we define
	\[
		\gamma_{\tau}:\rF(\tau) \to \rF(\tau')
	\]
	as the inverse of the isomorphism $\rF(\tau') \to \rF(\tau)$ induced by applying $\rF$ to the birational morphism $D'_{\tau'} \to D_{\tau}$.
	\item If $\tau \in \Star(\sigma_0)$, then for each $a + 1 \leq k \leq b$, we define
    \[
    	\gamma_{\tau,k}:\rF(\tau) \to \rF(\sec(\tau, k))
    \]
    to be the inverse of the isomorphism $\rF(\sec(\tau, k)) \to \rF(\tau)$ induced by applying $\rF$ to the birational morphism $D'_{\sec{\tau, k}} \to D_{\tau}$.
\end{enumerate}
 
\end{definition}

\begin{lemma}
\label{lemma:inverse-map-is-complex-map}
There is a map of chain complexes $\gamma\col \cech(X,D) \ra \cech(X',D')$ determined by the following condition: For a stratum $\tau \in \rP(X, D)$ of codimension $n + 1$, the restriction of $\gamma$ to the summand $\rF(\tau)$ of $\cech(X, D)$ is given by
\[
    \gamma =
        \begin{cases}
            \gamma_{\tau} & \tau \notin \Star(\sigma_0) \\
            \displaystyle{\sum_{k=\ell+1}^n (-1)^{n+k}} \gamma_{\tau,k} & \tau \in \Star(\sigma_0).
        \end{cases}
\]
\end{lemma}

\begin{proof}
    For each stratum $\tau$ in $(X,D)$, we show that $d\circ \gamma = \gamma\circ d$ on $\rF(\tau)$. 
    We first treat the simpler case when $\tau \notin \Star(\sigma_0)$.
    If $\tau = [i_0, \dots \mid i_{\ell + 1}, \dots i_n]_{\tau}$, then the strict transform $\tau'$ is given in double bar notation
    \[
    	\tau' = [i_0, \dots \mid i_{\ell + 1}, \dots, i_n \mid -]_{\tau}.
    \]
    As before, the subscript $[\dots]_{\tau}$ refers to the unique connected component of the given intersection of divisors whose image in $X$ contains $D_\tau$.
    It is clear that taking the strict transform commutes with the face maps $d_j$, and it follows easily that $\gamma$ commutes with $d$.

    Now, suppose that $\tau \in \Star(\sigma_0)$.
    We write
    \[
    	\tau = [i_0, \dots \mid i_{\ell + 1}, \dots, i_n]_{\tau}.
    \]
    By assumption, $\{i_{\ell + 1}, \dots, i_n\} = \{a + 1, \dots, b\}$, but we retain the former notation because it determines the signs in the differential. 
    On $\rF(\tau)$, we have
    \[
        d\circ \gamma = \sum\limits_{k=\ell+1}^n \sum\limits_{j=0}^n (-1)^{n+k+j}d_j\circ \gamma_{\tau,k}.
    \]
    Similarly, on $\rF(\tau)$, we have
    \begin{align*}
        \gamma\circ d|_{F(\tau)} &= \sum\limits_{j=0}^n (-1)^j\gamma\circ d_j \\
        &=\left(\sum\limits_{k=\ell}^{n-1}\sum\limits_{j=0}^\ell  (-1)^{n+k+j-1} \gamma_{d_j\tau,k}\circ d_j \right) + \left(\sum\limits_{j=\ell+1}^n  (-1)^{j} \gamma_{d_j\tau}\circ d_j \right).
    \end{align*}
    The first set of terms in the second line has the faces of $\tau$ which remain in $\Star(\sigma_0)$, while the second set of terms has the faces which are contained in $\Link(\sigma_0)$.

	We refer to Table~\ref{table:gamma-circ-diff} and Table~\ref{table:diff-circ-gamma} for the computation of the terms of $\gamma \circ d$ and $d \circ \gamma$ on $\rF(\tau)$.
	We match terms:
    \begin{itemize}
    	\item For $\ell+1\le p\le n$, there are two maps with codomain
    \[
        [i_0, \dots \mid i_{\ell + 1}, \dots, \hat{i}_p, \dots \mid - ]_{\tau}.
    \]
    	The first is $\gamma_{d_p \tau} \circ d_p$, and the second is $d_n \circ \gamma_{\tau, p}$. Both appear with sign $(-1)^p$.
    	\item For $0\le p \le \ell$ and $\ell+1\le q \le n$ there are two maps with the codomain
    \[
        [i_0, \dots, \hat{i}_p, \dots \mid i_{\ell + 1}, \dots, \hat{i}_q, \dots \mid b + 1]_{\tau} .
    \]
    	The first is $\gamma_{d_p \tau, q - 1} \circ d_p$, and the second is $d_p \circ \gamma_{\tau, q}$. 
    	Both appear with sign $(-1)^{n + p + q}$.
    	\item For $\ell+1\le p< q\le m$, there are two maps with codomain
    \[
        [i_0, \dots \mid i_{\ell + 1}, \dots, \hat{i}_p, \dots, \hat{i}_q, \dots \mid b + 1]_{\tau} .
    \]
    	In this case, both maps come from $d \circ \gamma$, and are $d_p \circ \gamma_{\tau, q}$, and $d_{q - 1} \circ \gamma_{\tau, p}$.
    	They appear with opposite signs.
    \end{itemize}

    \begin{table}[ht]
    \caption{The terms of $\gamma \circ d$ on $\rF(\tau)$}
    \label{table:gamma-circ-diff}
    	$\begin{array}{c|c|c}
                    \text{Map} & \text{Codomain} & \text{Sign}\\
                    \hline
                    \begin{array}{c}
                        \gamma_{d_j\tau,k}\circ d_j \\
                        0\le j\le \ell\le k \le n-1
                    \end{array}
                    & [i_0, \dots, \hat{i}_j, \dots \mid i_{\ell + 1}, \dots, \hat{i}_{k + 1}, \dots \mid b + 1]_{\tau}
                    & (-1)^{n+k+j-1}\\
                    \hline
                    \begin{array}{c}
                        \gamma_{d_j\tau}\circ d_j \\
                        \ell+1 \le j \le n
                    \end{array}
                    & [i_0, \dots \mid i_{\ell + 1}, \dots, \hat{i}_j, \dots \mid -]_{\tau} & (-1)^{j}\\
                \end{array}$
    \end{table}
        \begin{table}[ht]
        \caption{The terms of $d \circ \gamma$ on $\rF(\tau)$}
        \label{table:diff-circ-gamma}
        $\begin{array}{c|c|c}
            \text{Map} & \text{Codomain} & \text{Sign}\\
            \hline
            \begin{array}{c}
                d_j\circ \gamma_{\tau,k}\\
                \ell+1\le k \le n, 0\le j\le \ell
            \end{array} 
            & [i_0, \dots, \hat{i}_j, \dots \mid i_{\ell + 1}, \dots, \hat{i}_{k}, \dots \mid b + 1]_{\tau}
            & (-1)^{n+k+j}\\

            \hline

            \begin{array}{c}
                d_j\circ \gamma_{\tau,k}\\
                \ell+1\le k \le n, \ell+1\le j < k
            \end{array} 
            & [i_0, \dots \mid i_{\ell + 1}, \dots, \hat{i}_j, \dots, \hat{i}_{k}, \dots \mid b + 1]_{\tau}
            & (-1)^{n+k+j}\\

            \hline
            
            \begin{array}{c}
                d_j\circ \gamma_{\tau,k}\\ 
                \ell+1\le k \le n, k\le j \le n-1
            \end{array}
            & [i_0, \dots \mid i_{\ell + 1}, \dots, \hat{i}_{k}, \dots, \hat{i}_{j + 1}, \dots \mid b + 1]_{\tau}
            & (-1)^{n+k+j}\\
            \hline

            \begin{array}{c}
                d_j\circ \gamma_{\tau,k}\\
                \ell+1\le k \le n, j =n
            \end{array}
            & [i_0, \dots \mid i_{\ell + 1}, \dots, \hat{i}_{k}, \dots \mid -]_{\tau}
            & (-1)^{2n+k}
        \end{array}$
        \end{table}

    The terms listed above account for all of the terms of $d \circ \gamma$ and $\gamma \circ d$ on $\rF(\tau)$, and we will be done as long as we can show that the matched maps are equal to each other.
    In other words, we need to show the identities
    \begin{align*}
    	\gamma_{d_p \tau} \circ d_p &= d_n \circ \gamma_{\tau, p} \\ 
    	\gamma_{d_p \tau, q - 1} \circ d_p &= d_p \circ \gamma_{\tau, q} \\
    	d_p \circ \gamma_{\tau, q} &= d_{q - 1} \circ \gamma_{\tau, p}
    \end{align*}
    in the ranges for $p,q$ indicated above.

    To prove the first one, observe that there is a commutative diagram
    \[
    	\begin{tikzcd}
    		D'_{\sec(\tau, p)} \ar[d, "f"] \ar[r, "d_n", hook] & D'_{d_n\sec(\tau, p)} = D'_{d_p{\tau}} \ar[d, "g"] \\
    		D_{\tau} \ar[r, "d_p", hook] & D_{d_p \tau},
    	\end{tikzcd}
    \]
    where (abusing notation) $D'_{d_p \tau}$ is the strict transform of $d_p \tau$, and $f, g$ are the restrictions of $\mu$.
    Then Proposition~\ref{prop:existence} implies that
    \[
    	\rF(d_p) \circ \rF(f) = \rF(g) \circ \rF(d_n).
    \]
    The identity follows by rearranging terms, since $\gamma_{d_p \tau} = \rF(g)^{-1}$ and $\gamma_{\tau, p} = \rF(f)^{-1}$.
    A similar argument handles the other two identities. 
\end{proof}

\subsection{The easy direction of Theorem~\ref{thm:star-subdivision-at-a-stratum}}

As we mentioned above, to prove Theorem~\ref{thm:star-subdivision-at-a-stratum}, we just need to show that $\cech(\mu_*)$ is a homotopy equivalence because of Theorem~\ref{thm:cech-comparison}.
In fact, we will show that $\gamma$ is a homotopy inverse of $\cech(\mu_*)$, and in this section, we prove the elementary result that $\cech(\mu_*) \circ \gamma$ is equal (not just homotopic) to the identity. 
First, we compute the composition:
\begin{itemize}
	\item If $\tau \notin \Star(\sigma_0)$, then $\cech(\mu_*) \circ \gamma$ is the identity. 
	\item If $\tau \in \Star(\sigma_0)$, then 
	\[
		\cech(\mu_*) \circ \gamma |_{\rF(\tau)} = \sum_{k = \ell + 1}^n (-1)^{n + k} \cech(\mu_*) \circ \gamma_{\tau, k}
	\]
\end{itemize}
We observe that $\cech(\mu_*) \circ \gamma_{\tau, k} = 0$ unless $k = n$. Indeed, if $k < n$, then $\gamma(\tau, k)$ belongs to the third case of Lemam~\ref{lem:compute-cech-map}.

When $k = n$, 
the codomain of $\gamma_{\tau, n}$ is 
\[
	\sec(\tau, n) = [i_0, \dots \mid a + 1, \dots, b - 1 \mid b + 1]_{\tau}.
\] 
In particular, $\sec(\tau, n)$ belongs to the second case of Lemma~\ref{lem:compute-cech-map}, so $\cech(\mu_*)$ has the form
\[
	\rF([i_0, \dots \mid a + 1, \dots, b - 1 \mid b + 1]_{\tau}) \to \rF([i_0, \dots \mid a + 1, \dots, b]_{\tau}).
\]
This is nothing other than the natural map
\[
	\rF(\sec(\tau, n)) \to \rF(\tau),
\]
whose inverse is $\gamma_{\tau, n}$ by definition.
This concludes the proof that $\cech(\mu_*) \circ \gamma = \id$. 

\subsection{Constructing the homotopy}

We next turn to the more difficult direction of Theorem~\ref{thm:star-subdivision-at-a-stratum}.

\begin{definition}
\label{def:wedge-product}
    Let $\tau' = [i_0, \dots \mid i_{\ell + 1}, \dots i_{n - 1} \mid b + 1]_{\tau} \in \rP(X', D')$ be a stratum contained in $D'_{b + 1}$, with $\tau = \mu_* \tau'$.
    We assume that:
    \begin{enumerate} 
    	\item
    	\label{item:pos-rel-dim}
    	$D'_{\tau'} \to D_{\tau}$ has positive relative dimension.
    	Concretely, this means that the complement of $\{i_0, \dots, i_{n - 1}\} \subset \{a + 1, \dots, b\}$ has at least two elements.
    	\item 
    	\label{item:not-contained-in-last-div}
    	$D'_{\tau}$ is not contained in $D'_b$, i.e., $i_{n - 1} < b$.
    \end{enumerate}
    Then we define $\tau' \wedge b$ as a connected component of the intersection of $D_{\tau'} \cap D'_b$:
    \[
    	\tau' \wedge b = [i_0, \dots \mid i_{\ell + 1}, \dots, i_{n - 1}, b \mid b + 1]_{\tau}. 
    \]
    As usual, the subscript $[\dots]_{\tau}$ indicates that there is a unique connected component whose image in $X$ contains $D_{\tau}$. 
\end{definition}

\begin{remark}
	The stratum $D'_{\tau' \wedge b}$ is nonempty.
    Indeed, if $v \in |\sigma_0|$ is the primitive integral vector at which we are starring,
	then the cone $|\tau' \wedge b|$ can be described in combinatorial terms as the join of $v$ inside $|\tau|$ with the cone
	\[
		[i_0, \dots, i_\ell \mid i_{\ell + 1}, \dots, i_{n - 1}, b]_{\tau} ,
	\]
	which lies in $\Link(\sigma_0)$.
\end{remark}

Let $\tau'$ satisfy the assumptions of Definition~\ref{def:wedge-product}. 
Then 
\[
	d_n \col \rF(\tau' \wedge b) \to \rF(\tau')
\]
is an isomorphism, since both strata are birationally projective bundles over their common image, $\tau = \mu_* \tau' = \mu_* \tau' \wedge b$.
Define
\[
	\epsilon_{\tau'} = d_n^{-1}:\rF(\tau') \to \rF(\tau' \wedge b)
\]
to be the inverse morphism.
We collect some identities involing $\epsilon_{\tau'}$ for later use.

\begin{lemma}[Identities for $\epsilon$]
\label{lem:epsilon-identities}
    Suppose that $\tau'$ satisfies the assumptions of Definition~\ref{def:wedge-product}.
    \begin{enumerate} 
    	\item \label{item:epsilon-id-one}
        If $\ell + 1 \leq k \leq n$, then $d_k \circ \epsilon_{\tau'}$ is an isomorphism.
    	\item \label{item:epsilon-id-two}
        When $k = n$, $d_n \circ \epsilon_{\tau'} = \id \col \rF(\tau') \to \rF(\tau')$.
    	\item \label{item:epsilon-id-three}
        If $\ell + 1 \leq k \leq n - 1$, then $d_k \tau'$ satisfies the assumptions of Definition~\ref{def:wedge-product}, and
    	\[
    	 	d_k \circ \epsilon_{\tau'} = \epsilon_{d_k \tau'} \circ d_k .
    	 \] 
    \end{enumerate}
\end{lemma}

\begin{proof}
    For Item~\ref{item:epsilon-id-one}, all three strata involved (that is, $\tau', \tau' \wedge b$, and $d_k(\tau' \wedge b)$) are birationally projective bundles over $\tau = \mu_* \tau'$.
    Item~\ref{item:epsilon-id-two} is by definition.
    Item~\ref{item:epsilon-id-three} comes from rearranging 
    \[
        d_{n - 1} \circ d_k = d_k \circ d_n,
    \]
    which is an identity of maps $\rF(\tau' \wedge b) \to \rF(d_k \tau')$.
\end{proof}

\begin{definition}
\label{def:homotopy-for-star}
    For each $n \geq 0$, we define a map
    \[
    	h_n: \cech_n(X', D'; \rF) \to \cech(X, D; \rF)
    \]
    as follows: For a stratum $\tau'$, $h_n$ is given on $\rF(\tau')$ by $(-1)^n \epsilon_{\tau'}$ if $\tau'$ satisfies the assumptions of Definition~\ref{def:wedge-product}, and $0$ otherwise.
\end{definition}

\subsection{Conclusion of the proof}

To conclude Theorem~\ref{thm:star-subdivision-at-a-stratum}, we show that the maps $(h_n)$ give a homotopy between $\gamma \circ \cech(\mu_*)$ and the identity.
In other words, we show that for each $\tau'$, we have
\[
	h_{n - 1} \circ d + d \circ h_n = \id - \gamma \circ \cech(\mu_*)
\]
on $\rF(\tau')$.

The simplest case occurs if $D'_{\tau'} \notin D'_{b + 1}$.
In this case, $\mu_* \tau'$ does not lie in $\Star(\sigma_0)$, so $\gamma \circ \cech(\mu_*) = \id$, and we are left to prove that
\[
	h_{n - 1} \circ d + d \circ h_n = 0 
\]
on $\rF(\tau')$.
Since we have assumed that $D'_{\tau'} \notin D'_{b + 1}$, $h_n$ is identically $0$ on $\rF(\tau')$.
Similarly, $D'_{d_j \tau'} \notin D'_{b + 1}$, so $h_{n - 1}$ vanishes on the codomain of $d|_{\rF(\tau')}$, and we are done.

We now turn to the case when $D'_{\tau'}$ is contained in $D'_{b + 1}$.
With $\tau = \mu_* \tau'$, we may write
\begin{align*}
	\tau' &= [i_0, \dots i_\ell \mid i_{\ell + 1}, \dots, i_{n - 1} \mid b + 1]_{\tau} \\
	\tau &= [i_0, \dots i_{\ell} \mid a + 1, \dots, b]_{\tau}.
\end{align*}
A computation shows that
\begin{equation}
\label{eq:formula-for-gamma-circ-cech}
	\gamma \circ \cech(\mu_*) = \begin{cases}
		d_{n + 1} \circ \epsilon_{\tau'} & D_{\tau'} \notin D'_b, \tau' \neq \sec(\tau, b) \\
		\displaystyle{\sum_{k = \ell + 1}^n} (-1)^{n + k} \gamma_{\tau, i} \circ \rF(D'_{\tau'} \to D_{\tau}) & \tau' = \sec(\tau, b) \\
		0 & \textrm{else, i.e., } D'_{\tau'} \subset D'_b . 
	\end{cases}
\end{equation}
Note that the conditions defining the first case are exactly the assumptions of Definition~\ref{def:wedge-product}.
We will analyze each of the three cases in turn. 

\emph{First case:} Here, $\tau' \notin D'_b$, $\tau' \neq \sec(\tau, b)$: 
    \begin{align*}
        h_{n-1}\circ d + d\circ h_n
        &= \left[ \sum\limits_{j=0}^n (-1)^j h_{n-1}\circ d_j \right] 
        + \left[\sum\limits_{j=0}^{n+1} (-1)^j d_j\circ h_{n} \right] \\
        &= \left[ \sum\limits_{j=0}^{n-1} (-1)^{j+n-1} \epsilon_{d_j\tau'} \circ d_j \right]
        + \left[ \sum\limits_{j=0}^{n+1} (-1)^{j+n} d_j\circ \epsilon_{\tau'} \right] \\
        &= \sum\limits_{j=0}^{n-1} (-1)^{j+n} \left(d_j\circ \epsilon_{\tau'} -\epsilon_{d_j \tau'} \circ d_j \right)\\
        &\quad\quad + d_n\circ \epsilon_{\tau'} - d_{n+1}\circ \epsilon_{\tau'}
    \end{align*}
    In the second line, we have used that $h_{n - 1} \circ d_n = 0$, since $d_n \tau'$ no longer satisfies the assumptions of Definition~\ref{def:wedge-product}.
    According to Lemma~\ref{lem:epsilon-identities}, we have the identities
    \begin{align*}
        d_j\circ \epsilon_{\tau'}|_{F(\tau')} &= \epsilon_{d_j \tau'} \circ d_j|_{F(\tau')},\text{ (for }0\le j \le n-1\text{)},\text{ and}\\
        d_n\circ \epsilon_{\tau'}|_{F(\tau')} &= \id,
    \end{align*}
    so we are done.

\emph{Second case:} Here, $\tau' = \sec(\tau,b)$, and
in this case, $h_n$ vanishes on $\rF(\tau')$:
    \begin{align*}
        h_{n-1}\circ d + d\circ h_n|_{F(\tau')} &= \left( \sum\limits_{j=0}^n (-1)^j h_{n-1}\circ d_j \right) + \left( \sum\limits_{j=0}^{n+1} (-1)^j d_j\circ h_{n}  \right)\\
        &= \sum\limits_{j=\ell+1}^{n-1} (-1)^{j+n-1} \epsilon_{d_j\tau'} \circ d_j .
    \end{align*}
    We expand $\id - \gamma \circ \cech(\mu_*)$:
    \begin{align*}
        \id -\gamma\circ \rC(\mu_*) 
        &=  \id - \left(\sum\limits_{k=\ell+1}^{n} (-1)^{n+k}\gamma_{\tau,k}\circ \rF(D'_{\tau'} \to D_{\tau})\right)\\
        &= \id - (-1)^{2n} \gamma_{\tau,n}\circ \rF(D'_{\tau'} \to D_{\tau})
        + \left(\sum\limits_{k=\ell+1}^{n-1} (-1)^{n+k-1}\gamma_{\tau,k}\circ \rF(D'_{\tau'} \to D_{\tau}) \right)\\
        &= \sum\limits_{k=\ell+1}^{n-1} (-1)^{n+k-1}\gamma_{\tau,k}\circ \rF(D'_{\tau'} \to D_{\tau}).
    \end{align*}
    Going from the second line to the third, we have used that $\gamma_{\tau, n}$ is the inverse of $\rF(D'_{\tau'} \to D_{\tau})$ by definition.

    To conclude, we match up terms via $k = j$; we need to show that
    \begin{equation}
    \label{eq:case-two-last-identity}
    	\gamma_{\tau,j}\circ \rF(D'_{\tau'} \to D_\tau) = \epsilon_{d_j\tau'}\circ d_j, \quad \ell + 1 \leq j \leq n - 1
    \end{equation}
    on $\rF(\tau')$.
    This follows from considering the diagram
    \[
    	\begin{tikzcd}
    		\rF(d_j \tau' \wedge b) \ar[dr, "a"'] \ar[r, "d_n"] & \rF(d_j \tau') \ar[d] & \rF(\tau') \ar[dl] \ar[l, "d_j"'] \\
    		& \rF(\tau),
    	\end{tikzcd}
    \]
    where the horizontal arrows come from inclusions between strata, and the vertical arrows come from the map $\mu$. 
    Moreover, the vertical arrows are isomorphisms, since the strata upstairs are generically projective bundles over $\tau$.
    The diagram is commutative by Proposition~\ref{prop:existence}.
    We observe that $d_j \tau' \wedge b = \sec(\tau, j)$, so that $a^{-1} = \epsilon_{d_j \tau'}$. 
    Then \eqref{eq:case-two-last-identity} follows from a diagram chase.

    \emph{Third case:} When $D'_{\tau'} \subset D'_b$. 
    In this case, $h_n$ vanishes on $\rF(\tau')$, and we compute:
    \begin{align*}
        h_{n-1}\circ d|_{F(\tau')}+d\circ h_n &= \left( \sum\limits_{j=0}^n (-1)^j h_{n-1}\circ d_j \right)
        +\left( \sum\limits_{j=0}^{n+1} (-1)^j d_j\circ h_{n} \right)\\
        &= (-1)^{n-1} h_{n-1}\circ d_{n-1} \quad +\quad \left( 0 \right)\\
        &= (-1)^{n-1} (-1)^{n-1} \epsilon_{(d_{n-1} \tau')} \circ d_{n-1} \\
        &= \epsilon_{(d_{n-1} \tau')} \circ d_{n-1} \\
        &= \id
    \end{align*}
    The last line follows from Lemma~\ref{lem:epsilon-identities} applied to $d_{n - 1} \tau'$, which satisfies the assumptions of Definition~\ref{def:wedge-product}.
    This concludes the proof of Theorem~\ref{thm:star-subdivision-at-a-stratum}.
    \qed

\section{Blowing up subvarieties of snc pairs}
\label{sec:blowing-up}

\begin{definition}
\label{def:transverse-subvariety}
    Let $(X, D)$ be an snc pair over $R$ or $k$. 
    A regular subscheme $Z \subset X$ has \emph{simple normal crossings} with $D$ if, for any $x \in X$, there are local coordinates $z_1, \dots, z_k, \dots, z_d$ in $\fm_x \subset \cO_{X, x}$ so that $z_1, \dots, z_k$ are local equations for the components of $D$ containing $x$,
     and $\{z_{i_1}, \dots, z_{i_\ell}\}$ cut out $Z$ for some $\{i_1, \dots, i_\ell\} \subset \{1, \dots, d\}$.
\end{definition}

The effect on $(X, D)$ of blowing up a subvariety $Z$ having simple normal crossings with $D$ depends on whether $Z$ is properly contained in a stratum, is equal to a stratum, or is not contained in a stratum. 

\begin{lemma}
\label{lem:contains-implies-equals}
    Let $(X, D)$ be a vertical snc pair over $R$ or $k$, and let $Z \subset X$ be a connected subvariety which has simple normal crossings with $D$.
    If $Z$ contains a stratum of $(X, D)$, then $Z$ is equal to a stratum.
\end{lemma}

\begin{proof}
    Let $D_{\sigma}$ be a stratum contained in $Z$.
    At a general point $x$ of $D_{\sigma}$, we have
    \[
    	\{i_1, \dots, i_{\ell}\} \subset \{1, \dots, k\}
    \]
    in the notation of Definition~\ref{prop:transverse-subvariety}.
    Hence, in an open neighborhood $U$ of $x$, $Z \cap U = D_{\tau} \cap U$ for some stratum $D_{\tau}$.
    After Zariski closure, $Z = D_{\tau}$.
\end{proof}

Let $(X, D)$ be an snc pair over $R$ or $k$.
If $Z$ has simple normal crossings with $D$, then the blowup $X' = \Bl_Z X$ is an snc model over $R$ or $k$, with reduced special fiber $D'$.
There is an induced morphism of snc pairs
\[
	\mu\col(X', D') \to (X, D),
\]
and the goal of this section is to show that the corresponding map from $\Sd(X', D'; \rF)$ to $\Sd(X, D; \rF)$ is a homotopy equivalence.
We first treat the simplest case.

\begin{lemma}
\label{lem:not-contained-in-stratum}
	Let $Z$ be a connected subvariety of $X$ which has simple normal crossings with $D$, and suppose that $Z$ is not contained in a stratum.
	If $\mu\col (X', D') \to (X, D)$ is the blowup of $Z$, then 
	\[
		\Sd(\mu_*)\col \Sd(X', D', \rF) \to \Sd(X, D; \rF)
	\]
	is an isomorphism of chain complexes.
\end{lemma}

\begin{proof}
   	It follows from Lemma~\ref{lem:contains-implies-equals} that neither $Z$ nor the exceptional divisor (which has simple normal crossings with $D'$) contain any strata.
        Therefore, $\mu_*\col \rP(X', D') \to \rP(X, D)$ is a bijection (with inverse given by strict transform), and  $\mu$ maps each $D'_{\tau'}$ birationally onto $D_{\mu_*\tau}$.
   	These two properties imply that $\Sd(\mu_*)$ is an isomorphism.
\end{proof}

\begin{proposition}
\label{prop:transverse-subvariety}
    Let $\mu\col(X', D') \to (X, D)$ be the blowup of a vertical snc pair over $R$ or $k$ in a regular subscheme $Z \subset X$ which has simple normal crossings with $D$. 
    Then 
    \[
    	\Sd(\mu_*):\Sd(X', D'; \rF) \to \Sd(X, D; \rF)
    \]
    is a homotopy equivalence.
\end{proposition}

\begin{proof}

By treating the connected components of $Z$ separately, we may assume that $Z$ is connected.
After Corollary~\ref{cor:blowup-of-snc-in-stratum} and Lemma~\ref{lem:not-contained-in-stratum}, the only remaining case of Proposition~\ref{prop:transverse-subvariety} is when $Z$ is properly contained in a stratum.
We will restrict ourselves to this case.

First, we introduce some notation. 
Let $\sigma_0$ be the smallest stratum of $(X, D)$ which contains $Z$.
We adopt similar conventions as in the proof of Theorem~\ref{thm:star-subdivision-at-a-stratum}
For instance, we write
\begin{align*}
	D &= D_0 + \cdots + D_a + D_{a + 1} + \cdots + D_b, \\
	D' &= D'_0 + \cdots + D'_a + D'_{a + 1} + \cdots + D'_b + D'_{b + 1} 
\end{align*}
where $D_{a + 1}, \dots, D_b$ are the divisors containing $D_{\sigma_0}$, $D'_i$ is the strict transform of $D_i$, and $D'_{b + 1}$ is the exceptional divisor.

Proceeding as in the proof of Theorem~\ref{thm:star-subdivision-at-a-stratum}, 
it is enough to show that 
\[
	\cech(\mu_*) \col \cech(X, D; \rF) \to \cech(X, D; \rF)
\]
is a homotopy equivalence, where the \v{C}ech complexes are defined with respect to the orderings chosen above, and $\cech(\mu_*)$ just as in Definition~\ref{def:cech-map}.
The computation of $\cech(\mu_*)$ from Lemma~\ref{lem:compute-cech-map} holds verbatim.

We repeatedly use the observation that a stratum $\tau'$ of $(X', D')$ is the strict transform of a stratum in $(X, D)$ (i.e., maps birationally to a stratum in $(X, D)$) if and only if it is not contained in the exceptional divisor $D'_{b + 1}$.

\begin{definition}
\label{def:maps}
    \begin{enumerate} 
    	\item For any stratum $\tau$ in $(X, D)$ with strict transform $\tau'$, we define 
    	\[
    		\gamma_{\tau} \col \rF(\tau) \to \rF(\tau')
    	\]
    	to be the inverse of the isomorphism $\rF(\tau') \to \rF(\tau)$ induced by $\mu$.
    	\item For $\tau'$ such that $D'_{\tau'} \subset D'_{b + 1}$ and $D'_{\tau'} \not \subset D'_b$, we define
    	\[
    		\epsilon_{\tau'}:\rF(\tau') \to \rF(\tau' \wedge b)
    	\]
    	as the inverse of the isomorphism
    	\[
    		d_n \col \rF(\tau' \wedge b) \to \rF(\tau').
    	\]
    	Here, $\tau' \wedge b$ is the unique connected component of $D'_{\tau'} \cap D'_b$ whose image contains $D_{\tau}$,
    	and $d_n$ is an isomorphism since both $D'_{\tau' \wedge b}$ and $D'_{\tau'}$ are projective bundles over $D_{\mu_*\tau}$; cf. the proof of Theorem~\ref{thm:star-subdivision-at-a-stratum}.
    \end{enumerate}
\end{definition}

It is straightforward to check that there is a morphism of complexes
\[
	\gamma \col \cech(X, D; \rF) \to \cech(X', D'; \rF),
\]
given by $\gamma_{\tau}$ on each summand $\rF(\tau)$.
In fact, $\gamma$ is an isomorphism onto the subcomplex of $\cech(X', D'; \rF)$ spanned by the summands $\rF(\tau')$ such that $\tau' \not \subset D'_{b + 1}$. 

It follows from the definitions that $\cech(\mu_*) \circ \gamma$ is equal to the identity, so we show that $\gamma \circ \cech(\mu_*)$ is homotopic to the identity.
Using Lemma~\ref{lem:compute-cech-map} (which holds verbatim in our present situation), we compute $\gamma \circ \cech(\mu_*)$ on $\rF(\tau')$:
\begin{equation}
\label{eq:gamma-circ-cech-subvar}
    \gamma \circ \rC(\mu_*) = 
    \begin{cases}
        \id & D'_{\tau'} \not \subset D'_{b + 1} \\
        d_{n+1} \circ \epsilon_{\tau'} & D'_{\tau'} \subset D'_{b+1} \text{ and } D'_{\tau'} \not\subset D'_{b} \\
        0&\text{else}.
    \end{cases}
\end{equation}
Define a homotopy $h_n$ on $\rF(\tau')$ as follows:
\[
    h_n =
    \begin{cases}
        (-1)^n\epsilon_{\tau'} &D'_{\tau'} \subset D'_{b+1} \text{ and } D'_{\tau'} \not\subset D'_{b} \\
        0 &\text{else}.
    \end{cases}
\]
Our goal is to show the homotopy relation
\begin{equation}
\label{eq:homotopy-rel-subvar}
	d \circ h_n + h_{n - 1} \circ d = \id - \gamma \circ \cech(\mu_*)
\end{equation}
on the summand $\rF(\tau')$.
We divide into the three cases indicated by Equation~\ref{eq:gamma-circ-cech-subvar}.

\emph{First case:} $D'_{\tau'} \not \subset D'_{b + 1}$.
It is clear that $\id = \gamma \circ \cech(\mu_*)$ from the definitions, so we show that the right-hand side of Equation~\eqref{eq:homotopy-rel-subvar} is $0$.
Clearly, $h_n = 0$ on $\rF(\tau')$. Moreover, $h_{n - 1} \circ d = 0$ on $\rF(\tau')$ since the strata $d_j \tau'$ are not contained in $D'_{b + 1}$, so we are done.

\emph{Second case:} Here, $D'_{\tau'} \subset D'_{b+1}$ and $D'_{\tau'} \not\subset D'_{b}$. We compute:
    \begin{align*}
        d\circ h_n + h_{n-1}\circ d 
        &= \left(\sum\limits_{j=0}^{n+1}(-1)^{j+n} d_j \circ \epsilon_{\tau'}\right) 
        + \left( \sum\limits_{j=0}^n (-1)^{j}h_{n-1}\circ d_j\right) \\
        &=\left(\sum\limits_{j=0}^{n+1}(-1)^{j+n} d_j \circ \epsilon_{\tau'}\right) 
        + \left( \sum\limits_{j=0}^{n-1} (-1)^{n+j-1}\epsilon_{d_j\tau'}\circ d_j\right)\\
        &= \left( \sum\limits_{j=0}^{n-1} (-1)^j+n (d_j\circ \epsilon_{\tau'} - \epsilon_{d_j\tau'}\circ d_j) \right) 
        + d_n\circ \epsilon_{\tau'} - d_{n+1}\circ \epsilon_{\tau'}.
    \end{align*}
    We conclude by the following observations:
    \begin{itemize}
        \item $d_j\circ \epsilon_{\tau'} - \epsilon_{d_j\tau'}\circ d_j = 0$,
        \item $d_n\circ \epsilon_{\tau'} = \id$,
        \item $d_{n+1}\circ \epsilon_{\tau'} = \gamma\circ \rC(\mu_*)$,
    \end{itemize}
    which are proved by direct computation (cf. Lemma~\ref{lem:epsilon-identities}).

    \emph{Third case:} $\tau'\subset D'_{b+1}$ and $\tau'\subset D'_b$:
    \begin{align*}
        d\circ h_n + h_{n-1}\circ d
        &= \left(\sum\limits_{j=0}^{n+1}(-1)^{j+n} d_j \circ \epsilon_{\tau'}\right) 
        + \left( \sum\limits_{j=0}^n (-1)^{j}h_{n-1}\circ d_j \right)\\
        &= \left( 0 \right) + \left( \sum\limits_{j=0}^n (-1)^{j}h_{n-1}\circ d_j \right)\\
        &= (-1)^{2n-2}\epsilon_{d_{n-1}\tau'}\circ d_{n-1} \\
        &= \id.
    \end{align*}
    This completes the proof by the observation that $\gamma\circ \rC(\mu_*) = 0$ on $\rF(\tau')$. 
\end{proof}

\section{The categorical motivic volume} 
\label{sec:the-categorical-motivic-volume}

\subsection{The Kahn--Sujatha category}

In this section, we briefly recall the category $\kscat_k$ constructed \cite{kahn-sujatha}.
By definition, $\kscat_k$ is defined to be the localization
\[
    \kscat_k = \SB^{-1}\SmProj_k
\]
in the sense of Gabriel--Zisman \cite{gabriel-zisman}, where $\SB$ is the set of dominant morphisms $X \to Y$ in $\SmProj_k$ inducing a purely transcendental extension of function fields.
As with any Gabriel--Zisman localization, $\kscat_k$ satisfies the property that any functor
\[
    \rF : \SmProj_k \to \cC
\]
sending morphisms in $\SB$ to isomorphisms, extends uniquely to a functor $\rF \col \kscat_k \to \cC$.

As in the introduction, we write $\{X\}_{\re}$ for the image of $X \in \SmProj_k$ in $\kscat_k$.
A special case of the main result of \cite{kahn-sujatha} is the following:

\begin{theorem}[Kahn--Sujatha]
\label{thm:kahn-sujatha}
    For $X, Y \in \SmProj_k$, there is a natural bijection
    \[
        \Mor_{\kscat_k}(\{Y\}_{\re}, \{X\}_{\re}) \simeq Y(k(X))/\rR,
    \]
    where the right-hand side is the set of $\rR$-equivalence classes of rationality maps from Definition~\ref{def:intro-r-equiv}.
\end{theorem}

\begin{remark}
    The universal property of $\kscat_k$ shows that, for a smooth, proper variety $X$ over $k$, the following are equivalent:
    \begin{enumerate} 
        \item $X$ is universally $\rR$-trivial.
        \item For any contravariant stable birational invariant $\rF\col \SmProj_k^\opp \to \Ab$, valued in abelian groups, the natural homomorphism $\rF(X) \to \rF(\Spec k)$ is an isomorphism.
    \end{enumerate}
        This is analogous to a result of Merkurjev \cite{merk}, who proved a similar characterization of universal $\CH_0$-triviality in terms of unramified elements in cycle modules.
        
    To prove that (2) implies (1), the idea is that for a given $T \in \SmProj_k$, one can consider the functor
    \[
        \rF_T(X) = \bZ[T(k(X))/\rR],
    \]
    where the right-hand side is the free abelian group on $\rR$-equivalence classes of rational maps.
    If the natural homomorphism $\rF_T(X) \to \rF_T(\Spec k) \simeq \bZ$ is an isomorphism for all $T$, then $X$ is universally $\rR$-trivial by Yoneda's lemma for $\kscat_k$.
\end{remark}

\subsection{Making categories additive}

Associated to any category $\cC$, there is a natural way to extend $\cC$ to an additive category $\cC \subset \bZ[\cC]$, so that any functor $\cC \to \cA$, where $\cA$ is additive, extends uniquely to a functor $\bZ[\cC] \to \cA$.
The construction is as follows:
\begin{enumerate} 
    \item Let $\cC'$ be the category with the same objects as $\cC$, but where $\Hom_{\cC'}(x, y)$ is the free abelian group on $\Hom_{\cC}(x, y)$, with composition defined by extending the composition in $\cC$ bilinearly. 
    The category $\cC'$ is a preadditive category.
    \item Next, let $\bZ[\cC]$ be the category obtained from $\cC'$ by formally adjoining a zero object $0$, as well as finite coproducts $\bigoplus x_i$. 
    The morphisms from $\bigoplus x_i$ to $\bigoplus y_i$ are given by matrices of morphisms in $\cC'$, and composition is given by matrix multiplication.
\end{enumerate}

\begin{lemma}
\label{lem:k-group}
    Let $\cC$ be an essentially small category, and let $S$ be the set of isomorphism classes of objects in $\cC$. 
    The natural homomorphism
    \begin{equation}
    \label{eq:kzero-split}
        \bZ[S] \to \rK_0(\bZ[\cC])
    \end{equation}
    is an isomorphism, where $\bZ[S]$ is the free abelian group on $S$, and the homomorphism is determined by the map of sets $x \in S \mapsto [x] \in \rK_0(\bZ[\cC])$.
\end{lemma}

\begin{proof}
    The function $\bigoplus x_i \mapsto \sum [x_i] \in \bZ[S]$ is additive for the split exact structure on $\bZ[\cC]$. 
    Hence, it descends to a homomorphism $\rK_0(\bZ[\cC]) \to \bZ[S]$, giving an inverse to \eqref{eq:kzero-split}.
\end{proof}

\begin{lemma}
\label{lem:extending-to-ksadd}
    Let $\rF \col \SmProj_k \to \cA$ be a stable birational invariant valued in an additive category $\cA$.
    Then $\rF$ extends uniquely to an additive functor 
    \[
        \rF : \bZ[\kscat_k] \to \cA.
    \]
\end{lemma}

\begin{proof}
    Combine the universal property of $\kscat_k$ with the universal property of $\bZ[\cC]$.
\end{proof}

\subsection{Constructing the functor}

\begin{theorem}
\label{thm:existence-of-functor}
    Let $\rF\col\SmProj_k \to \cA$ be a stable birational invariant, where $\cA$ is an additive category.
    There exists a functor 
    \[
        \sVol_{\rF}:\SmProj_{\bar K} \to K^b(\cA),
    \]
    satisfying the following properties:
    \begin{enumerate} 
         \item 
         \label{item:main-thm-homotopy-equiv}
         For any semistable model $(X, D)$ of $X_{\bar K}$ over $k \llbracket t^{1/d} \rrbracket$, for any $d > 0$, there is a homotopy equivalence between $\sVol(X_{\bar K})$ and $\Sd(X, D; \rF)$.
         \item 
         \label{item:main-thm-birational}
         $\sVol_{\rF}(-)$ sends birational morphisms to isomorphisms. 
     \end{enumerate} 
\end{theorem}

Before embarking on the proof, we show that Theorem~\ref{thm:existence-of-functor} implies the results stated in the introduction.

\begin{proof}[Proof of Theorem~\ref{thm:intro-categorical-motivic-volume} and Theorem~\ref{thm:intro-arbitrary-additive-category}]
    We prove Theorem~\ref{thm:intro-arbitrary-additive-category}, since Theorem~\ref{thm:intro-categorical-motivic-volume} is the special case when $\rF = \{-\}_{\re}$.
    By the universal properties of $\kscat_k$ and $\bZ[\kscat_k]$, the functor $\sVol_{\rF}$ descends to a functor
    \[
        \sVol_{\rF}: \bZ[\kscat_{\bar K}] \to K^b(\cA).
    \]
    such that if $(X, D)$ is a semistable model of $X_{\bar K}$, then $\sVol_{\rF}(X_{\bar K})$ is homotopic to $\Sd(X, D; \rF)$.
    On the other hand, by Lemma~\ref{lem:inclusion-of-nondegenerate} and \v{C}ech comparison (Theorem~\ref{thm:cech-comparison}), $\Sd(X, D; \rF)$ is homotopic to the \v{C}ech complex $\cech(X, D; \rF)$.
\end{proof}

\begin{proof}[Proof of Corollary~\ref{cor:intro-re-motivic-volume}]
We recall that the Grothendieck group of $\bZ[\kscat_{\bar K}]$ (with the split exact structure) is $\bZ[\RE_{\bar K}]$ by Lemma~\ref{lem:k-group}.
The exact structure of $K^b(\bZ[\kscat_k])$, on the other hand, is given by the exact triangles. 
The functor
\[
    \bZ[\kscat_k] \to K^b(\bZ[\kscat_k])
\]
sending an object into degree $0$ is exact, and induces an isomorphism between Grothendieck groups (see, e.g., \cite{rose}).
Finally, the functor
\[
    \sVol_{\re} : \bZ[\kscat_{\bar K}] \to K^b(\bZ[\kscat_k])
\]
is exact, since it sends direct sums to direct sums.
Then Corollary~\ref{cor:intro-re-motivic-volume} follows from taking the induced homomorphism on Grothendieck groups.
\end{proof}

We return to the task of proving Theorem~\ref{thm:existence-of-functor} along the lines of \cite[Appendix A]{NicShin19} and \cite[Theorem 4]{KonTsch19}.
We first collect the necessary tools (semistable reduction, weak factorization) in the following lemma:

\begin{lemma}
\label{lem:semistable-reduction}
    \begin{enumerate} 
        \item \label{item:semistable-red-snc}
        If $X$ is an snc model over $R_d$, then exists an integer $n_0 > 0$ such that for all $n_0 \mid n$, there is a semistable model $X'$ over $R_{de}$ and a morphism $X' \to X$, inducing the identity on geometric generic fibers.
        \item \label{item:semistable-red-ss}
        If $X$ is a semistable model, then in the context of Item~\ref{item:semistable-red-snc}, we may take $n_0 = 1$, and $X'$ may be chosen so that 
        \[
            \Sd(f_*): \Sd(X', D'; \rF) \to \Sd(X, D; \rF)
        \]
        is a homotopy equivalence.
        \item \label{item:semistable-red-bir}
        If $f\col X' \to X$ is a morphism of snc models of $X_{\bar K}$ over $R$ inducing an isomorphism on generic fibers $X'_K \to X_K$, then $\Sd(f_*)$ is a homotopy equivalence. 
    \end{enumerate}
\end{lemma}

\begin{proof}
    Item~\ref{item:semistable-red-snc} is the Semistable Reduction Theorem from \cite{toroidal-embeddings}, where $n_0$ is the least common multiple of the multiplicites of the components of $X_0$.
    For Item~\ref{item:semistable-red-ss}, the point is that if $X$ is a semistable model over $S_d$, then the base change $X'' = X \times_{S_d} S_{de}$ ($S_d = \Spec R_d$, $S_{de} = \Spec R_{de}$) is a toroidal model.
    The reduction of the special fiber $X''_0$ is isomorphic to $X_0$, so the map
    \[
        \Sd(X'', D''; \rF) \to \Sd(X, D; \rF)
    \]
    is an isomorphism, not merely a homotopy equivalence. 
    To conclude, we take a toroidal resolution $X' \to X''$ as in Corollary~\ref{cor:toroidal-resolution}; the composition $X' \to X$ satisfies the conclusions of Item~\ref{item:semistable-red-ss}.

    For Item~\ref{item:semistable-red-bir}, we apply the Weak Factorization Theorem from \cite{at} to $f\col X' \to X$, which yields a sequence of birational maps between regular schemes
    \[
        \begin{tikzcd}
            X' = V_0 \ar[r, "\phi_{0}", dashed] & V_{1} \ar[r, "\phi_{1}", dashed] & \cdots \ar[r, "\phi_{\ell - 1}", dashed] & V_{\ell} = X, 
        \end{tikzcd}
    \]
    satisfying the following properties:
    \begin{enumerate} 
        \item For each $0 \leq j \leq \ell - 1$, the induced maps $V_j \dra V_\ell$ extend to morphisms $f_j \col V_j \to V_\ell = X$, which are isomorphisms over the complement of the special fiber $D \subset X = V_\ell$.
        Finally, $f_0 = f$.
        Moreover, the preimage of $D$ in $V_\ell$ is an snc divisor, say $D_i \subset V_i$.
        \item For each $0 \leq j \leq \ell - 1$, either $\phi_j$ or $\phi_j^{-1}$ is the blowup of a regular subscheme having simple normal crossings with $D_i$ or $D_{i + 1}$, respectively.
    \end{enumerate}
    We prove by descending induction that $\Sd(f_{j,*})$ is a homotopy equivalence for all $j$. 
    Indeed, there is a diagram (which commutes since it commutes at the generic point of $V_j$):
    \[
        \begin{tikzcd}
            V_{j} \ar[d, dashed, "\phi_j"] \ar[dr, "f_j"] &  \\
            V_{j + 1} \ar[r, "f_{j + 1}"] & V_\ell = X .
        \end{tikzcd}
    \]
    We apply Proposition~\ref{prop:transverse-subvariety} to whichever of $\phi_j$, $\phi_j^{-1}$ is a morphism.
    In the base case $j = \ell - 1$, we take $f_\ell$ to be the identity.
\end{proof}

\begin{proof}[Proof of Theorem~\ref{thm:existence-of-functor}]
    Let $X_{\bar K}$ be a smooth, projective variety over $\bar K$.
    Let $\cSS(X_{\bar K})$ be the category of semistable models of $X_{\bar K}$, i.e., pairs $(X, \varphi)$ where $X/R_d \in \TorMod_{\infty}$ is a semistable model for some $d > 0$, and $\varphi$ is an isomorphism between the geometric generic fiber of $X$ over $R_d$ and $X_{\bar K}$.
    A morphism $(X', \varphi') \to (X, \varphi)$ in $\cSS(X_{\bar K})$ is a morphism in $\TorMod_{\infty}$ compatible with $\varphi$ and $\varphi'$.
    We usually omit $\varphi$ from the notation.

    The category $\cSS(X_{\bar K})$ is cofiltered, i.e., any two $X, X'$ are dominated by a third $X'' \to X$, $X'' \to X'$. 
    Indeed, by Lemma~\ref{lem:semistable-reduction}, we find semistable models $Y \to X$, $Y' \to X'$ such that $Y, Y'$ are both defined over the same base $R_d$.
    After further extension of the base (and applying Lemma~\ref{lem:semistable-reduction} once more), we may suppose that the generic fibers of $Y$ and $Y'$ over $R_d$ are isomorphic.
    Then $X$ and $X'$ are birational, so by taking a log resolution of the closure of the graph of $X \dra X'$, we see that both are dominated by an snc model, say $W$, of $X_{\bar K}$.
    Then we apply Lemma~\ref{lem:semistable-reduction} again to get a semistable model $X'' \to W$ mapping to $X$ and $X'$.

    If $f\col X' \to X$ is a morphism in $\cSS(X_{\bar K})$, then $\Sd(f_*)$ is a homotopy equivalence. 
    Indeed, running the argument of the previous paragraph, we produce a diagram
    \[
        \begin{tikzcd}
            W \ar[r, "\phi"] \ar[dr, "\psi"] & Y' \ar[r, "a"] \ar[d] & X' \ar[d, "f"] \\
            & Y \ar[r, "b"] & X,
        \end{tikzcd}
    \]
    where $\Sd(a_*)$ and $\Sd(b_*)$ are homotopy equivalences by Lemma~\ref{lem:semistable-reduction}, Item~\ref{item:semistable-red-ss}, and $\phi$ and $\psi$ are birational morphisms. 
    By Lemma~\ref{lem:semistable-reduction}, Item~\ref{item:semistable-red-bir}, $\Sd(\phi_*)$ and $\Sd(\psi_*)$ are homotopy equivalences.

    We define the functor $\sVol_{\rF}$ on objects by
    \begin{equation}
    \label{eq:define-svol-inverse}
        \sVol(X_{\bar K}) := \limit \Sd(X, X_0; \rF),
    \end{equation}
    as $X$ runs over $\cSS(X_{\bar K})$.
    We have shown that all of the morphisms in the inverse system \eqref{eq:define-svol-inverse} are isomorphisms, and that the inverse limit is cofiltered. 
    These properties formally imply that the inverse limit exists and is naturally isomorphic to $\Sd(X, X_0; \rF)$ for each $X \in \cSS(X_{\bar K})$.

    If $f_{\bar K} \col X_{\bar K} \to Y_{\bar K}$ is a morphism in $\SmProj_{\bar K}$, then for any semistable model $Y$ of $Y_{\bar K}$, we may find a morphism
    \[
        f\col X \to Y
    \]
    in $\TorMod_{\infty}$ whose geometric generic fiber is isomorphic (as a morphism) to $f_{\bar K}$. 
    If $\cSS(f_{\bar K})$ is the category of such morphisms, then a similar argument as the one given for $\cSS(X_{\bar K})$ shows that $\cSS(f_{\bar K})$ is cofiltered, 
    and that the inverse limit
    \begin{equation}
        \label{eq:define-svol-morphism}
        \sVol(f_{\bar K}) = \limit \Sd(f_*), \quad f \in \cSS(f_{\bar K})
    \end{equation}
    defines a morphism from $\sVol(X_{\bar K}) \to \sVol(Y_{\bar K})$.
    
    We check functoriality: Given $f_{\bar K}\col X_{\bar K} \to Y_{\bar K}$, $g_{\bar K}\col Y_{\bar K} \to Z_{\bar K}$, and $h_{\bar K} = g_{\bar K} \circ f_{\bar K}$, we may find morphisms
    \[
        \begin{tikzcd}
            X \ar[r, "f"] & Y \ar[r, "g"] & Z
        \end{tikzcd}
    \]
    in $\TorMod_{\infty}$ recoving $f_{\bar K}$ and $g_{\bar K}$ on generic fibers.
    Setting $h = f \circ g$, we get two commutative diagrams:
     \[
        \begin{tikzcd}[column sep=large]
            \sVol(X_{\bar K}) \ar[d, "\simeq"] \ar[r, "\sVol(f_{\bar K})"] & \sVol(Y_{\bar K}) \ar[r, "\sVol(g_{\bar K})"] \ar[d, "\simeq"] & \sVol(Z_{\bar K}) \ar[d, "\simeq"] \\
            \Sd(X, \rF) \ar[r, "\Sd(f_*)"] & \Sd(Y, \rF) \ar[r, "\Sd(g_*)"] & \Sd(Z, \rF),
        \end{tikzcd} \quad
        \begin{tikzcd}[column sep=large]
            \sVol(X_{\bar K}) \ar[d, "\simeq"] \ar[r, "\sVol(h_{\bar K})"] & \sVol(Z_{\bar K}) \ar[d, "\simeq"] \\
            \Sd(X, \rF) \ar[r, "\Sd(h_*)"] & \Sd(Z, \rF).
        \end{tikzcd}
    \]
    Since $\Sd(h_*) = \Sd(f_*) \circ \Sd(g_*)$, functoriality for $\sVol(-)$ follows.

    The only part of Theorem~\ref{thm:existence-of-functor} remaining is Item~\ref{item:main-thm-birational}.
    By the Weak Factorization Theorem in $\SmProj_{\bar K}$ from \cite{WeakFactW}, \cite{WeakFactW2}, it is enough to prove that if $Z_{\bar K} \subset X_{\bar K}$ is a smooth, irreducible subvariety, and $f_{\bar K} \col X'_{\bar K} \to X_{\bar K}$ is the blowup along $Z_{\bar K}$, then $\sVol(f_{\bar K})$ is an isomorphism. 
    (The reduction is similar to the proof of Lemma~\ref{lem:semistable-reduction}, Item~\ref{item:semistable-red-bir}.)
    
    Let $X$ be a semistable model of $X_{\bar K}$; after extending the base to $R_d$, $d$ sufficiently divisible, we may assume that $Z_{\bar K}$ descends to a subvariety of the generic fiber of $X$. 
    We claim that we may find a morphism of snc models $Y \to X$ over $R_d$, inducing an isomorphism on generic fibers, such that the scheme-theoretic closure (say, $Z$) of $Z_{K_d}$ in $Y$ is smooth and has simple normal crossings with $Y_0$.
    Indeed, this is an instance of an ``embedded resolution with a simple normal crossings divisor at year $0$,'' applied to the embedding of the scheme-theoretic closure of $Z_{X_d}$ into $X$, with the simple normal crossings divisor $X_0 \subset X$, cf. \cite[\S 1.2.3]{WeakFactAKMW}.
    Let $f\col Y' \to Y$ be the blowup of $Y$ at $Z$.
    Then $f$ belongs to $\cSS(f_{\bar K})$, so $\sVol(f_{\bar K})$ is an isomorphism if and only if $\Sd(f_*)$ is a homotopy equivalence, but Proposition~\ref{prop:transverse-subvariety} implies that $\Sd(f_*)$ is a homotopy equivalence.
\end{proof}

\section{Specialization and variation}
\label{sec:spec-and-var}

In this section, we record two basic results on the behavior of $\rR$-equivalence in smooth, proper families. 
First, in Lemma~\ref{lem:specialization-of-universal-r-triv}, we give an elementary proof that universal $\rR$-triviality specializes.
Second, in Theorem~\ref{thm:constructibility}, we prove a basic result on how $\rR$-equivalence classes vary in a family.

\subsection{An elementary proof of specialization}
\label{ssec:elementary-proof}

Corollary~\ref{cor:intro-re-motivic-volume} immediately implies that universal $\rR$-triviality specializes in smooth, proper families.
In this section, we give an elementary proof, which is based on the argument that a decomposition of the diagonal specializes.

Strictly speaking, we have stated the definition of $\rR$-equivalence (Definition~\ref{def:intro-r-equiv}) only over an algebraically closed field of characteristic $0$, but the definition over an arbitrary field is identical; we invoke the definition over an arbitrary field in the following statement:

\begin{lemma}
\label{lem:universal-r-trivial}
    Let $K$ be an arbitrary field, and let $X/K$ be a smooth, proper variety with function field $K(X)$.
    The following are equivalent:
    \begin{enumerate}
        \item $X$ is universally $\rR$-trivial over $K$.
        \item For any $p \in X(K)$, the constant map 
        \begin{equation*}
            \begin{tikzcd}
                \Spec K(X) \ar[r] & \Spec K \ar[r, "p"] & X
            \end{tikzcd}
        \end{equation*}
        is $\rR$-equivalent to the inclusion of the generic point $\eta_X\col \Spec K(X) \to X$.
        \item For any $p \in X(K)$, the diagonal map $\delta_X = \eta_X \times_K \id_{K(X)} \col \Spec K(X) \to X_{K(X)}$ is $\rR$-equivalent to the constant map $p \times_K \Spec K(X) \col \Spec K(X) \to X_{K(X)}$.
    \end{enumerate}
\end{lemma}

\begin{proof}
    For the equivalence of the first two conditions, see \cite[Theorem 8.5.1]{kahn-sujatha}.
    For (2) implies (3), given $h \col \bP^1_{K(X)} \to X$, there is an induced map $h_{K(X)} \col \bP^1_{K(X)} \to X_{K(X)}$.
    Applying this to a chain of direct $\rR$-equivalences between the generic point $\eta_X$ and a constant map, we get an $\rR$-equivalence between the diagonal and a constant map.
    Similarly, for (3) implies (2), a direct $\rR$-equivalence $h \col \bP^1_{K(X)} \to X_{K(X)}$ can be composed with the projection $X_{K(X)} \to X$, and the diagonal projects to the generic point.
\end{proof}

\begin{lemma}[Specialization of universal $\rR$-triviality]
\label{lem:specialization-of-universal-r-triv}
    Let $k$ be an algebraically closed field, and let $X \to S$ be a smooth, proper morphisms, where $S = \Spec k\llbracket t \rrbracket$.
    If the geometric generic fiber $X_{\bar \eta}$ is universally $\rR$-trivial over $\bar K$, then the special fiber is universally $\rR$-trivial over $k$.

\end{lemma}

\begin{proof}
    First, up to a ramified base change $t \mapsto t^d$, we may assume that on the generic fiber $X_{\eta}$, there is an $\eta$-point $p \in X(\eta)$, and an $\rR$-equivalence between the generic point $\eta_X$ and the constant map $p_{k(X)}$.
    This is arranged by descending an $\rR$-equivalence from $X_{\bar \eta}$.

    Let $R'$ be the stalk of $\cO_X$ at the generic point of $X_0$, and let $S' = \Spec R'$.
    We write $X' = X \times_S S'$.
    Note that $R'$ is a DVR with fraction field $L := k(X)$ and residue field $\ell := k(X_0)$, and we write $X'_L$ and $X'_{\ell'}$ for the generic and special fibers, respectively.
    As explained in \cite{kollar-spec}, the specialization map for rational points
    \begin{equation}
    \label{eq:specialization-for-ratl-pts}
        X'_L(L) \to X'_\ell(\ell)
    \end{equation}
    descends to a specialization map for $\rR$-equivalence classes
    \[
        X'_{L}(L)/\rR \to X'_{\ell}(\ell)/\rR .
     \] 
    Let $\delta \in X'_L(L)$ correspond the diagonal map on $X_{\eta}$, and $\delta_0 \in X'_\ell(\ell)$ correspond to the diagonal map on $X_0$. 
    Similarly, let $p_L \in X'_L(L)$ be a constant map corresponding to a point in $p \in X_\eta(k)$. 
    The point $p$ specializes to $p_0 \in X_0(k)$, and we write $p_\ell$ for the corresponding constant map in $X'_\ell(\ell)$.
    Now, it is easy to see that $\delta$ specializes to $\delta_0$ and $p_L$ specializes to $p_\ell$ under \eqref{eq:specialization-for-ratl-pts}. 
    Hence, $\delta_\ell$ is $\rR$-equivalent to $p_\ell$. 
    By part (3) of Lemma~\ref{lem:universal-r-trivial}, $X_0$ is universally $\rR$-trivial.
\end{proof}

\subsection{Variation}

\begin{theorem}
\label{thm:constructibility}
    Let $k$ be an algebraically closed field of characteristic $0$. 
    Let $X$ and $Y$ be smooth and proper with connected fibers over a finite-type $k$-scheme $S$.
    The locus 
    \[
    	S_{\re} = \{s \in S \mid X_{\bar s} \textrm{ is $\rR$-equivalent to } Y_{\bar s}\},
    \]
    where $\bar s \to s$ is any geometric point, is a countable union of closed sets.
\end{theorem}

We prove Theorem~\ref{thm:constructibility} along the lines of the proof of the analogous result for rationality from \cite{de-fernex-fusi}.

\begin{lemma}
\label{lem:subscheme-vs-r-equiv}
    Let $X$ and $Y$ be smooth, projective, connected varieties over $k$.
    Let $\Gamma \subset \bP^1 \times X \times Y$ be an irreducible closed subscheme whose projection to $\bP^1 \times X$ is birational.
    Then:
    \begin{enumerate} 
    	\item $\Gamma$ contains the graph of a rational map $h \col \bP^1 \times X \dra Y$ as a dense open subset.
    	\item For each $t \in \bP^1$, there is a unique irreducible component $Z_t$ of the scheme-theoretic intersection $\Gamma_t := \Gamma \cap (\{t\} \times X \times Y)$ which dominates $X$ via the projection.
    	Moreover, the projection $Z_t \to X_t$ is birational, and $Z_t$ contains the closure of the graph of a rational map $h_t \col X \dra Y$ as a dense open subset.
    \end{enumerate}
\end{lemma}

\begin{proof}
    For (1), the rational map $h$ is defined by the rational inverse $X \times \bP^1 \dra \Gamma$, followed by the projection to $Y$. 
	For (2), $h$ is defined on an open subset of the form $U \times \bP^1$.
	Then $\Gamma_t \cap U$ is the graph of the composition
	\[
		\begin{tikzcd}
			h_t : U \cap (\{t\} \times X) \ar[r] & U \ar[r, dashed] & Y,
		\end{tikzcd}
	\]
	which is a birational map.
	We take $Z_t$ to be the unique irreducible component of $\Gamma_t$ containing $U \cap \Gamma_t$. 
	
	It only remains to show that $Z_t$ is the unique component of $\Gamma_t$ which dominates $X$. 
	Let $F \subset \Gamma$ be the preimage of the generic point of $\{t\} \times X$ under the projection $\Gamma \to \bP^1 \times X$.
	By Zariski's main theorem, $F$ is connected; on the other hand, for dimension reasons, no two elements of $F$ have a common specialization in $F$.
	Hence, $F$ is a singleton.
\end{proof}

In what follows, we use the notion of a rational map $f\col X \dra Y$ between smooth, proper schemes over a base scheme $S$. 
By this, we mean a subscheme $\bar \Gamma_f \subset X \times_S Y$ (regarded as the closure of the graph of a rational map $f$), with the property that the fibers of $\bar \Gamma_f$ over $S$ are geometrically irreducible, and that for each $s \in S$, $\bar \Gamma_{f, s} \to X_s$ is birational.
Given a morphism $T \to S$, we can form a pullback $f_T \col X_T \dra Y_T$, by pulling back the subscheme $\bar \Gamma_{f}$.

\begin{lemma}
\label{lemma:locus-of-r-equiv}
    Let $X \to S$ and $Y \to S$ be smooth, projective morphisms with geometrically connected fibers to a finite-type $k$-scheme $S$, and let $f, g \col X \dra Y$ be rational maps over $S$.
    The locus 
   	\[
   		\{s \in S \mid f_{\bar s} \sim_{\rR} g_{\bar s}\} \subset S
   	\]
   	for any geometric point $\bar s \to s$, is a countable union of locally closed subsets of $S$.
\end{lemma}

\begin{proof}
    For each $r > 0$, let $H^{\times r}$ be the $r$-fold self-product of the relative Hilbert scheme
    \[
    	H^{\times r} = \prod_{j = 1}^r H, \quad H = \Hilb(\bP^1_S \times_S X \times_S Y / S).
    \]
    We write $\Gamma^j \to H^{\times r}$ for the pullback of the universal subscheme over $H$ along the $j$th projection $\pr_j: H_r \to H$. 
    For each $S$-point $t \in \bP^1_S(S)$, we write
    \[
    	\Gamma_t^j = \pr(\Gamma^j \cap (\{t\} \times_S X \times_S Y \times_S H^{\times r})) \subset X \times_S Y \times_S H^{\times r}
    \]
    where $\pr$ is the projection. 
    For each $h \in H^{\times r}$, we get an $r$-tuple of subschemes
    \[
    	\Gamma_{h}^1, \dots, \Gamma_h^r \subset \bP^1_h \times_h X_h \times_h Y_h ,
    \]
    where the subscript $h$ refers to the base change along the induced map $h \to S$.
    Similarly, given $t \in \bP^1_h(h)$, we get subschemes
    \[
    	\Gamma^j_{h, t} \subset X_h \times_h Y_h
    \]
    for each $j \geq 0$.

	The idea is to define a subset $V_r \subset H^{\times r}$ parameterizing sequences of direct $\rR$-equivalences of length $r$, starting at $f$ and ending at $g$. 
	To that end, we declare that a point $h \in H^{\times r}$ lies in $V_r$ if the following conditions are met:
    \begin{enumerate} 
    	\item \label{item:urh-is-irred}
    	For each $1 \leq j \leq r$, $\Gamma^j_h$ is geometrically irreducible, and its projection to $\bP^1_h \times_h X$ is birational. 
    	\item \label{item:constr-isom-one}
    	The closed subscheme $\Gamma^1_{h, 0} \subset X_h \times_h Y_h$ contains $\bar \Gamma_{f_h}$, the closure of the graph of the pullback $f_h = f \times_S h:X_h \dra Y_h$.
    	\item \label{item:intermediate-step}
    	For $1 \leq j \leq r - 1$, the intersections
    	\[
    		\Gamma_{h, \infty}^j \cap \Gamma_{h, 0}^{j + 1} \subset X_h \times_h Y_h
    	\]
    	surject onto $X_h$ under the projection to the first factor.
    	\item \label{item:constr-isom-two}
    	The closed subscheme $\Gamma^r_{h, \infty} \subset X_h \times_h Y_h$ contains $\bar \Gamma_{g_h}$, the closure of the graph of the pullback $g_h = g \times_S h:X_h \dra Y_h$.
    \end{enumerate}
    We first show that $V_r$ is a countable disjoint union of schemes of finite type over $S$.

    Since $H^{\times r}$ is a countable disjoint union of schemes of finite type over $S$, it is enough to show that Conditions~\ref{item:urh-is-irred}--\ref{item:constr-isom-two} define constructible subsets of the connected components of $H^{\times r}$.
    For Condition~\ref{item:urh-is-irred}, the geometrically irreducible locus is constructible by \cite[\href{https://stacks.math.columbia.edu/tag/055B}{Tag 055B}]{stacks-project}.
    The constructibility of the birational locus is \cite[Lemma 2.4]{de-fernex-fusi} applied to the projection
    \[
    	\Gamma^j \to \bP^1_S \times_S X \times_{S} H^{\times r}, 
    \]
    since the domain is flat and the codomain is projective over each connected component of $H_r$.

    Conditions~\ref{item:constr-isom-one}--\ref{item:constr-isom-two} are closed conditions.
    Indeed, for Condition~\ref{item:constr-isom-one}, 
    let $Z$ be the intersection of $\Gamma_0^1$ and $\bar \Gamma_f \times_S H^{\times r}$ inside $X \times_S Y \times_S H^{\times r}$.
    Then Condition~\ref{item:constr-isom-one} holds exactly on the locus where the fibers of $Z \to H^{\times r}$ have the maximum possible dimension, $\dim X_h$. 
    Since $Z$ is proper over $H^{\times r}$, the locus is closed by upper semicontinuity of fiber dimension. 
    The argument for Condition~\ref{item:constr-isom-two} is identical.

    Similarly, for Condition~\ref{item:intermediate-step},
    let $W$ be the intersection of $\Gamma^{j + 1}_0$ and $\Gamma^j_{\infty}$ in $X \times_S Y \times_S H^{\times r}$.
    The projection $W \to X \times_S H^{\times r}$ is proper, so it is easy to see that the locus in $H^{\times r}$ where the projection is fiberwise surjective is closed. 
    For instance, since the fibers of $X \times_S H^{\times r}$ over $H^{\times r}$ are irreducible, one may apply semicontinuity of fiber dimension to the image of $W$. 

   	To conclude the proof, it suffices to show that $f_{\bar s}$ is $\rR$-equivalent to $g_{\bar s}$ for any geometric point $\bar s \to s$ if and only if $s$ lies in the image of a point $h \in V_r$ for some $r > 0$.
   	Indeed, the image of each $V_r$ in $S$ is a countable union of locally closed sets by Chevallay's theorem, so the same holds for the image of $\bigcup_{r > 0} V_r$.

   	First, let $h_1, \dots, h_r \col X_{\bar s} \times_{\bar s} \bP^1_{\bar s} \dra Y_{\bar s}$ be the rational maps giving an $\rR$-equivalence from $f_{\bar s}$ to $g_{\bar s}$.
   	The closures of their graphs, $\Gamma^1_{\bar s}, \dots, \bar \Gamma^r_{\bar s}$, give a geometric point $\bar s \to h \in H_r$, and we claim that $h$ lies in $V_r$.
   	In fact, Conditions~\ref{item:urh-is-irred}--\ref{item:constr-isom-two} can be checked after base change to $\bar s$, so it is enough to show that $\Gamma^j_{\bar s}$ satisfy the analogous conditions over $\bar s$. 
   	This is easily done using the argument of Lemma~\ref{lem:subscheme-vs-r-equiv}:
   	At each time $t \in \bP^1_{\bar s}$, the unique irreducible component of $\Gamma_{\bar s, t}^j$ dominating $X_{\bar s}$ is the closure of the graph of the rational map $h_{j, t} \col \{t\} \times_{\bar s} X_{\bar s} \dra Y_{\bar s}$.
   	In particular, the condition 
   	\[
   		h_{j, \infty} = h_{j + 1, 0} ,	
   	\]
   	which is part of the definition of direct $\rR$-equivalence, is equivalent to Condition~\ref{item:intermediate-step} that  the intersection $\Gamma^j_{\bar s, \infty} \cap \Gamma^{j + 1}_{\bar s, 0}$ dominates $X_{\bar s}$.
   	The argument for Conditions~\ref{item:constr-isom-one} and \ref{item:constr-isom-two} is similar.

    Conversely, suppose that $s$ lies in the image of $V_r$.
    Since $V_r$ is constructible in each connected component of $H_r$, the connected components of $V_r$ are of finite type over $S$. 
    In particular, there exists a point $h \in V_r$, $h \mapsto s$ whose residue field is finite over $\kappa(s)$; it follows that any geometric point $\bar s \to s$ factors through $h$.
    
    By Condition~\ref{item:urh-is-irred} and Lemma~\ref{lem:subscheme-vs-r-equiv}, the subschemes $\Gamma^j_{\bar s}$ are the graph-closures of rational maps $h_j \col \bP^1_{\bar s} \times_{\bar s} X_{\bar s} \dra Y_{\bar s}$, and for each time $t \in \bP^1_{\bar s}$, we get rational maps $h_{j, t} \col \{t\} \times X_{\bar s} \dra Y_{\bar s}$.
    Condition~\ref{item:constr-isom-one} implies that $h_{1, 0} = f_{\bar s}$. 
    Indeed, by Lemma~\ref{lem:subscheme-vs-r-equiv}, $\bar \Gamma_{h_{1, 0}}$ is the unique component of $\Gamma^1_{\bar s, 0}$ dominating $X_{\bar s}$.
    Since $\bar \Gamma_{f_{\bar s}}$ also dominates $X_{\bar s}$ and is contained in $\Gamma^1_{\bar s, 0}$ by assumption, $h_{1, 0}$ and $f_{\bar s}$ have the same graph.
    Similar arguments show that $h_{j, \infty} = h_{j + 1, 0}$, and $h_{r, 0} = g_{\bar s}$, so $f_{\bar s}$ and $g_{\bar s}$ are $\rR$-equivalent.
\end{proof}

\begin{lemma}
\label{lem:locus-of-compositions}
    Let $X$ and $Y$ be smooth, projective $S$-schemes, where $S$ is a variety over $k$.
    Given a pair of rational maps $f \col X \dra Y$ and $g \col X \dra Y$ over $S$, the locus of $s \in S$ such that $f_{\bar s} \circ g_{\bar s}$ and $g_{\bar s} \circ f_{\bar s}$ are $\rR$-equivalent to the identities for any geometric point $\bar s \to s$, is a countable union of locally closed sets.
\end{lemma}

The expressions $f \circ g$ and $g \circ f$ in the statement of Lemma~\ref{lem:locus-of-compositions} are the compositions of the $\rR$-equivalence classes of $f$ and $g$;
we do not assume that $f$ and $g$ are composable as rational maps, although is arranged in the course of the proof.

\begin{proof}
	We proceed by noetherian induction on $S$.
	Let $\eta$ be a generic point of $S$.
	Choose a birational morphism $\pi_{\eta} \col X'_{\eta} \to X$, where $X'_{\eta}$ is smooth and projective over $\eta$, which resolves the indeterminacy of $f_{\eta}$.
	Then $\pi_{\eta}$ spreads out to a morphism $\pi_U \col X'_U \to X_U$ between smooth, proper schemes over an open subscheme $U \subset S$.
	Moreover, after shrinking $U$, we may assume that $\pi_s \col X'_s \to X_s$ is birational for every $s \in U$, since the property of being birational is constructible \cite[Lemma 2.4]{de-fernex-fusi}. 

	Hence, after replacing $S$ by a dense open subscheme, there is a morphism $\pi \col X' \to X$ such that $X'$ is smooth and proper over $S$ and $\pi_s$ is birational for all $s \in S$, and a morphism $f'\col X' \to Y$ such that $f'_s = \pi_s \circ f_s$ as rational maps for all $s \in S$.
	Then the locus where $g \circ f$ is $\rR$-equivalent to the identity coincides with the locus where $g \circ f'$ is $\rR$-equivalent to $\pi'$, which is a countable union of locally closed subsets by Lemma~\ref{lemma:locus-of-r-equiv}.

	By an identical argument, the locus where $f \circ g$ is $\rR$-equivalent to the identity is a countable union of locally closed subsets. 
	We conclude by taking the intersection of the locus for $g \circ f$ and the locus for $f \circ g$.
\end{proof}

\begin{proof}[Proof of Theorem~\ref{thm:constructibility}]
	Using Corollary~\ref{cor:intro-re-motivic-volume}, one can follow the proof of \cite[Theorem 4.14]{NicShin19} to show that $S_{\re}$ is closed under specialization.
	Hence, it suffices to show that it is countable union of locally closed subsets.
    By noetherian induction and Chow's lemma, we may assume that $X$ and $Y$ are projective over $S$.

    Let $V_1 \subset \Hilb(X \times_S Y / S)$ be the set of points $h$ such that the corresponding subscheme $\Gamma_h \subset X_h \times_h Y_h$ is geometrically irreducible and projects birationally onto $X_h$.
    In other words, $V_1$ parameterizes the graph-closures of rational maps $X \dra Y$ over $S$, and the universal subscheme over $V_1$ gives a rational map
    \[
    	X \times_S V_1 \dra Y \times_S V_1
    \]
    relative to $V_1$.
	Similarly, we define $V_2$, which parameterizes the graph-closures of rational maps $Y \dra X$.
	Both $V_1$ and $V_2$ are countable disjoint unions of finite-type $S$-schemes by the argument of \cite[Proposition 2.3]{de-fernex-fusi}. 

    Let $V = V_1 \times V_2$. 
    There are universal rational maps
    \[
    	f :X \times_S V \dra Y \times_S V, \quad g : Y \times_S V \dra X \times_S V
    \]
    over $V$, obtained by pulling back the universal subschemes from $V_1$ and $V_2$.
    By Lemma~\ref{lemma:locus-of-r-equiv}, the locus $V' \subset V$ of $v \in V$ where $f_{\bar v}$ and $g_{\bar v}$ are inverse up to $\rR$-equivalence for any geometric point $\bar v \to v$, is a countable union of locally closed sets in each connected component of $V$.
	Since $V$ has countably many connected components and each component is finite-type over $S$, the image of $V'$ in $S$ is a countable union of locally closed sets.
	On the other hand, the image is exactly $S_{\re}$.
\end{proof}

	\bibliography{draft-bib}
	\bibliographystyle{amsalpha}

\end{document}